\documentclass[11pt]{article}
\usepackage{amsthm,amssymb,amsmath,amsfonts,amscd}
\usepackage{graphicx,epsfig,latexsym}
\usepackage{cite,calc,xcolor,subfigmat,mathrsfs,dsfont,epstopdf}
\usepackage{hyperref}
\usepackage[lined,commentsnumbered]{algorithm2e}
\usepackage{tabularx}
\usepackage{circuitikz}

\newcommand{\LST}{\mathscr{L}}
\newcommand{\E}{\mathscr{E}}

\newcommand{\0}{\mathbf{0}}

\newcommand{\ie}{{\em i.e.,} }
\newcommand{\eg}{{\em e.g.,} }

\newtheorem{thm}{Theorem}[section]

\newtheorem{lem}[thm]{Lemma}
\newtheorem{prop}[thm]{Proposition}
\theoremstyle{definition}

\newtheorem{rem}[thm]{Remark}

\numberwithin{equation}{section}
\def\Blem {\begin{lem}}
\def\Elem {\end{lem}}
\def\be {\begin{equation}}
\def\ee {\end{equation}}
\def\ba {\begin{eqnarray}}
\def\ea {\end{eqnarray}}
\def\bes {\begin{equation*}}
\def\ees {\end{equation*}}
\def\bas {\begin{eqnarray*}}
\def\eas {\end{eqnarray*}}
\def\bpr {\begin{proof}}
\def\epr {\end{proof}}

\begin{document}
\baselineskip=17pt

\renewcommand {\thefootnote}{ }

\pagestyle{empty}

\begin{center}
\leftline{}
\vspace{-0.500 in}
{\Large \bf Symmetry-breaking singular controller design for
Bogdanov-Takens bifurcations with an application to Chua system} \\ [0.3in]

{\large Majid Gazor\(^{a, b}\)\footnote{$^\dag\,$Corresponding author. Phone: (98-31) 33913634; Fax: (98-31) 33912602;
Email: mgazor@iut.ac.ir; Email: n.sadri@ipm.ir} and Nasrin Sadri\(^{b, \dag}\) }

\vspace{0.105in} {\small {\em \({}^{a}\)Department of Mathematical Sciences, Isfahan University of Technology
\\[-0.5ex]
Isfahan 84156-83111, Iran\\[0.5ex] \(^b\)School of Mathematics, Institute for Research in Fundamental Sciences (IPM), \\[-0.5ex]
P.O. Box: 19395-5746, Tehran, Iran }}



\noindent
\end{center}

\vspace{-0.20in}

\baselineskip=16pt

\:\:\:\:\ \ \rule{5.88in}{0.012in}

\begin{abstract}
We provide a complete symmetry-breaking bifurcation control for \(\mathbb{Z}_2\)-equivariant smooth differential systems with Bogdanov-Takens singularities. Controller coefficient space is partitioned by {\it critical controller sets} into different {\it connected regions}. The connected regions provide a classification for all qualitatively different dynamics of the controlled system. Hence, a state feedback controller design with four small controller coefficients is proposed for an efficient and full singular symmetry-breaking control. We show that our approach works well for nonlinear control systems with both controllable and uncontrollable linearizations. Asymmetric bifurcations are all associated with the controlled system and they start with a primary controlled pitchfork bifurcation from the origin. Origin is a primary equilibrium for the uncontrolled system. This gives rise to two secondary local equilibria \(E_\pm\) for the controlled system. These equilibria further experience tertiary fold and hysteresis type bifurcations. The secondary and primary equilibria experience Hopf and Bautin bifurcations leading to the appearance of one limit cycle \(\mathscr{C}_0\) from primary equilibrium, and either one or two from each secondary equilibria; namely, \(\mathscr{C}^1_\pm\) and \(\mathscr{C}^2_\pm\). The collisions of limit cycles (\(\mathscr{C}_0\) and \(\mathscr{C}^1_\pm\)) with equilibria lead to either a heteroclinic cycle \(\Lambda\) or four different homoclinic cycles (\(\Lambda_\pm\) and \(\Gamma_\pm\)). Each pair of limit cycles \((\mathscr{C}^2_\pm, \mathscr{C}^1_\pm)\) may respectively merge together and disappear. This is a {\it saddle-node bifurcation of limit cycles}. Different combinations of these give rise to a rich list of bifurcation scenarios. Finite determinacy of each of these bifurcations has been thoroughly investigated. Subcritical and supercritical types of bifurcations can be switched using small changes into controller coefficients. This greatly influences their stabilization potential in applications. Our symbolic estimation of critical controller sets provide a computationally feasible approach for bifurcation control of such systems. We derive novel estimates for the heteroclinic, homoclinic and limit cycles to facilitate the amplitude size management and frequency control of the nearby oscillating dynamics. To illustrate our approach, we consider Chua system with a quadratic state-feedback controller. Our approach provides estimated controller sets in terms of the original controller coefficients and constants of the controlled Chua system. Controlled Chua system experiences a pitchfork bifurcation, three Hopf bifurcations and two homoclinic bifurcations. We show that there exist two different regions of controller coefficient choices for feedback regularization and two nearby regions for supercritical Hopf stabilization approach.
\end{abstract}

\vspace{0.05in}

\noindent {\it Keywords:} \ Singular control; Critical controller sets; Uncontrollable linearization; Subcritical and supercritical switching.

\vspace{0.05in} \noindent {\it 2010 Mathematics Subject Classification}:\, Primary: 34H20, 34K18, 34C20; Secondary: 58E25.

\vspace{0.05in}
\baselineskip=13pt

\section{Introduction }
Differential systems with symmetry (equivariant systems) frequently occur in many real life and engineering problems while qualitative changes are the intrinsic elements of their evolutions. A system is called {\it singular} when it experiences a {\it qualitative change} and each qualitative change is called a {\it bifurcation}. Hence, bifurcation control of equivariant singular systems is a natural contribution for the management of their qualitative evolutions. Due to the singularity around any qualitative change, an uncontrolled system may potentially experience varieties of desired and undesired dynamics. We distinguish different bifurcation scenarios for potential controlled dynamics, where they can be realised and/or switched as desired through a state-feedback singular bifurcation controller with small controller coefficients. These bifurcations can be quantitatively controlled, prevented, delayed or accelerated. Hence, the controlled system can experience any desired dynamics chosen from a rich list of bifurcation scenarios through tuning small controller coefficients. We refer to these by {\it bifurcation control problem}. Singularity lays an asset and important potential for engineering applications with high manoeuvring capability. Manoeuvrability here implies frequent quantitative and qualitative dynamics changes with minimal controller costs. Our proposed approach makes a full use of the internal singular dynamics of the uncontrolled system to enforce the desired dynamics. This is an alternative to many existing techniques in nonlinear control theory such as back-stepping method, input-state feedback linearization, Lyapunov functions, etc; \eg see \cite{Slotine}. Most of these techniques are oblivious of the uncontrolled singular dynamics: the designed controllers (fully or partially) eliminate the internal (uncontrolled) dynamics and replace it with an already-known desired and non-singular dynamics. Thus, they mainly fail to exploit the benefits of singularities; this includes the manoeuvring capabilities. The main obstacle originates from the underlying complexity (highly rich dynamics) of singular systems. Bifurcation control stands to efficiently make use of the intrinsic singular dynamics of the uncontrolled system. This justifies to call our proposed controller approach as {\it singular control}. Therefore, bifurcation control leads to an effective approach with low-cost controllers and high manoeuvrability.

\pagestyle{myheadings} \markright{{\footnotesize {\it M. Gazor and N. Sadri \hspace{2.5in} {\it $\mathbb{Z}_2$-symmetry breaking bifurcation control}}}}

Three main claimed contributions in this paper are as follows: (1) Complete symmetry-breaking classification for the highly rich bifurcation scenarios associated with \(\mathbb{Z}_2\)-equivariant Bogdanov-Takens singularity. (2) Novel symbolic estimates of critical controller sets, where they are sufficiently accurate for many of their potential applications. (3) An introduction of a practically feasible approach for singular control of linearly uncontrollable systems with arbitrary state dimension and two zero eigenvalues for its (non-hyperbolic) linearization.

There are a rich list of bifurcation scenarios where they can all be realised through our proposed approach. These are only useful when the controller design is adaptable based on the physics of the problem. This is, of course, one of our main claimed contributions. Furthermore, bifurcations may switch their subcritical type with supercritical types (or vice versa) when small changes are  applied to the controller coefficients; see Theorem \ref{THmBautin} and Remark \ref{SwitchingSub}. These signify the importance for the study of symmetry breaking bifurcations due to their influence for the stabilizing approach in applications. Small modeling imperfections for singular systems lead to bifurcations and thus, they can be a dominant factor for determining the dynamics. For \(\mathbb{Z}_2\)-equivariant systems, bifurcations include the loss of symmetry; this is technically called a {\it symmetry breaking} bifurcation. Hence, equivariant bifurcation control is not sufficient for systems whose noises and imperfections have the potentials for symmetry breaking. Thus, symmetry breaking bifurcation analysis and control for singular systems is necessary in these cases. Uncontrolled smooth differential systems whose linearization at a non-hyperbolic equilibrium has a pair of zero-eigenvalues (none semi-simple and \(\mathbb{Z}_2\)-equivariant mode cases) can be reduced to
\ba\label{Eq01}
&\frac{dx}{dt}= f(x, y),\; \frac{dy}{dt}= a_0x+ g(x, y),\; f(-x, -y)=-f(x, y),\; g(-x, -y)= -g(x, y), &
\\\nonumber
&\hbox{ for all } x, y\in \mathbb{R}, a_0\neq0, \hbox{ and }  f(0, 0)=g(0, 0)=0.&
\ea Functions \(f\) and \(g\) are assumed to be \(C^k\)-smooth for \(k\geq 5\). The state space of the original system of such types can be any arbitrarily large natural number and it can yet be reduced into the differential system \eqref{Eq01}. This is feasible through Jordan canonical transformation and a center manifold reduction; \eg see Section \ref{sec4} where this has been implemented to Chua system with a three dimensional state space. Hence, the bifurcation control problem for such \(\mathbb{Z}_2\)-equivariant systems (using center manifold reductions, also see\cite{HamziKangCenter05}) through polynomial controllers \((u_1, u_2)\) can be reduced to
\begin{small}
\ba\label{Eq02}
&\dot{x}=\frac{dx}{dt}= f(x, y)+u_1(x, y, \nu), \;\dot{y}=\frac{dy}{dt}=a_0 x+ g(x, y)+u_2(x, y, \nu), \;\, u_1(x, y, \0)=u_2(x, y, \0)= 0, &
\ea
\end{small} for \(\nu=(\nu_1, \ldots, \nu_m)\in \mathbb{R}^m.\) We call \(\nu_j\) (for \(j\leq m\)) by 
{\it controller coefficients}. We further assume that
\bas&\left({\frac{\partial^3 f}{\partial x^3} (0, 0)}\right)
\left(\frac{1}{4}\frac{\partial^3 f}{\partial x \partial y^2}(0, 0)+ 3\frac{\partial^3 g}{\partial y^3}(0, 0) \right)\neq 0.&
\eas Then, there are always locally invertible polynomial changes of state variables to transform system \eqref{Eq02} into (see Remark \ref{Rem2.1})
\ba\nonumber
&\dot{x}=\mu_0+\mu_1 y+\mu_2 x+a_1 y^3 +b_0 x y^2+\sum_{l=1}^{\lfloor \frac{N-1}{4}\rfloor} (b_{l}+\mu_{2l+2}) x y^{4l}
+\sum_{l=0}^{\lfloor \frac{N-2}{4}\rfloor} \mu_{2l+3} x y^{4l+1},&\\\label{Eq03}
& \dot{y}= -x+\mu_2 y  +b_0 y^3+\sum_{l=1}^{\lfloor \frac{N-1}{4}\rfloor} (b_{l}+\mu_{2l+2}) y^{4l+1}
+\sum_{l=0}^{\lfloor \frac{N-2}{4}\rfloor} \mu_{2l+3} y^{4l+2},&
\ea modulo higher degree terms than \(N\) where each \(\mu_i\) is a polynomial in terms of controller coefficients \(\nu,\) \(a_1:= \frac{1}{3!}\frac{\partial^3 f}{\partial x^3}(0,0),\) \(b_0:= \frac{1}{8}\frac{\partial^3 f}{\partial x \partial y^2}(0,0)+\frac{1}{2} \frac{\partial^3 g}{\partial y^3}(0,0)\) and \(b_{i-4}\) (when \(i>4\)); also see \cite[Section 5.2]{GazorSadriSicon}. When \((\mu_0, \mu_1, \mu_2, \mu_3)\) is invertible, \(m=4\) and \(\mu_i\)-s for \(i=0, 1, 2, 3\) can play the same role as controller coefficients \(\nu\). Two parameters are sufficient to unfold the system for symmetry-preserving bifurcation control of \eqref{Eq03}. However, all four parameters are needed for fully unfolding the system to include both symmetry-preserving and symmetry-breaking bifurcations. In other words, system \eqref{Eq03} is a codimension-four universal unfolding system. This means that the system has enough controller coefficients in the right places to disclose and enforce all potential local bifurcation scenarios from the differential system \eqref{Eq02}. Then, controller coefficients dominate possible modeling imperfections and small noises. The origin is always an equilibrium as a primary equilibrium when \(\mu_0=0\). The primary equilibrium is deviated from the origin when \(\mu_0\) (\(|\mu_0|\ll 1\)) is slightly varied. Two new equilibria bifurcate from the origin and thus, they are called by secondary equilibria \(E_\pm.\) Limit cycles \(\mathscr{C}^1_\pm\) and \(\mathscr{C}^2_\pm\) are bifurcated through Hopf and Bautin bifurcations from the {\it secondary equilibria} and are called as {\it tertiary} limit cycles. Hence, homoclinic and heteroclinic cycles are quaternary cycles and their bifurcations are called quaternary bifurcations. A homoclinic cycle is a trajectory that joins an equilibrium to itself as the red cycle in Figure \ref{fig23}. A heteroclinic cycle here consists of two trajectories where they connect two different equilibria; see the red cycle in Figure \ref{fig24}. {\it Critical controller sets} in terms of controller coefficients distinguish these bifurcations and determine where they precisely occur. Small size neighborhood validity of the secondary, tertiary, and quaternary bifurcations reduces potential applications of the theory in practical life problems. So, it is important to derive and present high order symbolic estimates for critical controller sets.

Bifurcation theory has an extensive literature of more than a century; \eg see \cite{GuckenheimerDangelmayr,Perko94,PerkoBook,GaetaBif}. However, their application in control engineering has had a slow progress due to its technical challenges. 
Most theoretical results address {\it possible bifurcation scenarios}. However, they fail to precisely locate the desired dynamics in terms of the original parameters and constants of practical life problems. Bifurcation control has recently attracted some researchers and includes important contributions; see \cite{GazorShoghiJDE,GazorShoghiCMD,GazorSadri,YuBifCont,ChenBifuControl,ChenBifControl2000,Kang98,KangIEEE,
Kang04,Hamzi2,GazorSadriSicon}.
Kang \cite{Kang98} considered control systems with a single input and characterized generic normal form systems by their quadratic invariants and the equilibrium sets. Kang et al \cite{KangIEEE} considered a singular system with uncontrollable linearization that includes a single uncontrollable mode (a zero eigenvalue). They showed that there are always nearby bifurcated equilibria around the origin, where they are controllable. This is a significant contribution as many of approaches in nonlinear control cannot be applied to systems with uncontrollable linearizations. Our introduced approach in this paper can be considered as a generalization of Kang et al \cite{KangIEEE} to include uncontrollable modes with two zero eigenvalues. Wu and Yu \cite{YuBifCont} proposed a method to delay static and dynamic bifurcations. Hamzi et al \cite{Kang04} considered a family of singular control systems with two purely imaginary uncontrollable modes. They presented state-feedback controller for their stabilization and the quadratic invariants characterizing the generic Hopf bifurcation control. Hamzi \cite{Hamzi2} considered the singular control families with double zero uncontrollable modes. He obtained the quadratic invariants of the family. The quadratic invariants are used for synthesizing a quadratic stabilizing controller and for characterizing the generic normal forms. Our results generalizes his results for a degenerate normal form family and symmetry-breaking cases, where we addresses full singular bifurcation control problem. We studied bifurcation control for families of Bogdanov-Takens singular systems including a \(\mathbb{Z}_2\)-equivariant family in \cite{GazorSadriSicon}. In this paper we skip the existing dynamics that are implied from our equivariant dynamical analysis in \cite{GazorSadriSicon}; see Remark \ref{SymBrkPrs}. Further, we have investigated the bifurcation control for a singular family on a three dimensional central manifold with two
imaginary uncontrollable modes in \cite{GazorSadri}. Gazor and Shoghi \cite{GazorShoghiJDE,GazorShoghiJDE22} considered applications of bifurcation control for sound intensities in music; also see \cite{GazorShoghiCMD,GuckenheimerDangelmayr,Perko94,PerkoBook,Novaes2020Nonl,Novaes2021}. An efficient nonlinear time transformation method has been recently developed in \cite{algabaSIADS} and applied to obtained highly accurate estimates for global bifurcations of homoclinic and heteroclinic varieties of codimension two singularities. We employ a novel generalisation of their approach to deal with our codimension four controlled system; also see \cite{AlgabaMelnikov3bogdanov,AlgabaMelnikov2020,AlgabaMelnikov,Algabachua1,Novaes2021,Novaes2020Nonl} for closely related results and techniques.

This draft is organized as follows. Jet sufficiency of equilibrium bifurcations are discussed in Section \ref{Sec2}. Finite determinancy In section \ref{sec2}, we study local symmetry breaking bifurcations for controlled system \eqref{Eq03}. We derive and present a rich list of  controller manifolds in terms of controller coefficients in this section. These give rise to an effective tool not only to cause or prevent bifurcations but also to satisfy control objectives such as feedback regularization and stabilization techniques via supercritical Hopf and homoclinic/heteroclinic bifurcations. Further, we derive leading estimates for homoclinic and heteroclinic cycles as well as the amplitudes and angular frequencies of bifurcated limit cycles in terms of controller coefficients. These provide efficient criteria for the amplitude size control and frequency management of the nearby oscillating dynamics. Bifurcation control of \(\mathbb{Z}_2\)-equivariant systems is considered in section \ref{SecBifCont} using a single input quadratic controller. We show how a system with an uncontrollable linearization can be treated through our proposed symmetry-breaking bifurcation control. We show that the controlled system admits two saddle node controller manifolds, two Hopf controller sets (one is supercritical and the other is subcritical) and two homoclinic controller sets for linearly uncontrollable case; see Theorem \ref{UnconThm} and its proof. Section \ref{sec4} is dedicated to illustrate our results on Chua differential system. This system in the vicinity of its Bogdanov-Takens singularity may undergo a pitchfork singularity, three Hopf bifurcations and three homoclinic bifurcations. They can be all realised through an input state-feedback controller design. We explore feedback regularization of the origin and state feedback stabilization via supercritical Hopf bifurcations. We prove that there exist two regions of controller coefficient choices for feedback regularization while feedback stabilization approach admits controller coefficient choices from other two regions. Freedoms of choices for controller coefficients within these regions facilitate the amplitude size and frequency managements of the oscillating dynamics. We show how a very small single-input controller readily enforces all these control objective; \eg see figures \ref{Fig7a}, \ref{FigStab1b} and \ref{FigStab2b}. Finally, conclusions are drawn in section \ref{secConclusion}.

\section{Finite determinacy for bifurcations of equilibria }\label{Sec2}

The infinite Taylor expansion of normal forms is an obstacle both in the theoretical analysis in bifurcation theory and in the practical computations using computer algebra systems; \eg see \cite{GazorSadri,GaetaFurther,GaetaPoincare,GaetaBif}. One usually truncates the infinite normal forms up to certain degree \(k\); this is called a \(k\)-{\it jet} of the differential system. In other words, one ignores the higher order terms in normal form expansion. Higher order terms are never derived due to the complexity of formulas and impractical computations using computers. Hence, the qualitative equivalence of the truncated normal forms and the original normal forms is an important question and needs to be thoroughly investigated. The bifurcation analysis of the truncated normal forms may be inaccurate, misleading and/or essentially wrong when the question of {\it jet sufficiency} is not investigated.

Consider an equivalence relation and recall that a property is defined as a {\it qualitative property} when it is invariant under the equivalence relation. When a truncated normal form reflects the qualitative dynamics of the original system, we refer to the system as a {\it finitely determined system}. The jet sufficiency refers to the degree upon which the truncated system is sufficient to fully represent the {\it qualitative dynamics} of the original system. Different equivalence relations are necessary depending on the intended properties in our analysis. We first consider contact equivalence relation and appeal to singularity theory. Contact equivalence is the finest equivalence relation for equilibrium bifurcations of a given vector field. The results from singularity theory are compatible with normal forms of various types, {\rm i.e.}, normal forms, orbital normal forms, parametric normal forms, and also concepts such as universal asymptotic unfolding, etc; see \cite{GazorSadri}.
We follow \cite{GolubitskyStewartBook,GazorKazemi} and consider one of the parameters as a distinguished parameter and denote it by \(\lambda\). Two mappings \(f(x, \lambda)\) and \(g(x, \lambda)\) are contact equivalent if there exists a local diffeomorphism germ \(X(x, \lambda)\) with \(X(0, 0) = 0,\) locally invertible map \(\Lambda(\lambda),\) \(\Lambda(0)=0,\) and a locally nonsingular \(n\times n\) matrix \(S(x, \lambda)\) such that \(f(x, \lambda)=S(x, \lambda)g(X(x, \lambda), \Lambda(\lambda))\); see \cite[page 166]{GolubitskyStewartBook} and \cite{GazorKazemi}.

\begin{rem}\label{Rem2.1}
Any \(\mathbb{Z}_2\)-equivariant differential system \eqref{Eq01} can be transformed into a \(\mathbb{Z}_2\)-equivariant differential system using permissible \(\mathbb{Z}_2\)-invariant polynomial transformations. Then, the remaining terms in the normal form system will be the \(\mathbb{Z}_2\)-equivariant terms from those remaining in the normal form system in \cite[Thoerem 2.9]{GazorSadriSicon}.  As similar to \cite[Theorem 3.3]{GazorSadriSicon}, any multiple parametric perturbation \eqref{Eq02} (including \(\mathbb{Z}_2\)-breaking perturbation terms) of the \(\mathbb{Z}_2\)-equivariant differential system \eqref{Eq01} can be transformed into \eqref{Eq03}.
\end{rem}
Since we are dealing with local bifurcations of vector fields, we define germs of vector fields at the origin. Two vector fields are defined as germ equivalent when there is a neighborhood where they are equal on that. Then, each equivalent class of vector fields from germ equivalent relation is called a germ vector field. The steady-state bifurcation problem associated with the normal form system of the generalized cusp case of Bogdanov-takens singularity is given by \(G(x, y, \lambda):=(G_1, G_2)=(0, 0)\) where \(G_1(0, 0, 0)= G_2(0, 0, 0)=0\),
\begin{eqnarray}\label{G1G2}
&(G_1, G_2):=\left(a_1y^{3}+b_0 x y^2+\lambda+\sum_{l=1}^{\lfloor \frac{N-1}{4}\rfloor} b_{l} x y^{4l},
a_0x+b_0  y^{3}+\sum_{l=1}^{\lfloor \frac{N-1}{4}\rfloor} b_{l} y^{4l+1}\right)+h.o.t.,&
\end{eqnarray} \(\lambda:= \mu_0,\) \(a_0a_1b_0 \neq0\); see \cite{GazorSadriSicon} for more details.

\begin{thm} \label{Thm2.3}
The germ \(G\) is contact equivalent with \(G+p\) for all \(p\in \overrightarrow{\mathcal{M}}^{4}\).
\end{thm}
\bpr Normal form results and the approach in singularity theory are compatible. The argument relies on the fact that smooth changes of coordinates, time rescaling and reparametrization all are compatible with contact equivalence relation and transform a germ to its contact equivalent germ. Furthermore, formal normal forms can be extended into smooth cases using Borel lemma modulo flat parts. We follow \cite[Definition 7.1, Proposition 1.4, Theorem 7.2 and Theorem 7.4]{GolubitskyStewartBook} and instead prove that \(\overrightarrow{\mathcal{M}}^{4}\subseteq\mathcal{K}(G),\) where \(\overrightarrow{\mathcal{M}}\) is the generated module
\begin{equation}\label{Mcal}
\overrightarrow{\mathcal{M}}:=\bigg<{x\choose 0}, {0\choose x}, {y \choose 0}, {0\choose y}, {\lambda \choose 0}, {0\choose \lambda }\bigg>
\end{equation}
over \(\E_{x, y, \lambda}.\) Thus, \(\overrightarrow{\mathcal{M}}^{4}\) is the space of all smooth vector field germs whose Taylor expansions do not have monomial vector field terms of degree less than 4. Here, \(\E_{x, y, \lambda}\) is the local ring of all smooth germs in \((x, y, \lambda)\)-variables. Define its unique maximal ideal by \(\mathcal{M}:=<x, y, \lambda>\).
Note that the flat vector fields live in \(\overrightarrow{\mathcal{M}}^3\). Define an \(\E_{x, y, \lambda}\)-module \(\mathcal{K}(G)\) that is generated by
\begin{equation} \label{Kappa}\mathcal{M}^2{G_{1x} \choose G_{2x}}, \mathcal{M}^2{G_{1\rho} \choose G_{2\rho}}, \mathcal{M}{G_{1} \choose 0}, \mathcal{M}{G_{2} \choose 0}, \mathcal{M}{0 \choose G_{1}}, \mathcal{M}{0 \choose G_{2}}.
\end{equation}
When \({\overrightarrow{\mathcal{M}}}^{k+1}\subseteq \mathcal{K}(f)\), by \cite[Theorem 7.2]{GolubitskyStewartBook}, \(G\) is \(k\)-sufficient with respect to contact equivalence relation. Now we choose \(a_1:=1\) to simplify the formulas.
Then, define
\begin{equation*}J:=\bigg<{0 \choose xy }\bigg>+{\mathcal{\overrightarrow{\mathcal{M}}}}^{4}.\end{equation*}
We first recall Nakayama Lemma. For any \(\E_{x, y, \lambda}\)-modules \(J\) and \(\mathcal{K},\) the Nakayama Lemma implies that \(J\subseteq \mathcal{K}\) if and only if \(J\subseteq \mathcal{K}+ \mathcal{M}J.\) Thereby, we denote \(\simeq\) for equalities modulo terms in \(\mathcal{M}J\):
\begin{equation*}
\lambda ^i x^j y^k{G_1\choose 0 }\simeq{\lambda ^{i+1} x^j y^k\choose 0 },  \hspace*{0.5cm}
\lambda ^i x^j y^k{0 \choose G_1 }\simeq {0\choose \lambda ^{i+1} x^j y^k}
,\\
\end{equation*} where \(i+j+k=3,\) \(i, j, k\geq 0.\) These conclude membership of \({\lambda^{4}\choose 0 },\)
\({\lambda x^{3}\choose 0 },\) \({\lambda x^{2}y\choose 0 },\) \({\lambda xy^2\choose 0 },\) \({\lambda y^{3}\choose 0 },\) \({0 \choose \lambda^{4} },\) \({0 \choose \lambda x^{3} },\) \({0 \choose \lambda x^{2}y },\) \({0 \choose \lambda xy^2 },\) and \({0 \choose \lambda y^{3} }\in \mathcal{K}_s+\mathcal{M}J.\) Next, we consider
\begin{eqnarray*}
&x^{3} {G_2\choose 0 }\simeq {-x^{4}\choose 0 },\quad
y^{3} {G_2\choose 0 }\simeq {-x y^{3}\choose 0 }, \quad
x^{3} {0\choose G_2 }\simeq {0 \choose -x^{4} },\quad
y^{3} {0 \choose G_2 }\simeq {0 \choose -x y^{3} }.&
\end{eqnarray*}
Similarly, we can show that \({ x^{i+1} y^j \choose 0 }, {0\choose x ^{i+1} y^j} \in \mathcal{K}_s+\mathcal{M}J\) for \(i+j=3\) and \(i,j\geq 0\). In the one hand, we have
\begin{equation*}
x^i y^j{G_2\choose 0 }\simeq { x^{i+1} y^j \choose 0 },  \hspace*{0.5cm} x^i y^j{0 \choose G_2 }\simeq {0\choose x ^{i+1} y^j}
\end{equation*}
Since \({x y^{3} \choose 0} \in \mathcal{K}+\mathcal{M}J\) and
\begin{equation*}
xy {G_{1x} \choose G_{2x} }\simeq{b_0 x y^{3} \choose -xy },\quad   {0\choose xy }\in \mathcal{K}+\mathcal{M}J.
\end{equation*}
On the other hand,
\begin{equation*}
y {0 \choose G_2}={0 \choose -xy+b_0 y^{4} } \quad \hbox{ implies that }\quad
{0 \choose  y^{4} } \in \mathcal{K}+\mathcal{M}J.
\end{equation*}
Finally, \(y^{2} {G_{1y} \choose G_{2y} }\simeq {y^{4}\choose 0 }\) completes the proof.
\epr

\section{Symmetry breaking and critical controller sets}\label{sec2}

Bifurcation control is facilitated by introducing {\it critical controller sets} or {\it controller manifolds}. Critical controller sets are typically codimension-one bifurcation manifolds within the controller coefficient space. This is a necessary condition of critical controller manifolds to provide a partition for the controller coefficient space into a finite number of connected regions. When controller coefficients from critical controller sets are subjected to small perturbations, the qualitative dynamics of the controlled system changes. The neighbourhood validity of controller sets is greatly influenced by the relative geometry of these manifolds in four-dimension. For example, a limit cycle \(\mathscr{C}_0\) may collide with equilibrium \(E_+\) and disappear at a homoclinic controller set \(T_{HmC+}.\) It can alternatively collide with another equilibrium \(E_-\) and disappear at different homoclinic controller set \(T_{HmC-}\). Therefore, controller manifold \(T_{HmC+}\) is no longer valid if \(\mathscr{C}_0\) has already been disappeared through \(T_{HmC-},\) or vice versa. We provide local criteria for our derived critical controller manifolds. These criteria are helpful and necessary for both deriving critical controller sets and the distinction of the neighbourhood validity within the controller coefficient space. Further, we apply alternative normalized systems and truncation degrees in our formulation. We derive the bifurcation controller sets in terms of symbolic constants and unknown controller coefficients. This is important for many control engineering applications.

Equation \eqref{Eq03} is \(\mathbb{Z}_2\)-equivariant for \(\mu_0=\mu_3=0\); see \cite{GazorSadriSicon} for \(\mathbb{Z}_2\)-equivariant bifurcation control. The symmetry-breaking occurs when either of \(\mu_3=0\) and \(\mu_0=0\) or both fail. Due to the geometric complexity of bifurcation controller sets in four dimensional space and errors of estimates, bifurcation controller sets are generally valid within certain neighborhoods. We provide these sets under two categories: symmetry breaking analysis using \(\mu_3\) (while  \(|\mu_0|\ll1\)) and symmetry-breaking analysis through \(\mu_0\) and \(\mu_3\) for \(\mu_0\mu_3 \neq 0\). Here, we use the notation of the little \({\scalebox{0.7}{$\mathscr{O}$}}\) for parameters, where they are generally polynomial functions of the original control parameters. 

\begin{prop}\label{Prop1} For \(|\mu_0|= {\scalebox{0.7}{$\mathscr{O}$}}(||(\mu_1, \mu_2, \mu_3)||^4),\) estimated primary pitchfork and Hopf bifurcation controller sets are
\be\label{SNr2s2nu10}
T_{P}:=\left\{(\mu_0, \mu_1, \mu_2, \mu_3)|\, \mu_1+{\mu_2}^2=0 \right\} \; \hbox{ and } \; T_{H}:=\left\{(\mu_0, \mu_1, \mu_2, \mu_3)|\, \mu_2=0, \mu_1>0\right\}.
\ee
\begin{itemize}
\item \cite[Theorem 5.1]{GazorSadriSicon}: We have a super-critical Hopf bifurcation when \(b_0<0.\) Bifurcated limit cycle \(\mathscr{C}_0\) appears when \(\mu_2>0\) and is asymptotically stable. For \(b_0>0,\) the system undergoes a sub-critical Hopf bifurcation and the limit cycle \(\mathscr{C}_0\) exists when \(\mu_2<0\) and is unstable. There is a local bifurcation of secondary equilibria \(E_{\pm}\) from the origin (or an equilibrium slightly deviated from the origin) in nearby of \(T_p.\)
\item Estimated radius and angular frequency for the bifurcated limit cycles \(\mathscr{C}_0\) are given by \(\sqrt{-\frac{2\mu_2}{b_0}}\) and \(\sqrt{-\mu_1}.\) These are
useful for magnitude and frequency management of the oscillating dynamics. There are a bifurcation controller set \(\mathscr{B}\) and an estimated hysteresis controller set \(\mathscr{H}\) given by
\ba\label{BHsets}
&\mathscr{B}:= \{(\mu_0, \mu_1, \mu_2, \mu_3)| \, \mu_0=0 \} \hbox{ and } \mathscr{H}:= \big\{(\mu_0, \mu_1, \mu_2, \mu_3)|\, \mu_0=\frac{8 {\mu_3}^3 {(-\mu_1)}^{\frac{3}{2}}}{27(a_1+2b_0\sqrt{-\mu_1})^2}
\big\}.&\ea
Critical controller sets \(\mathscr{B}\) and \(\mathscr{H}\) contribute into the symmetry-breaking and bifurcation diagram classification of the pitchfork singularity for controlled system \eqref{Eq03}; see Figure \ref{BHFig}.
\item A symbolic approximation for the secondary equilibria of system \eqref{Eq03} follows
\begin{footnotesize}
\begin{eqnarray}\label{Epmnu40}
E_\pm: \quad (x_\pm,y_\pm):=\left(\mu_2y_\pm+\mu_3{y_\pm}^2+b_0{y_\pm}^3, \dfrac{-\mu_2\mu_3\pm\sqrt{-a_1\mu_1-a_1{\mu_2}^2-2b_0\mu_1\mu_2-2b_0{\mu_2}^3-\mu_1{\mu_3}^2}}{a_1+2b_0 \mu_2+{\mu_3}^2}\right)\!.
\end{eqnarray}
\end{footnotesize}
\end{itemize}
\end{prop}

\begin{proof} The eigenvalues of the linearised system at the origin are \(\mu_2\pm\sqrt{-\mu_1},\) where we treat \(\mu_0\) as a small perturbation. Hence, equation \eqref{SNr2s2nu10} is a Hopf bifurcation manifold and \(\sqrt{-\mu_1}\) stands for the leading term in the angular velocity of the oscillating dynamics. Consider \(|\mu_2|, |\mu_3|={\scalebox{0.7}{$\mathscr{O}$}}(|\mu_1|).\) Note that parameters \(\mu_2\) and
Then, a normalized amplitude equation in polar coordinates \((\rho, \theta)\) is \(\dot{\rho}=\mu_2 \rho +\frac{b_0}{2}\rho^3+ {\scalebox{0.7}{$\mathscr{O}$}}(\rho^5, |\mu_1|^2).\) An estimated radius for the bifurcated limit cycle from this system is \(\sqrt{-\frac{2\mu_2}{b_0}}\); see proof of \cite[Theorem 5.1]{GazorSadriSicon}. Thus, the radius of the bifurcated limit cycle grows when \(\frac{\mu_2}{b_0}\) decreases. Further, we have a single zero singularity when \(\lambda_1=0\) for
\(\lambda_0:= {\mu_2}-\sqrt{-\mu_1}\); \ie \(\mu_1\leq 0.\) 
Then, a truncated and re-scaled differential equation on the center manifold (also see \cite{HamziKangCenter05}) is given by
\be\label{EqSteady}\dot{x}= \mu_0+\lambda_0 \sqrt{-\mu_1} x+\frac{\sqrt{-\mu_1}\mu_3}{2}x^2+\left(\frac{b_0{\sqrt{-\mu_1}}}{4}+\frac{a_1}{8}\right)x^3+ {\scalebox{0.7}{$\mathscr{O}$}}(||(\mu, x)||^4).
\ee
Now we appeal to singularity theory developed in \cite{GolubitskyStewartBook}. By \cite[Proposition 4.4]{GolubitskyStewartBook}, this equation is a universal unfolding for the pitchfork singularity. Critical set \(\mathscr{B}\) and hysteresis \(\mathscr{H}\) follow \cite[Pages 140]{GolubitskyStewartBook}. Treat \(\lambda_0\) as the distinguished bifurcation parameter and \(G(x, \lambda_0, \mu_0, \mu_1, \mu_3):= \mu_0+\lambda_0 \sqrt{-\mu_1} x+\frac{\sqrt{-\mu_1}\mu_3}{2}x^2+\left(\frac{b_0{\sqrt{-\mu_1}}}{4}+\frac{a_1}{8}\right)x^3\). Hence, the bifurcation set \(\mathscr{B}\) is obtained from \(G= \frac{\partial G}{\partial x}= \frac{\partial G}{\partial \lambda_0}=0\) while the hysteresis controller set \(\mathscr{H}\) is given by \(G= \frac{\partial G}{\partial x}= \frac{\partial^2 G}{{\partial x}^2}=0.\) Here,
\bes \frac{\partial G}{\partial x}= \lambda_0\sqrt{-\mu_1}+\mu_3 \sqrt{-\mu_1}x+3\left(\frac{b_0}{4}\sqrt{-\mu_1}+\frac{a_1}{8}\right) x^2,\ees
\(\frac{\partial G}{\partial \lambda_0}= x\sqrt{-\mu_1},\) and \(\frac{\partial^2 G}{{\partial x}^2}= \mu_3 \sqrt{-\mu_1}+3(\frac{b_0}{2}\sqrt{-\mu_1}+\frac{a_1}{4}) x.\)
Omitting \(x\) and \(\lambda_0\) from these equations gives rise to the governing equations for \(\mathscr{B}\) and \(\mathscr{H}\). To estimate the secondary equilibria \(E_\pm\), we derive \(x\) from the second steady-state equation corresponding with \eqref{Eq03}. Then, we substitute \(x\) into the first steady-state equation to obtain \((2b_0 \mu_2+{\mu_3}^2+a_1)y^3+2 \mu_2 \mu_3 y^2+(\mu_1+{\mu_2}^2)y=0\) modulo \(\mu_0\) and terms of degree four and higher in \(y\). Roots of this cubic polynomial are \(y=0\) and \(y_{\pm}\) in \eqref{Epmnu40}. Hence, the second steady-state equation for \eqref{Eq03} concludes the formula \eqref{Epmnu40}.
\end{proof}

\begin{thm}[Saddle-node controller sets]\label{Thm22} There are two saddle-node bifurcations at critical controller sets estimated by \(T^{\pm}_{SN}= \{(\mu_0, \mu_1, \mu_2, \mu_3)|\, \xi_\pm=0\}\) where
\begin{footnotesize}
\begin{eqnarray}\nonumber
&\xi_\pm:=\mu_0-\frac{\mp2\left(6b_0\mu_1\mu_2+3a_1\mu_1+3{\mu_3}^2\mu_1
+3a_1{\mu_2}^2-{\mu_2}^2{\mu_3}^2+6b_0 {\mu_2}^3\right)\sqrt{{\mu_2}^2{\mu_3}^2-6b_0 {\mu_2}^3
-3{\mu_3}^2\mu_1-3a_1{\mu_2}^2-6b_0\mu_1\mu_2-3a_1\mu_1}}{27\left(a_1+{\mu_3}^2+2b_0 \mu_2 \right) ^2}&\\\label{nu1eq}
&+\frac{2\left({\mu_2}^3{\mu_3}^3+18b_0 {\mu_2}^4\mu_3+9\mu_1\mu_2{\mu_3}^3+18b_0 \mu_1 {\mu_2}^2 \mu_3
+9a_1\mu_1\mu_2\mu_3+9a_1{\mu_2}^3\mu_3\right)}{27\left(a_1+ {\mu_3}^2+2b_0 \mu_2 \right) ^2}.&
\end{eqnarray}
\end{footnotesize}
\end{thm}
\begin{proof} The characteristic polynomial coefficients of \(\lambda^2+d_1 \lambda+d_2\) for Jacobian matrix of \eqref{Eq03} follows \(d_1\!:=\!4b_0 y^2+3\mu_3 y+2 \mu_2,\) and
\begin{eqnarray}\label{32}
& d_2\!:=\! 5 {b_0}^2 y^4+8b_0\mu_3{y}^3+3 \left( {\mu_3}^2+2 b_0\mu_2+a_1\right)y^2 +
4\mu_2\mu_3 y+{\mu_2}^2+\mu_1.\quad &
\end{eqnarray} The scalars \(d_1\) and \(d_2\) are arrays of the first column of Routh table (\eg see \cite{PHurwitz}).
We claim that there are two local equilibria, where they undergo saddle-node bifurcations. Let \((x, y)=(x_0, {y_0})\) be one of these two.
Then, \((x_0, {y_0})\) must satisfy \(d_2(y_0)= 0\) and the steady-state equations of \eqref{Eq03}. These provide a precise implicit formulation for the saddle-node singularity. Shift of coordinates \(Y=y-{y_0}\) and \(X=x-x_0\) give rise to
\begin{eqnarray}\nonumber
&\dot{X}\!=\!\left(\mu_3{ {y_0}}+b_0{{y_0}}^2+\mu_2 \right)\! X\!+ ( b_0{ {y_0}}\mu_2+3a_1{ {y_0}}+b_0\mu_3{{y_0}}^{2}+{{y_0}}^{3}{b_0}^{2}+b_0X) {Y}^{2}+ (\mu_3+2b_0{{y_0}}) XY&\\\nonumber
&+( \mu_1+\mu_3{ {y_0}}\mu_2+
2b_0{{y_0}}^{2}\mu_2+3a_1{{y_0}}^{2}+3\mu_3{{y_0}}^{3}b_0+2{{y_0}}^{4}{b_0}^2
+{\mu_3}^{2}{{y_0}}^2) Y+a_1{Y}^{3}+\xi_\pm,&
\\
&\dot{Y}= (\mu_2+2\mu_3{{y_0}}+3b_0{{y_0}}^2)Y-X+ (\mu_3+3b_0{{y_0}}) Y^2+b_0 Y^3.&
\end{eqnarray}
The eigenvalues \(\lambda_{\pm}\) of Jacobian are then given by
\begin{eqnarray*}
&\mu_2+\frac{3}{2}\mu_3 {y_0}+2b_0{{y_0}}^2\pm\frac{\mathbf{i}\sqrt{4{y_0}\mu_2\mu_3+8b_0{{y_0}}^2\mu_2+3{\mu_3}^2{{y_0}}^2+8{{y_0}}^3b_0\mu_3
+4{{y_0}}^4{b_0}^2+12a_1{{y_0}}^2+4\mu_1}}{2}.&
\end{eqnarray*} When \({ \emph{sign}}(\mu_2+\frac{3}{2}\mu_3 {y_0}+2b_0{{y_0}}^2\pm\frac{1}{2})=\pm1,\) there is a polynomial \(h_\pm(z, \xi_\pm):=\gamma_0\xi_\pm+\gamma_1z^2+\gamma_2\xi_\pm^2+\gamma_3z\xi_\pm,\) \(z\in \{X, Y\},\) for a quadratic approximation of the center manifold. We here assume that \(|\xi_\pm|\ll |y_0|.\) We apply the center manifold reduction procedure to obtain \(\gamma_0 =-\frac{1}{4b_0{{y_0}}^2},\) \(\gamma_1 = \frac{2{b_0}^2 {{y_0}}^2-3a_1}{64{b_0}^3 {y_0}^5},\) \(\gamma_2=0,\) and \(\gamma_3 =\frac{4{b_0}^2 {{y_0}}^2-3a_1}{64{b_0}^4 {y_0}^7}.\) By a time-rescaling, the governing differential equation on the center manifold follows
\begin{eqnarray}
&\dot{Y}=\big(\frac{5}{8}{y_0}^2+\frac{3a_1}{16{b_0}^{2}}\big) Y^2+\xi_\pm \big(\frac{3}{16b_0}+\frac{3a_1}{32{b_0}^3 {{y_0}}^2}\big) Y+\xi_\pm\big({y_0}^3
+\frac{3a_1{\xi_\pm}}{256{b_0}^4{y_0}^4}+\frac{\xi_\pm}{128{b_0}^{2}{y_0}^{2}}\big).&
\end{eqnarray} The discriminant is given by \(\frac{\xi_\pm\left(\xi_\pm-48a_1{y_0}^3-160{b_0}^2{y_0}^5\right)}{64{b_0}^2{y_0}^6}.\) This is a saddle-node bifurcation, where we have two new local equilibria for \({\emph sign}(3a_1{y_0}+10{b_0}^2{y_0}^3)\xi_\pm<0\). We have no new equilibrium when \({\emph sign}(3a_1{y_0}+10{b_0}^2{y_0}^3)\xi_\pm>0\) and \(|\xi_\pm|\) is sufficiently small. To obtain the symbolic estimated critical controller sets \eqref{nu1eq}, we consider cubic truncations of \(d_2(y_0)= 0\) with respect to \(y_0.\) Then, \(y_0\) follows equation \eqref{Epmnu40}. This confirms our claim for two local equilibria with a saddle-node singularity. Substituting them into the truncated equation for \(d_2(y_0)= 0,\) we derive the estimated critical controller sets \eqref{nu1eq}.
\end{proof}

\begin{thm}[Supercritical and subcritical Hopf bifurcations from \(E_\pm\)]\label{thm3}
We assume that \(|\mu_0|= {\scalebox{0.7}{$\mathscr{O}$}}(||(\mu_1, \mu_2, \mu_3)||^4),\) \(|\mu_3|={\scalebox{0.7}{$\mathscr{O}$}}(||\mu_1,\mu_2^2||),\) and \(a_1 > 0\). Then, approximated Hopf controller sets for degenerate Hopf singularities associated with \(E_+\) and \(E_{-}\) are defined by \(T_{H\pm}\):
\begin{footnotesize}
\ba\label{eta0}
&\left\{\mu |\, 2{a_1}^\frac{7}{2}\mu_2-3 {a_1}^\frac{5}{2} \mu_2 {\mu_3}^2-4b_0 {a_1}^{\frac{3}{2}} (a_1-2b_0 \mu_2-{\mu_3}^2)(\mu_1+{\mu_2}^2)\pm\frac{{a_1}^2(6 a_1-22 b_0\mu_2-3{\mu_3}^2)\mu_3\sqrt{-\mu_1-{\mu_2}^2}}{2}=0\right\},\quad&
\ea
\end{footnotesize} where \(\mu =(\mu_0, \mu_1, \mu_2, \mu_3).\)
Due to the restrictions on control coefficients, a full Bautin bifurcation does not occur here; see Theorem \ref{THmBautin}. Yet, one tertiary limit cycle \(\mathscr{C}^1_\pm\) bifurcates from either of the equilibria \(E_\pm\) when controller coefficients cross critical controller sets \(T_{H\pm}\) and \(b_0\eta_\pm>0\), respectively. Here, \(\eta_\pm:=\mu_2(-\mu_1-{\mu_2}^2)^{\frac{-1}{2}}\pm\frac{3}{2}\frac{\mu_3}{\sqrt{a_1}}+\frac{2b_0\sqrt{-\mu_1-{\mu_2}^2}}{a_1}. \)
The bifurcation is supercritical when \(\eta_\pm>0\) and subcritical for \(\eta_\pm<0.\) The leading terms for the radius and angular velocity of the bifurcated limit cycles are \(\frac{8\sqrt{2}\sqrt{7a_1{b_0}^{-1}\mu_2-46(\mu_1+{\mu_2}^2)}}{7a_1\emph{sign}(b_0)}-\frac{64\sqrt{-\mu_1-{\mu_2}^2}}{7a_1}\) and \(\sqrt{2}(-\mu_1-{\mu_2}^2),\) respectively.
\end{thm}
\begin{proof} To facilitate Hopf bifurcation analysis, we transform differential system \eqref{Eq03} into an alternative normal form system
\(\dot{\check{x}}=\mu_0+ \mu_1 \check{y}+ 2\mu_2 \check{x}+ {\mu_2}^2\check{y}+3 \mu_3 \check{x} \check{y}+a_1 \check{y}^3+4b_0 \check{x} \check{y}^2+2b_0\mu_2  \check{y}^3+{b_0}^2 \check{y}^5,\) \(\dot{\check{y}}=-\check{x}.\) The estimated secondary equilibria \(E_\pm\) in the new coordinate system turns out to be \((\check{x}, \check{y})=(0, y_\pm),\) where \(y_\pm\) is given in \eqref{Epmnu40}. Consider changes of controller coefficients \(\mu_1:=-{\nu_2}^2-{\mu_2}^2\) and \(\mu_2:=\nu_3\nu_2\). Then, four-degree truncated traces (multiplied with \({\nu_2}^{-1}\)) of Jacobian matrices at the equilibria \((\check{x}, \check{y})=(0, y_{\pm})\) follow
\begin{eqnarray*}
&8b_0{a_1}^{\frac{5}{2}}\nu_2\pm6{a_1}^{3}\mu_3+4{a_1}^{\frac{7}{2}}\nu_3\mp3{a_1}^{2}{\mu_3}^{3}-16{a_1}^{\frac{3}{2}}{b_0}^{2}\nu_3 {\nu_2}^{2}&\\&
\mp 22b_0{a_1}^{2}\mu_3\nu_2\nu_3-8{\mu_3}^{2}b_0{a_1}^{\frac{3}{2}}\nu_2-6{a_1}^{\frac{5}{2}}{\mu_3}^{2}\nu_3=0.&
\end{eqnarray*} This gives rise to Hopf controller manifolds for \(T_{H\pm}\) in \eqref{eta0}. For these Hopf singularities, we consider the quadratic truncated traces of Jacobian matrices and obtain the variable \(\nu_3.\) This gives rise to the introduction of \(\eta_\pm:= \nu_3\pm\frac{3}{2}\frac{\nu_4}{\sqrt{a_1}}+\frac{2b_0\nu_2}{a_1}.\) This is the same \(\eta_\pm\) as in the above in terms of \(\mu_i.\) Now we shift the equilibria to the origin via \(\tilde{x}=\check{x},\; \tilde{y}=\check{y}-y_{\pm}.\) Next, we apply linear transformations \(X=\sqrt{2} \sqrt{-\mu_1-{\mu_2}^2}\tilde{y},\) \(Y=\tilde{x},\) primary shift of coordinates and a time rescaling \(\tau= \nu_2t\) to obtain \(\dot{X}=-\sqrt{2} {\nu_2}^2Y\) and \(\dot{Y}\), where \(\dot{Y}\) is given by
\begin{eqnarray*}
&\left(\sqrt{2}{\nu_2}^2\mp\frac {9\sqrt {2}b_0{\nu_2}^3\mu_3}{2a_1 \sqrt{a_1}}
-{\frac{7\sqrt{2}{b_0}^2{\nu_2}^4}{2{a_1}^2}}-
{\frac{3\sqrt{2}{\nu_2}^2{\mu_3}^2}{2a_1}}\pm{\frac{2\sqrt{2}\eta{\nu_2}^2\mu_3}
{\sqrt {a_1}}}\pm{\frac {8b_0{\nu_2}^2 Y}{\sqrt {a_1}}}
\pm{\frac {3\sqrt{2}\eta b_0{\nu_2}^3}{a_1}} +3\mu_3\nu_2 Y\right) X&\\\nonumber
&+\left(4b_0 \nu_2 Y\mp {\frac{\sqrt{2}{b_0}^2{\nu_2}^3}{a_1 \sqrt{a_1}}}
\pm{\frac{3\sqrt{2}\eta b_0{\nu_2}^2}{\sqrt {a_1}}}+\sqrt{2}\eta \nu_2\mu_3-{\frac{\sqrt{2}b_0{\nu_2}^2\mu_3}{2a_1}}
\pm\frac{3\sqrt{2a_1}\nu_2}{2}\right) X^2+ h.o.t..&
\end{eqnarray*}
By a {\sc Maple} programming (\eg \cite{GazorYuSpec}), an estimated parametric normalized amplitude equation is given by
\begin{eqnarray*}
&\dot{\rho}= \rho(A_1+ A_2R+ A_3R^2)+{\scalebox{0.7}{$\mathscr{O}$}}(\rho^6), \;\; A_1=\eta_\pm\sqrt{-\mu_1-{\mu_2}^2},\;\:
A_2=-b_0\sqrt{-\mu_1-{\mu_2}^2}\pm\frac{9}{16}\mu_3 \sqrt{a_1}, &
\end{eqnarray*} and \(A_3=-\frac{7 a_1 b_0}{128}.\) Here \(R:= \rho^2.\) For sufficiently small values of \(\eta,\) the discriminant of \(A_1+ A_2R+ A_3R^2\) is always positive. Since \(|\mu_3|={\scalebox{0.7}{$\mathscr{O}$}}(||\mu_1,\mu_2^2||)\) and \(\frac{A_2}{A_3}=\frac{128\sqrt{-{\nu_3}^2-\nu_2}}{-7 a_1}+\frac{72 \mu_3}{7 b_0 \sqrt{a_1}},\) the sum of its roots is always negative while \(\frac{A_1}{A_3}=-\frac{128\eta_\pm\sqrt{-{\nu_3}^2-\nu_2}}{7 a_1 b_0}\) is negative iff \(b_0\eta_\pm>0.\)
Indeed, we have a positive root only when \(b_0\eta_\pm>0\) and otherwise, there is no positive root. These determine when the system admits a local limit cycle.
Despite the degeneracy of Hopf singularity, controller restrictions limit the bifurcations to at most one limit cycle from either of \(E_\pm.\)
\end{proof}
\begin{figure}
\begin{center}
\subfigure[Critical controller sets]
{\includegraphics[width=.2\columnwidth,height=.19\columnwidth]{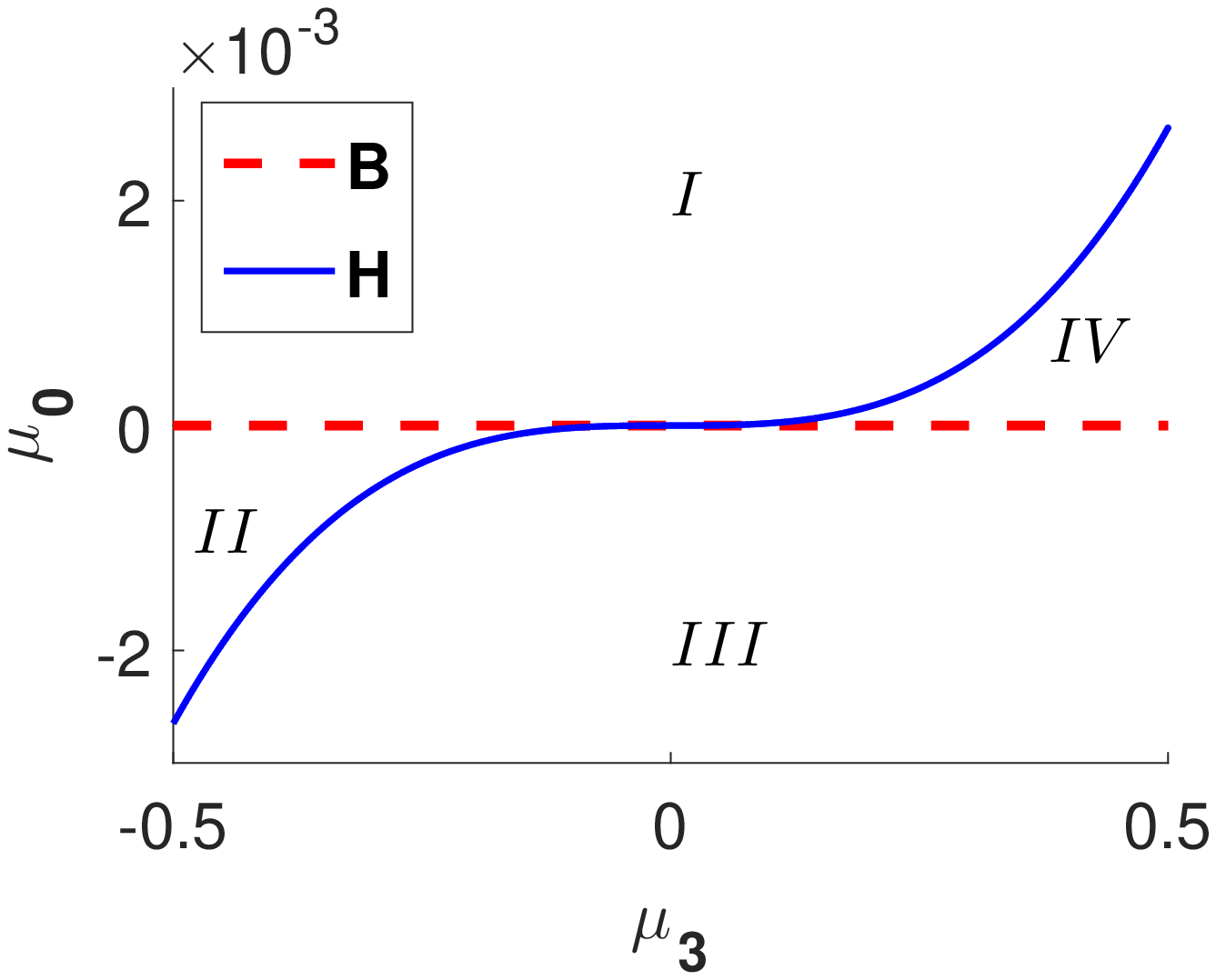}}
\subfigure[Region I]
{\includegraphics[width=.19\columnwidth,height=.19\columnwidth]{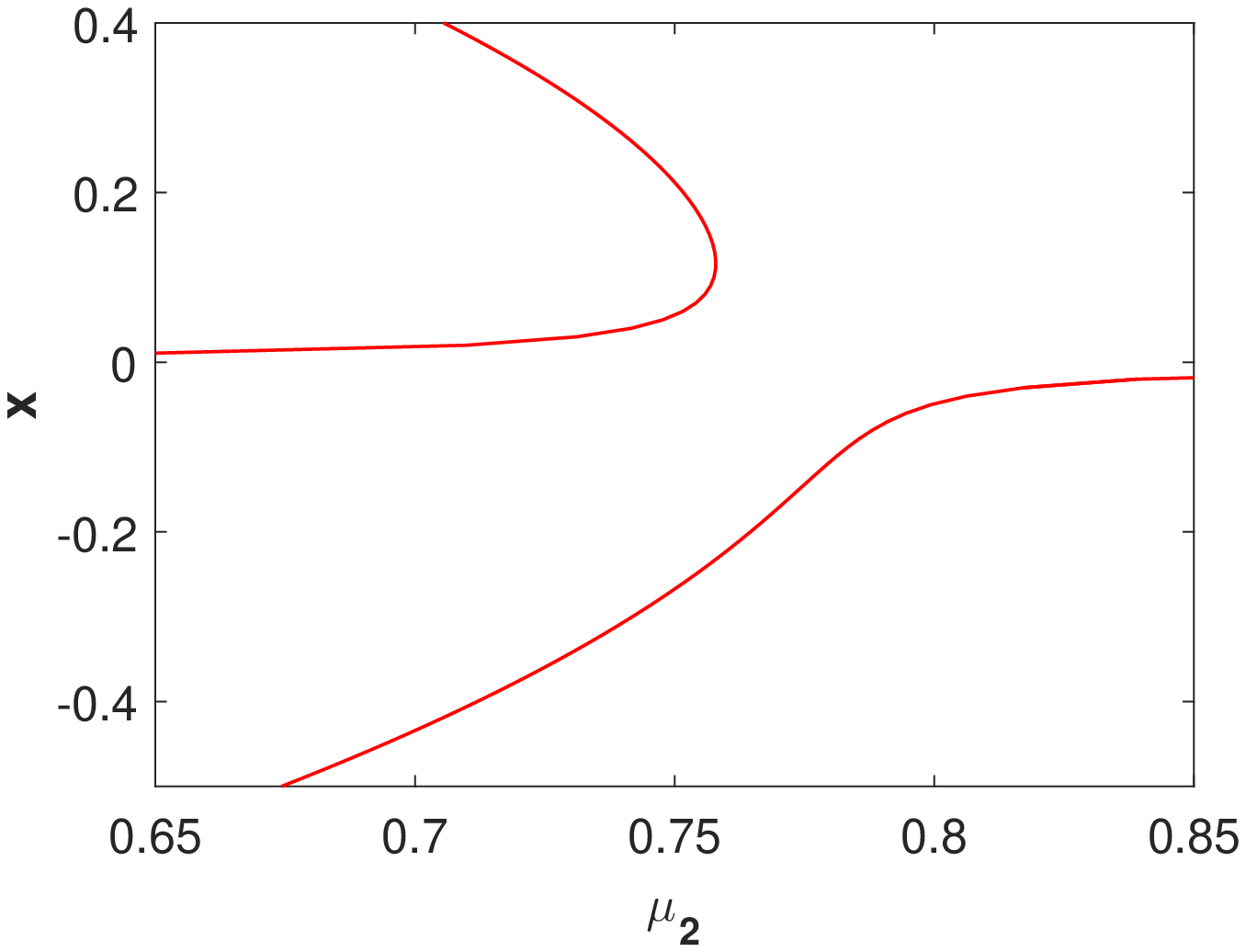}}
\subfigure[Region II]
{\includegraphics[width=.19\columnwidth,height=.19\columnwidth]{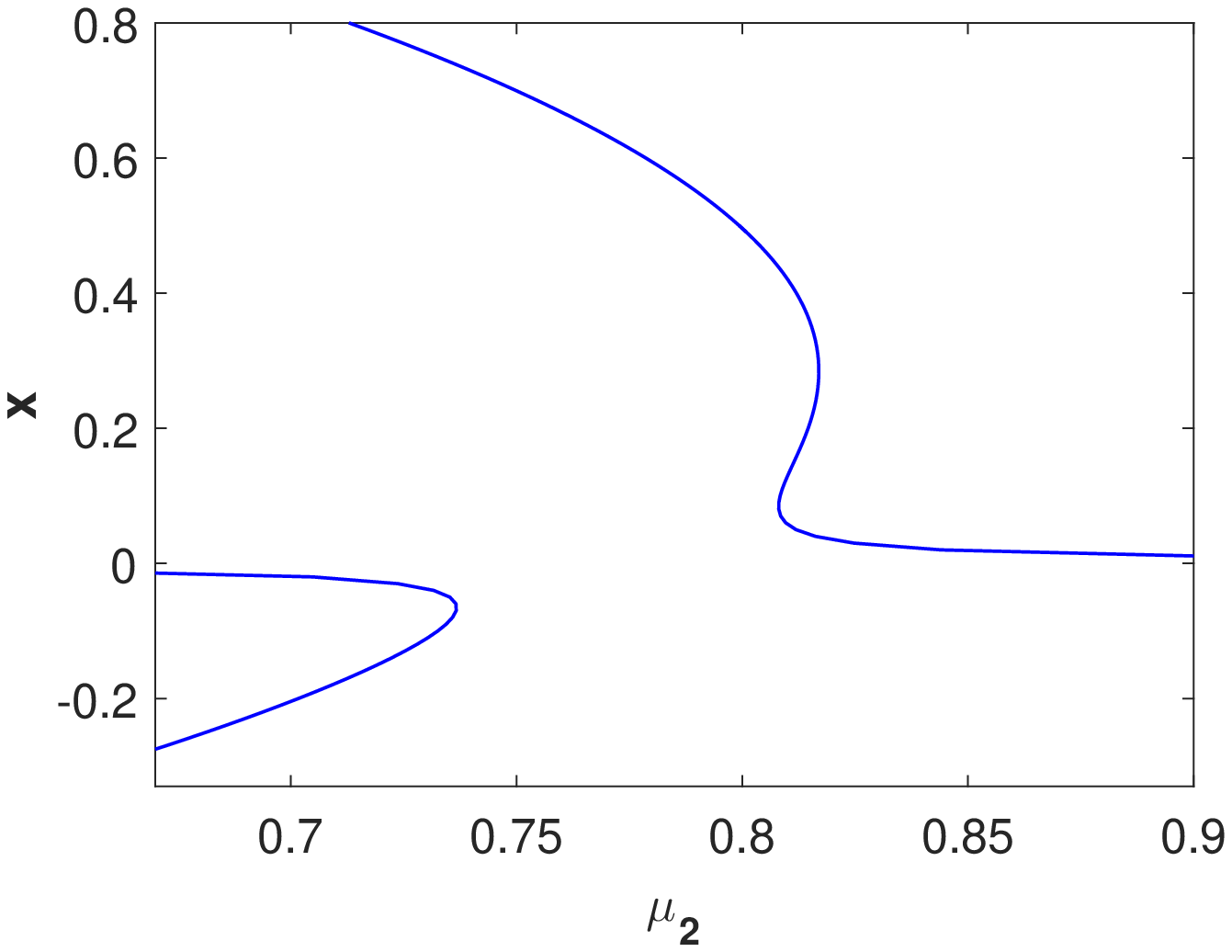}}
\subfigure[Region III]
{\includegraphics[width=.19\columnwidth,height=.19\columnwidth]{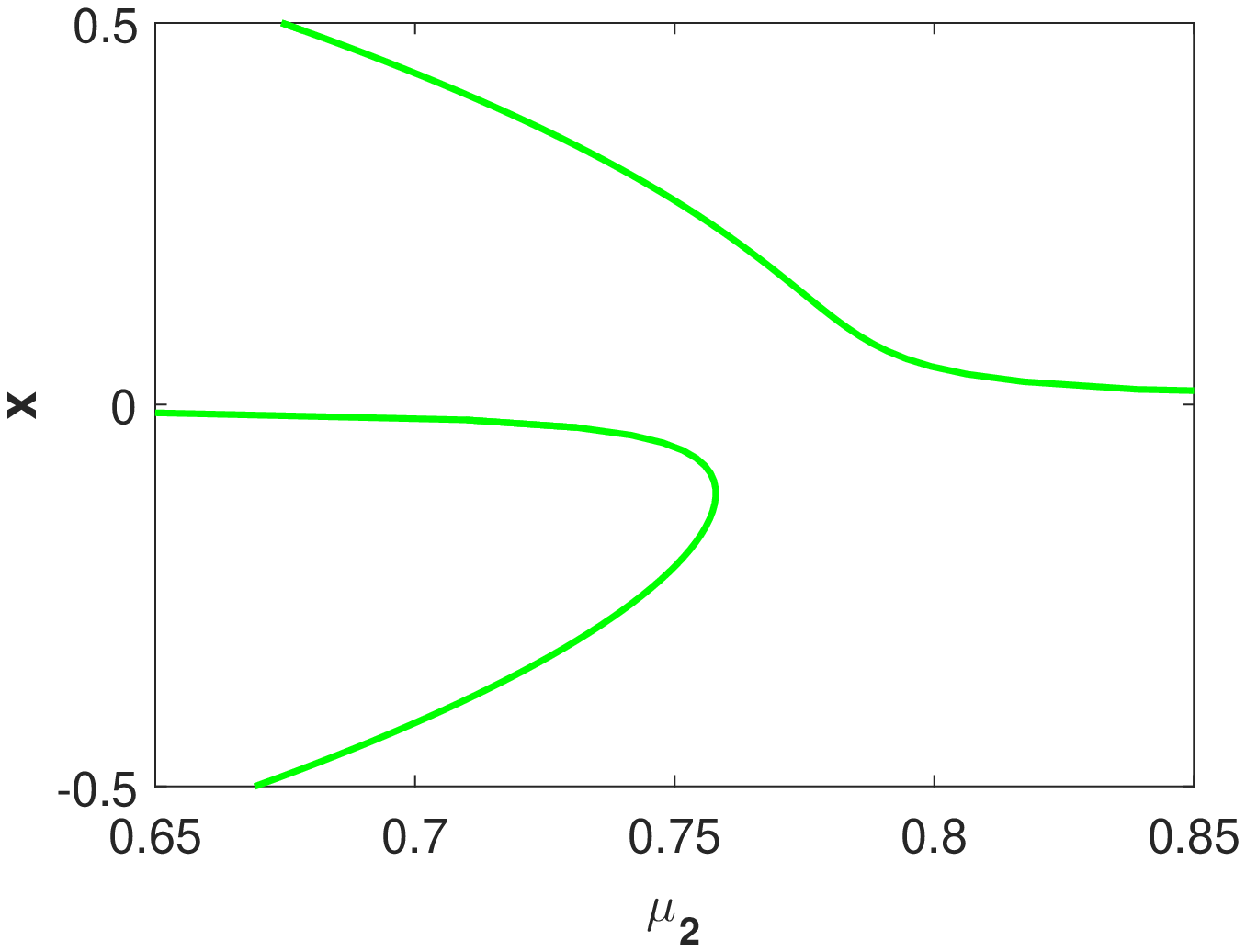}}
\subfigure[Region IV]
{\includegraphics[width=.19\columnwidth,height=.19\columnwidth]{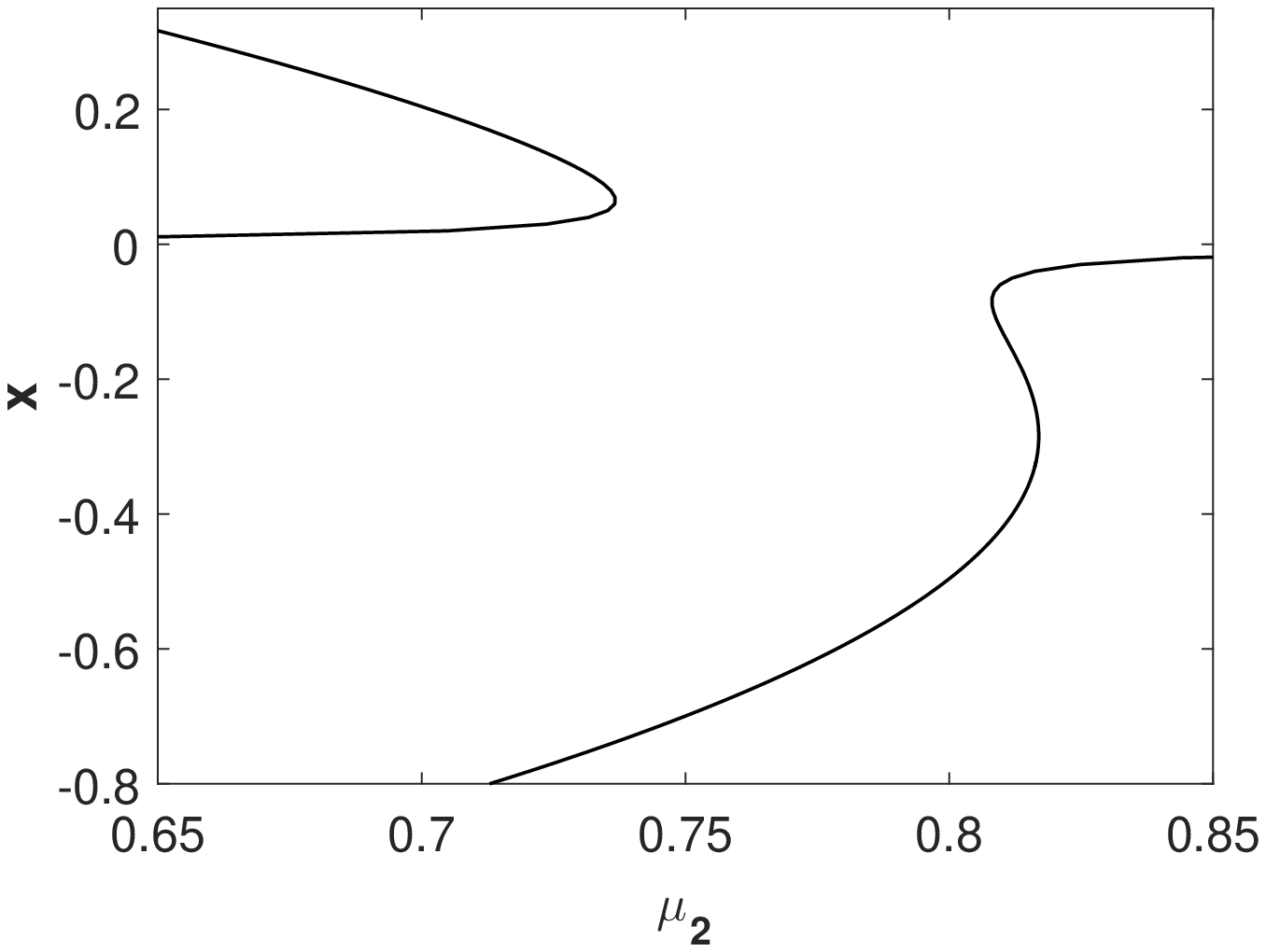}}
\end{center}
\vspace{-0.250 in}
\caption{Controller varieties \(\mathscr{B}\) and \(\mathscr{H}\) in equations \eqref{BHsets} and the numerical steady-state bifurcation diagrams associated with equation \eqref{EqSteady} and controller coefficient regions I-IV.}\label{BHFig}
\vspace{-0.100 in}
\end{figure}
\begin{figure}
\begin{center}
\subfigure[\(\mu_3=0.1,\) \(a_1= b_0=1\)]
{\includegraphics[width=.24\columnwidth,height=.19\columnwidth]{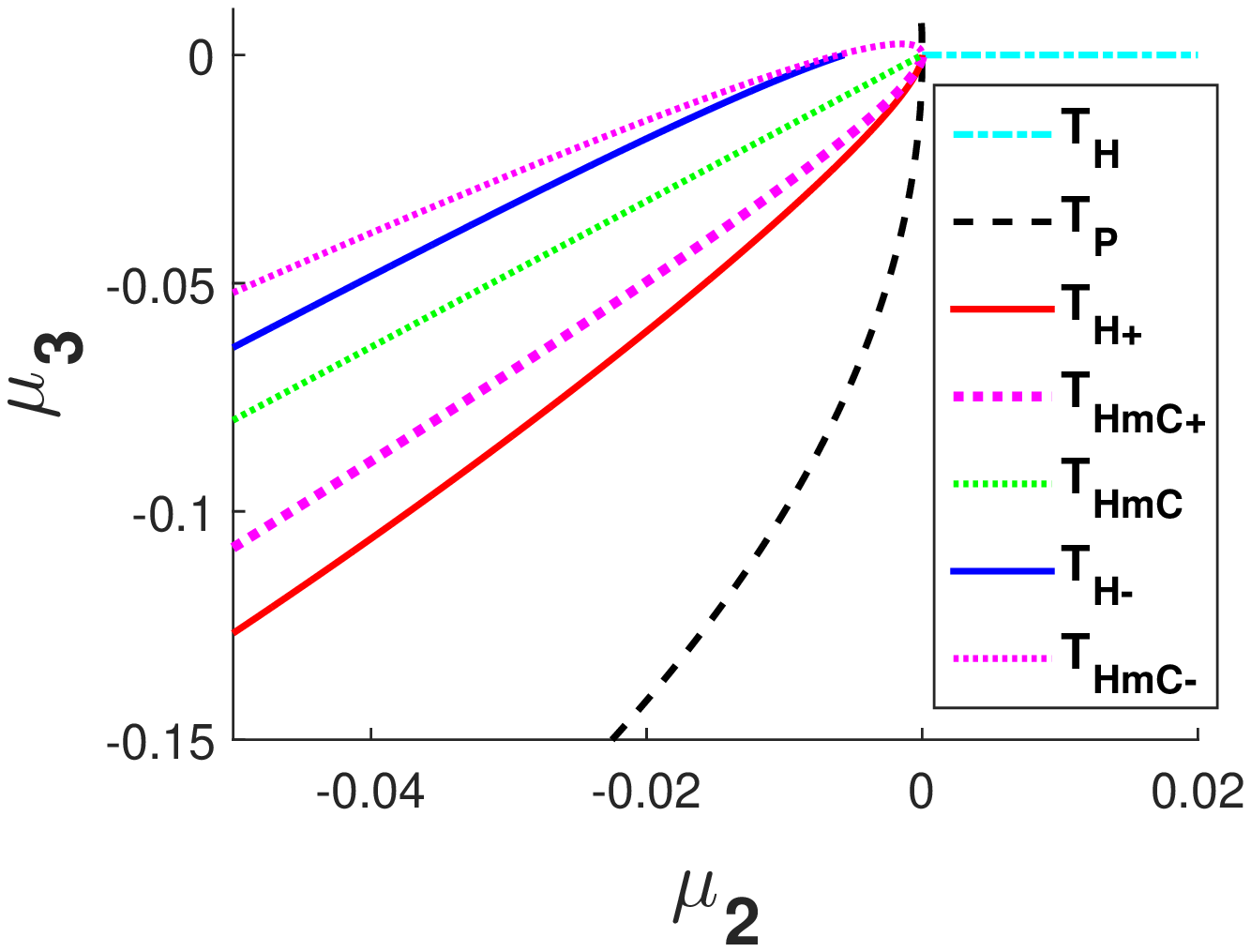}}
\subfigure[\(\mu_3=-0.1,\) \(a_1= b_0=1\)]
{\includegraphics[width=.24\columnwidth,height=.19\columnwidth]{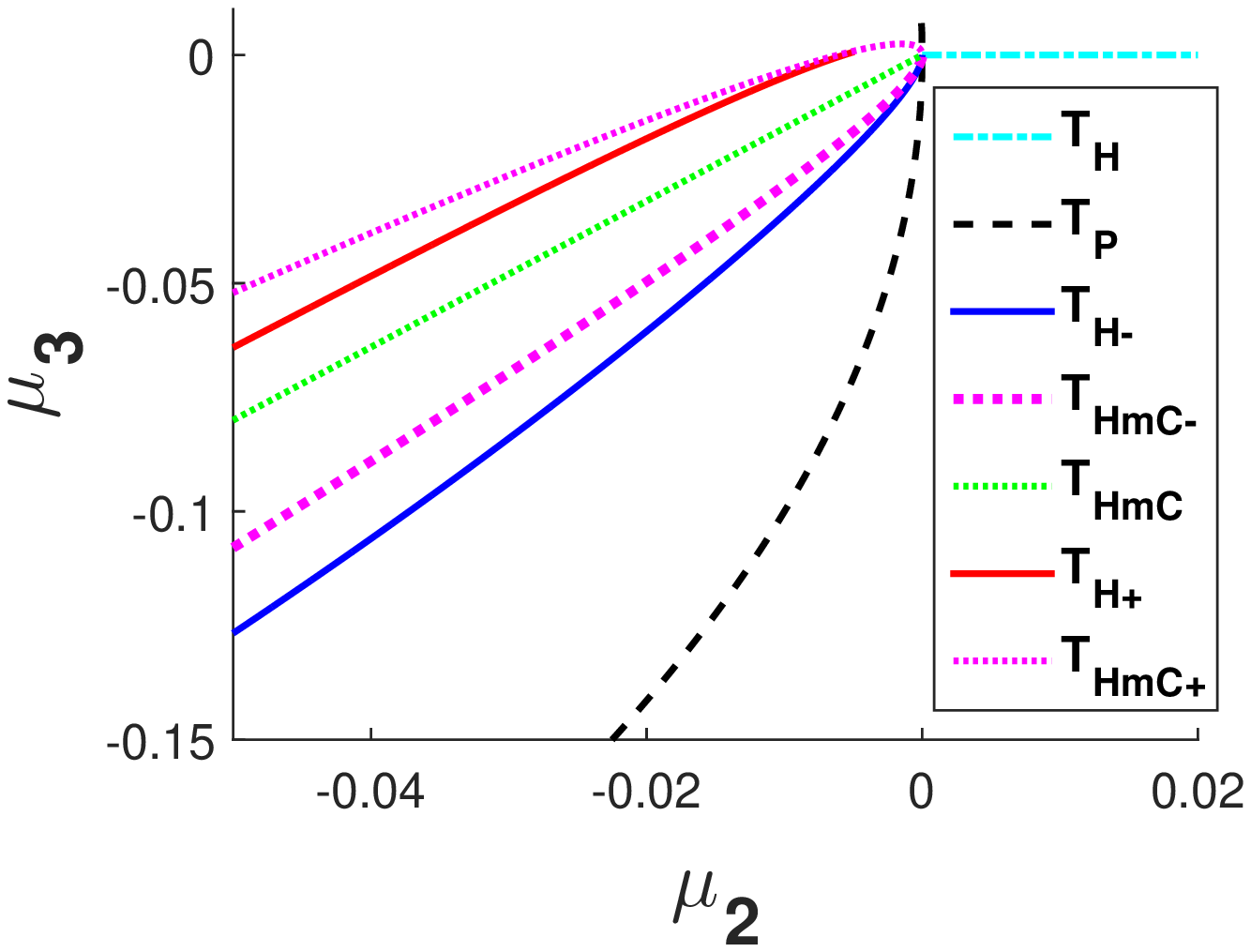}}
\subfigure[\(\mu_3=0.1,\) \(a_1=- b_0=1\)]
{\includegraphics[width=.24\columnwidth,height=.19\columnwidth]{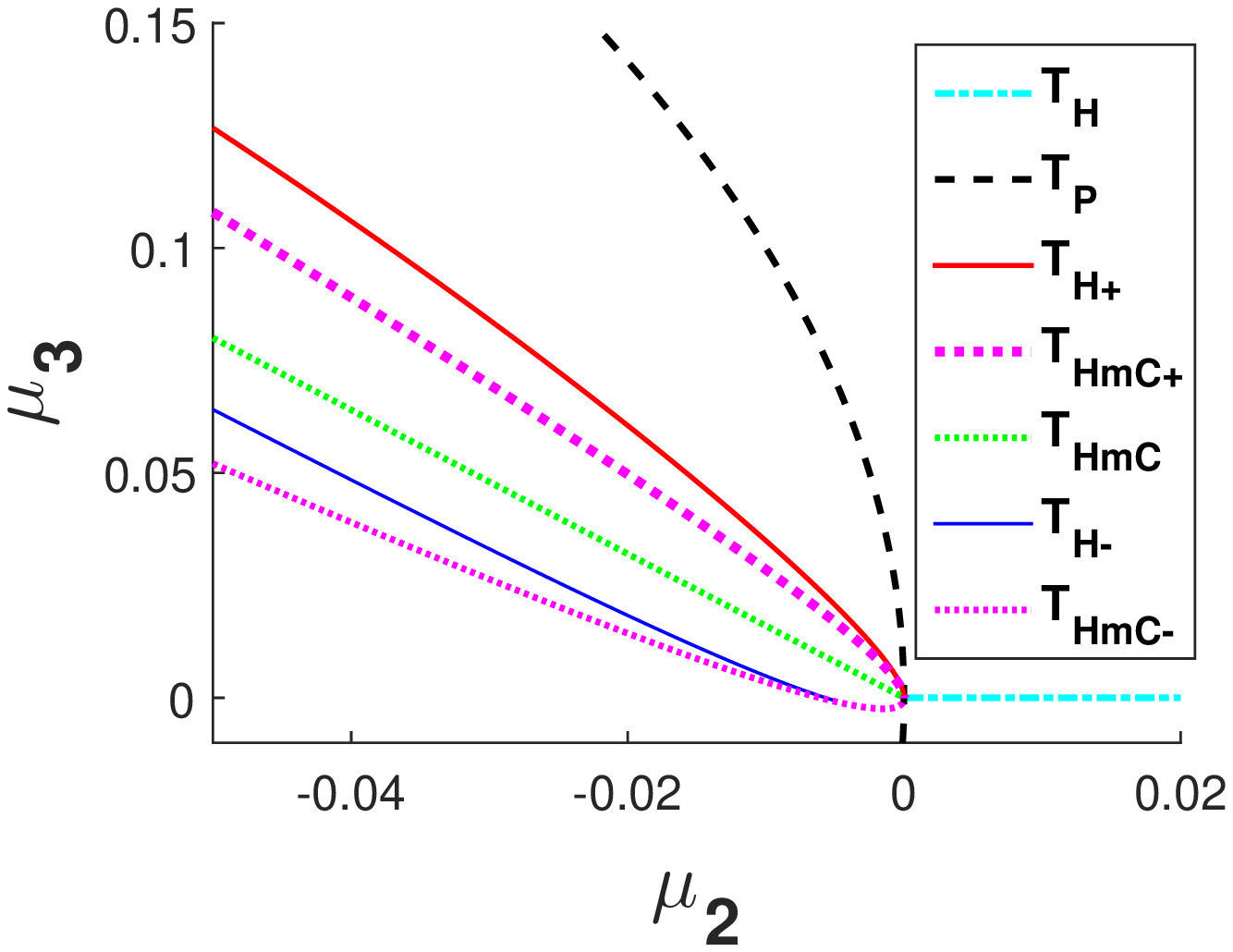}}
\subfigure[\(\mu_3=-0.1,\) \(a_1= -b_0=1\)]
{\includegraphics[width=.24\columnwidth,height=.19\columnwidth]{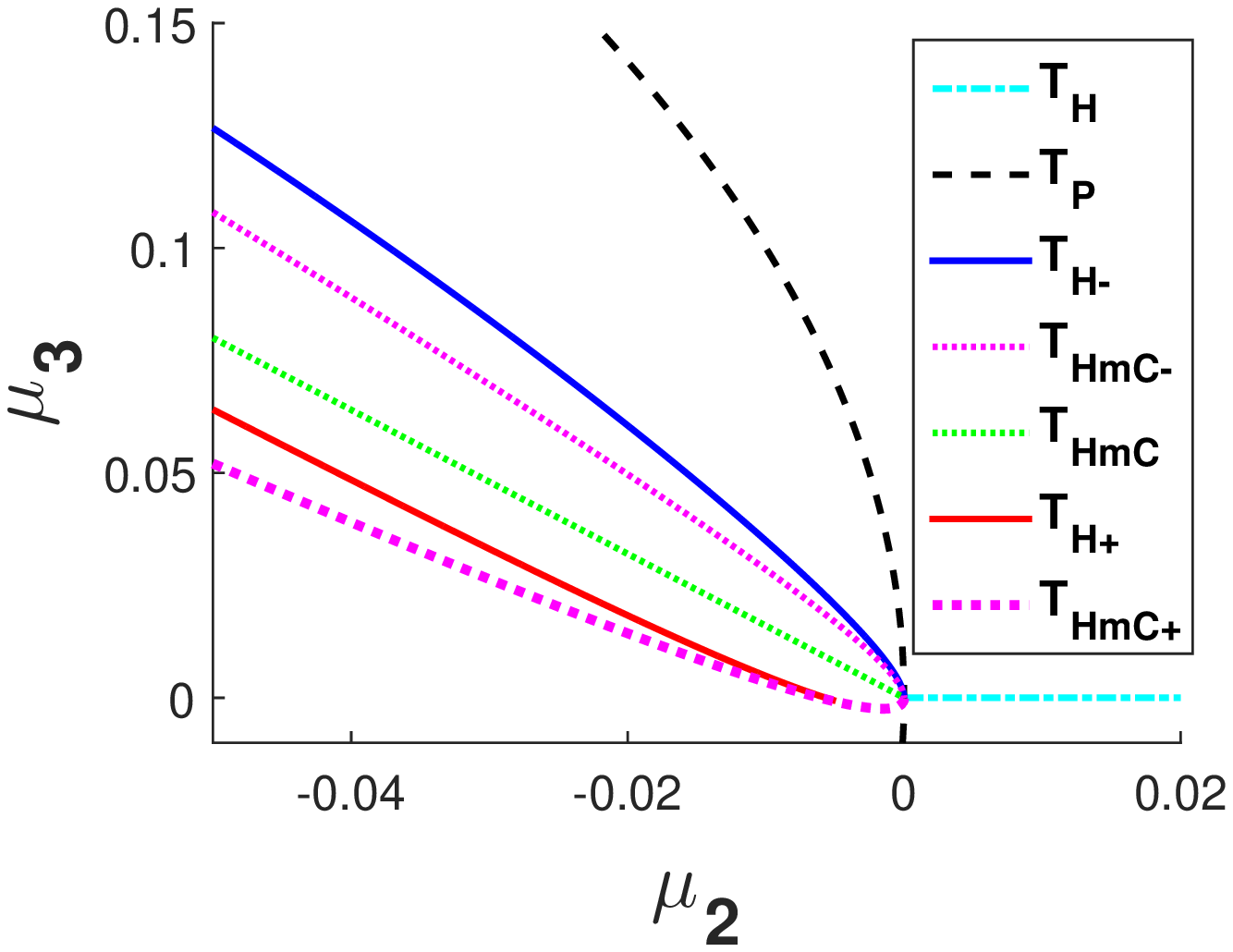}}
\end{center}
\vspace{-0.250 in}
\caption{Estimated critical controller sets for system \eqref{Eq03} when \(\mu_0:=0.\)}\label{nu4pnu10}
\vspace{-0.100 in}
\end{figure}

Figures \ref{nu4pnu10} demonstrate the estimated critical controller sets associated with system \eqref{Eq03} when \(\mu_3:=\pm 0.1\) for \(\mu_1= {\scalebox{0.7}{$\mathscr{O}$}}(||\mu_2,\mu_3||^2)\) and \(|\mu_0|\ll1.\) These figures include a pitchfork controller set \(T_P,\) where two equilibria \(E_{\pm}\) collide with  the origin. Hopf controller sets for the origin and \(E_{\pm}\) are denoted by \(T_H\), \(T_{H+}\) and \(T_{H-}\), according to Proposition \ref{Prop1} and Theorem \ref{thm3}. Each of the bifurcated limit cycles disappears when controller coefficients pass through the homoclinic controller sets \(T_{HmC}\) and \(T_{HmC\pm}\). More precisely, the limit cycles \(\mathscr{C}^1_\pm\) (bifurcated from \(E_\pm\)) grow in size and collide with the origin. These construct homoclinic cycles \(\Gamma_\pm\); see Figure \ref{fig23}. The next theorem deals with deriving the corresponding estimated controller sets \(T_{HmC\pm}\) and homoclinic orbits.

\begin{rem}
Simultaneous collisions of \(\mathscr{C}^1_+\) and \(\mathscr{C}^1_-\) with the origin give rise to a saddle-connection (double homoclinic). This is the only dynamics possibility for the equivariant cases; see \cite{GazorSadriSicon}. When \(\mu_0={\scalebox{0.7}{$\mathscr{O}$}}(|\mu_1|^2),\)
\bas
&T_{SC}:= \left\{(\mu_0,\mu_1, \mu_2,\mu_3)|\, \mu_2=\frac{8b_0}{5a_1}\mu_1+\sqrt{-\mu_1}{\scalebox{0.7}{$\mathscr{O}$}}(|\mu_1|, |\frac{{\mu_0}^2}{{\mu_1}^3}|)\right\}&
\eas is the estimated saddle-connection variety. This is a tuning controller coefficient manifold between (in the middle of) \(T_{HmC+}\) and \(T_{HmC-}\); see \cite[Lemma 5.6]{GazorSadriSicon} and equations \eqref{HmCpm}. The relative geometry of these manifolds determine their validity. For example, when the limit cycles \(\mathscr{C}^1_+\) and \(\mathscr{C}^1_-\) have already disappeared through homoclinic bifurcations, the saddle-connection variety is no longer valid.
\end{rem}

An efficient nonlinear time transformation method has been recently developed and applied for global bifurcation varieties of homoclinic and heteroclinic varieties of codimension two singularities \cite{AlgabaMelnikov,algabaSIADS,AlgabaMelnikov2020}. This is an efficient alternative approach to the classical use of Melnikov functions; \eg see \cite{Homburg,PerkoBook}. Both approaches have been usually applied using one-small scaling variable. Since all parameters are scaled using one parameter, the approach typically lead to a one-dimensional transition variety and fits well within a codimension-two singularity. Transition varieties must have a dimension of three in order that they would partition the parameter space in four dimensions. Although the scaling constants play a role in accommodating the higher dimensional transitions sets (\eg see \cite{AlgabaMelnikov3bogdanov}), we include three scaling parameters \(\epsilon_1, \epsilon_2, \epsilon_3.\) We derive an estimation for controller sets for homoclinic and heteroclinic bifurcations. Our symbolic estimations are accurate enough for many control engineering applications. Higher order approximations than our derived formulas are also feasible, but it is beyond the scope of this paper; \eg see \cite{AlgabaMelnikov,AlgabaMelnikov3bogdanov,algabaSIADS,AlgabaMelnikov2020} for highly accurate one- and two-dimensional transition varieties. Symbolic estimations for these bifurcations are useful for an efficient management of the nearby oscillating dynamics.

\begin{thm}[Homoclinic cycles \({\Gamma}_\pm\)]\label{Hom0}
When \(a_1>0\) and \(\mu_0={\scalebox{0.7}{$\mathscr{O}$}}(|\mu_1|^2),\) the bifurcated limit cycles disappear via two distinct quaternary homoclinic controller sets estimated by
\begin{equation}\label{HmCpm}
T_{HmC{{\pm}}}:= \left\{\mu \big|\, \mu_2=\frac{8b_0}{5a_1}\mu_1\pm \dfrac{9 \sqrt{2}\pi}{32}\mu_3\sqrt{-\mu_1}\mp \frac{9 \sqrt{2}\pi}{32}\frac{\mu_0}{\sqrt{-\mu_1}}+\sqrt{-\mu_1}\,{\scalebox{0.7}{$\mathscr{O}$}}\left(|\mu_1|, \left|\frac{{\mu_0}^2}{{\mu_1}^3}\right|\right) \right\},
\end{equation} for \(\mu= (\mu_0,\mu_1, \mu_2,\mu_3).\)
The leading estimated terms for the homoclinic cycles \(\Gamma_\pm\) give rise to an effective criteria for the magnitude control of the nearby oscillating dynamics. These are given by
\begin{eqnarray*}
&(x(\varphi), y(\varphi))=\left(-\frac{\sin^2(\varphi)\cos(\varphi)\sqrt{3-cos(2\varphi)}}{\sqrt{a_1}}\mu_1, \pm\frac{\sqrt{2}}{2}\sqrt{-\mu_1}(\cos(2\varphi)-1)\right)+({\scalebox{0.7}{$\mathscr{O}$}}(|\mu_1|^{\frac{3}{2}}), {\scalebox{0.7}{$\mathscr{O}$}}(|\mu_1|)),  &
\end{eqnarray*} for \(\varphi\in [0, \pi].\)
\end{thm}
\begin{proof} We apply a nonlinear time transformation method and include multiple scaling parameters \(\epsilon_i\) for \(i=1, 2, 3\); see
\cite{AlgabaMelnikov3bogdanov,algabaSIADS,AlgabaMelnikov}. Namely, we use the rescaling transformations \(x={\epsilon_1}^2\tilde{x}, y=\epsilon_1\tilde{y},\)
\begin{eqnarray}\nonumber
&
t={\epsilon_1}^{-1}\tilde{t}, \mu_0={\epsilon_1}^3\left(\gamma_1+ \gamma_{01}\epsilon_1+\gamma_{02}\epsilon_2\right), \mu_1={\epsilon_1}^2\left( \gamma_2+{\epsilon_1}\gamma_{11}+\epsilon_2 \gamma_{12}+\epsilon_3 \gamma_{13}\right),  &\\\label{Rescaling}
&\!\!\! \mu_2={\epsilon_1}^2\gamma_{21}+\epsilon_1\epsilon_2\gamma_{22}+{\epsilon_1}\epsilon_3\gamma_{23}+ \epsilon_1{\scalebox{0.7}{$\mathscr{O}$}}(|(\epsilon_1, \epsilon_2, \epsilon_3)|^2),
\mu_3={\epsilon_1}\gamma_{31}+\epsilon_2\gamma_{32}+\epsilon_3\gamma_{33}+{\epsilon_1}
\epsilon_2\gamma_{34}. \;
&
\end{eqnarray}
These transform the differential system \eqref{Eq03} into
\begin{eqnarray}\nonumber
\dot{\tilde{x}}&=&\gamma_1\left(\gamma_2\epsilon_1+\gamma_{02}\epsilon_2 \right)
+\gamma_2\left(1+\epsilon_1\gamma_{11}+\epsilon_2\gamma_{12}+\epsilon_3\gamma_{13}\right)\tilde{y}
+\left(\epsilon_1\gamma_{21}+\epsilon_2\gamma_{22}+\epsilon_3\gamma_{23}\right)\tilde{x}+a_1\tilde{y}^3\\\label{rescaled}
&&+\left(\epsilon_1\gamma_{31}+\epsilon_2\gamma_{32}+\epsilon_3\gamma_{33}
+\epsilon_1\epsilon_2\gamma_{34}\right)\tilde{x}\tilde{y}+\epsilon_1b_0 \tilde{x} \tilde{y}^2,\\\nonumber
\dot{\tilde{y}}&=&-\tilde{x}+\left(\epsilon_1\gamma_{21}+\epsilon_2\gamma_{22}
+\epsilon_3\gamma_{23}\right)\tilde{y}
+\left(\epsilon_1\gamma_{31}+\epsilon_2\gamma_{32}+\epsilon_3\gamma_{33}
+\epsilon_1\epsilon_2\gamma_{34}\right)\tilde{y}^2+\epsilon_1b_0\tilde{y}^3.
\end{eqnarray}
The unperturbed system, \ie when \(\epsilon=(\epsilon_1, \epsilon_2, \epsilon_3)=\0,\) is a Hamiltonian system with Hamiltonian \(H=\gamma_1\tilde{y}+ \frac{1}{2}\tilde{x}^2+ \frac{1}{2} \gamma_2\tilde{y}^2+ \frac{1}{4} a_1\tilde{y}^4.\) We further Taylor-expand the new state variables and a time-rescaling transformation in terms of the scaling parameters \(\epsilon_i\) for \(i=1, 2, 3\) as
\begin{eqnarray}\nonumber
&\tilde{x}(\varphi):=\tilde{x}_0(\varphi)+\sum {\epsilon_{j}}^i \tilde{x}_{ij}(\varphi), \quad \tilde{y}(\varphi):=\tilde{y}_0(\varphi)+\sum {\epsilon_{j}}^i\left(p_{ij} \cos(2 \varphi)+q_{ij}\right),&\\\label{EQ1}& \tilde{t}=\Phi \tau, \quad\Phi:=\phi_0+\sum{\epsilon_{j}}^i \phi_{ij},&
\end{eqnarray} where the sum \(\sum\) without indices stands for the double sum \(\sum_{i=1}^{\infty}\sum_{j=1}^{3}\) and \(\varphi\in [0, \pi].\) Let
\(\gamma_1:=0,\) \(\gamma_{01}:=0,\) \(\gamma_{02}:=0,\) \(\gamma_2:=-1,\) \(\gamma_{34}:=1.\) Then, Hamiltonian of the unperturbed system holds a homoclinic cycle that connects the stable and unstable manifolds of the origin, \ie the homoclinic orbit follows \(H(\tilde{x}, \tilde{y})=0\). When the rescaling variables \(\epsilon_i\)  for \(i=1, 2, 3\) becomes non-zero, the homoclinic cycle still holds for a homoclinic variety of codimension-one in the parameter space. The idea of the nonlinear time transformation method is to iteratively calculate the homoclinic cycle and homoclinic variety in terms of powers of \(\epsilon_i.\) We here only deal with zero and first order approximations, \ie \((p_0, q_0, x_0, \phi_0)\) and \((p_{1j}, q_{1j}, x_{1j}, \phi_{1, j})\) for \(j=1,2,3\). We remark that there is only a homoclinic cycle for system \eqref{rescaled}. However, this will turn out to be two homoclinic cycles \(\Gamma_\pm\) for \eqref{Eq03}, depending on the sign of \(\epsilon_1\) in \eqref{values}. The zero order approximation is given by \((\tilde{x}_0(\varphi), \tilde{y}_0(\varphi)),\) where we assume that
\ba\label{2.20}
&\tilde{y}_0:=p_0\cos(2\varphi)+q_0\quad\hbox{ and }\quad \tilde{x}_0(0)= \tilde{x}_0(\frac{\pi}{2})=0.&
\ea Hence, \((\tilde{y}_0(0), \tilde{y}_0(\frac{\pi}{2}))=(p_0+q_0, q_0-p_0).\) Since Hamiltonian is constant over the homoclinic cycle, we have \(H(\tilde{x}_0(\frac{\pi}{2}), \tilde{y}_0(\frac{\pi}{2}))= H(0, p_0+q_0)\). Furthermore, \(\frac{\partial H}{\partial \tilde{y}}(0, p_0+q_0)=0\) due to the fact that \((\tilde{x}_0(0), \tilde{y}_0(0))\) is an equilibrium for the unperturbed Hamiltonian system. These equations give rise to
\begin{eqnarray*}
&p_0= \frac{\sqrt{2}}{2\sqrt{a_1}}, \quad q_0= -\frac{\sqrt{2}}{2\sqrt{a_1}},\qquad \tilde{y}_0= \frac{\sqrt{2}}{2\sqrt{a_1}}\cos(2\varphi)-\frac{\sqrt{2}}{2\sqrt{a_1}}, &\\& \tilde{x}_0=\pm\frac{\sin^2(\varphi)\cos(\varphi)\sqrt{3-cos(2\varphi)}}{\sqrt{a_1}}, \quad \hbox{ and } \quad\phi_0(\varphi):= -\frac{\tilde{x}_0}{{\tilde{y}_0}'}.&
\end{eqnarray*}
Let \(q_{1j}:=0\) for \(j=1, 2, 3.\) Then, \(\tilde{y}_{11}= p_{11}\cos(2\varphi),\) \(\tilde{y}_{12}= p_{12}\cos(2\varphi),\) \(\tilde{y}_{13}= p_{13}\cos(2\varphi)\) and the first-order approximation follows
\begin{eqnarray}\label{gamma}
&\tilde{x} = \tilde{x}_0+\epsilon_1 \tilde{x}_{11}+\epsilon_2 \tilde{x}_{12}+\epsilon_3 \tilde{x}_{13}, \;
\tilde{y} = \tilde{y}_0+\epsilon_1 p_{11}\cos(2\varphi)+\epsilon_2 p_{12}\cos(2\varphi)+\epsilon_3 p_{13}\cos(2\varphi).\quad&
\end{eqnarray} Next, the terms of the first-order in terms of \(\epsilon_i\) for \(i=1, 2, 3\) in \(\phi\dot{x}\) give rise to
\begin{eqnarray}\nonumber
&\!\!\!\!\!\phi_0 x_{11}'+p_{11}\cos{(2\varphi)}-x_0\gamma_{21}-x_0y_0\gamma_{31}-y_0\gamma_{11}
-\gamma_{01}-3a_1{y_0}^2 p_{11}\cos{(2\varphi)}+\phi_{11} x_0'-b_0 x_0  {y_0}^2=0,&\\\nonumber
&\phi_0 x_{1i}'+p_{1i}\cos{(2\varphi)}-x_0\gamma_{3(i+1)}-x_0y_0\gamma_{4(i+1)}-y_0\gamma_{2(i+1)}&\\\label{phi}&
-\gamma_{1(i+1)}-3a_1{y_0}^2 p_{1i}\cos{(2\varphi)}+\phi_{1i} x_0'=0, \quad&
\end{eqnarray} see also \cite[Equation 2.22a and 2.22b]{algabaSIADS}.
Now consider the first-order \(\epsilon_i\)-terms in \(\phi\dot{y}\) along with equations \eqref{phi}. By eliminating \(\phi_{1j}\)-terms from these equations, an integrating factor and an integration, similar to the proof of \cite[Equation 2.30]{algabaSIADS}, we derive
\begin{eqnarray}\nonumber
&\int^{\varphi}_{0} y_0'\big(y_0\gamma_{11}-p_{12}\cos{(2\varphi)}+x_0\gamma_{21}
+y_0x_0\gamma_{31}+b_0x_0{y_0}^2+\gamma_{01}+3a_1{y_0}^2 p_{12}\cos{(2\varphi)}\big)d\varphi&\\\label{intHom}
&+x_0x_{11}+y_{11}g(y_0)+\int^{\varphi}_{0}x_0'\left({y_0}^2\gamma_{31}+\gamma_{21}y_0+b_0{y_0}^3\right)d\varphi=0,&\\\nonumber
&\int^{\varphi}_{0} y_0'\big(y_0\gamma_{2(i+1)}-p_{1i}\cos{(2\varphi)}+x_0\gamma_{3(i+1)}
+y_0x_0\gamma_{4(i+1)}+\gamma_{1(i+1)}+3a_1{y_0}^2 p_{1i}\cos(2\varphi)\big)d\varphi&\\\nonumber
&\qquad\qquad\;\;+x_0x_{1i}+y_{1i}g(y_0)+\int^{\varphi}_{0} x_0'\left({y_0}^2\gamma_{4(i+1)}+\gamma_{3(i+1)}y_0\right)d\varphi=0, \hbox{ for } i=2, 3.&
\end{eqnarray}
Evaluating equation \eqref{phi} at \(\varphi=\pi\) and \eqref{intHom} at \(\varphi=\pi, \pi/2\), we obtain nine number of linear equations. These give rise to the scaling parameters
\begin{eqnarray}\label{values}
&\gamma_{21}=-\frac{8b_0}{5a_1}+\frac{9 \sqrt{2}\pi}{32}\gamma_{31},\quad \gamma_{22}=\frac{9 \sqrt{2}\pi}{32}\gamma_{32},\quad \gamma_{23}=\frac{9 \sqrt{2}\pi}{32}\gamma_{33},&\\\nonumber& \gamma_{11}=0,\quad \gamma_{12}=0,\quad \gamma_{13}=0, \quad\epsilon_1:= \pm\sqrt{-\mu_1}.&
\end{eqnarray} Finally, we substitute these into the equation for \(\mu_2\) in \eqref{Rescaling} and derive transition sets \(T_{HmC{{\pm}}}\) given in equation \eqref{HmCpm}.
\end{proof}
Theorem \ref{thm3} implies that control coefficients \((\mu_0, \mu_1, \mu_2, \mu_3)\) for \(|\mu_0|= {\scalebox{0.7}{$\mathscr{O}$}}(||(\mu_1, \mu_2, \mu_3)||^4),\) \(|\mu_3|= {\scalebox{0.7}{$\mathscr{O}$}}(||\mu_1,\mu_2^2||)\) are not enough for fully unfolding a Bautin bifurcation around \(E_{\pm}.\) Now we show that system \eqref{Eq03} undergoes full Bautin bifurcation scenarios (in particular, bifurcations of two limit cycles and saddle-node bifurcation of limit cycles) when these restrictions on control coefficients are removed. For the distinction of different bifurcation scenarios, we will instead denote the equilibria with \(E^*_{\pm}\) in Lemma \ref{lem3.6}, Theorem \ref{THmBautin} and Remark \ref{rem3.8}.

\begin{lem}[Critical controller sets and normalized amplitude equation for generalised Hopf singularity]\label{lem3.6} Let \(9{\mu_3}^2\geq 32b_0 \mu_2,\) \(\delta:=\sqrt{9{\mu_3}^2-32b_0 \mu_2},\) and \(\zeta_\pm\) be given by
\begin{small}
\ba\label{Th}
&\zeta_\pm=\mu_0-\frac{\mu_3 (7{\mu_2}^2+12\mu_1)}{32 b_0}+\frac{9\mu_2\mu_3 (3{\mu_3}^2+16a_1)}{256 {b_0}^2}-\frac{27 {\mu_3}^3 ({\mu_3}^2+16a_1)}{2048{b_0}^3}&\\\nonumber
&\qquad\quad\pm \frac{\delta \left( 9{\mu_3}^4-56 b_0 \mu_2 {\mu_3}^2+64{b_0}^2{\mu_2}^2+144 a_1 {\mu_3}^2-128 a_1b_0 \mu_2+256{b_0}^2\mu_1\right)}{2048 {b_0}^3}.&
\ea
\end{small}
Further, let \(\kappa:=432a_1{\mu_3}^2-3\mu_3\delta \left(48 a_1+16 b_0 \mu_2-3{\mu_3}^2 \right)-768a_1b_0\mu_2+512{b_0}^2\mu_1+192b_0\mu_2{\mu_3}^2
-384{b_0}^2{\mu_2}^2-27{\mu_3}^4\) and \(\zeta={\scalebox{0.7}{$\mathscr{O}$}}(\kappa).\) Then,
\begin{enumerate}
  \item[a.] There are two generalized Hopf singularities for \(E^*_{\pm}\) at \(\zeta_\pm=0.\)
  \item[b.] These generalized Hopf singularities are determined by the cubic jet of the system \eqref{Eq03}.
  \item[c.] The normalized amplitude equation of the cubic truncated system is given by
\begin{eqnarray}\label{HopfNormalForm}
&\dot{\rho}_{\pm}= \pm256\,\delta \zeta_\pm \rho+ \frac{\kappa}{2{b_0}} \rho^3- 64 a_1{b_0} \rho^5+{\scalebox{0.7}{$\mathscr{O}$}}(\rho^7).&
\end{eqnarray}
\end{enumerate}
\end{lem}
\begin{proof} Since we are dealing with a Bautin bifurcation, we initially consider the truncated normal form system \eqref{Eq03} up to degree five, that is,
\begin{eqnarray}\label{withb3}
\dot{x}= \mu_0+\mu_1 y+\mu_2 x+\mu_3 x y+a_0 y^3+b_0 x y^2+b_1 x y^4,\,
\dot{y}= -x+\mu_2 y+\mu_3 y^2+b_0 y^3+b_1 y^5.
\end{eqnarray} We can show that the corresponding second Lyapunov coefficients is estimated by
\begin{eqnarray*}\label{Hophb3}
&\frac{3}{256}b_0 b_1 \kappa^2-64 a_0 b_0-15a_0b_0 b_1\delta^2+\frac{53}{16}b_0\delta^2.&
\end{eqnarray*} For all values of \(b_1\) and sufficiently small choices for the parameters, we can ensure that the second Lyapunov coefficient is always non-zero. In other words,  Bautin bifurcation here is completely determined by the cubic terms of the differential system \eqref{Eq03}.  We can thus truncate the normal form system \eqref{Eq03} at degree three, \ie let \(b_1:=0\). Now recall \(d_1\) and \(d_2\) from the first column of Routh table in equations \eqref{32}. We solve \(\dot{y}=0\) for \(x\) and substitute it into \(\dot{x}=0\) to obtain
\begin{eqnarray}\label{33}
\mu_0 +\mu_1 y+{\mu_2}^2 y+2\mu_2 \mu_3 y^2+2b_0 \mu_2 y^3+{\mu_3}^2y^3+2b_0 \mu_3 y^4+a_1 y^3+{b_0}^2 y^5=0.
\end{eqnarray} Hopf bifurcation occurs when \(d_1 = 0\) and \(d_2>0.\) Thus, there are two local Hopf singularities at
\bas
&E^*_\pm:=\left(x^*_{\pm}, y^*_{\pm}\right), \quad\hbox{ for } \; y^*_{\pm}:=\frac{-3\mu_3\pm \delta}{8 b_0} \;\hbox{ and }\; x^*_{\pm}:=\mu_2 y^*_\pm+\mu_3 {y^*_\pm}^2+b_0 {y^*_\pm}^3.&
\eas Their associated controller sets follow \eqref{Th}. Using shift of coordinates \((x, y)-(x^*_\pm, y^*_\pm)\) on equation \eqref{Eq03} and transforming the linear part into Jordan canonical form, we obtain \(\dot{x}=-\frac{1}{32}\sqrt{\kappa} y\pm\frac{3}{8}\delta x^2-\frac{1}{8}\mu_3 x^2+ b_0 x^3\) and
\bas
&\dot{y}=\frac{1}{32}\sqrt{\kappa} x+\frac{64a_1b_0 x ^3-64b_0 \eta^{\pm}+9{\mu_3}^3x^2\pm 8b_0\delta \mu_2 x^2\mp 3\delta{\mu_3}^2 x^2-72a_1\mu_3 x^2-40x^2b_0\mu_2\mu_3- 24a_1\delta x^2}{2b_0\sqrt{\kappa}}&\\&
+\frac{1}{4}xy \left(4b_0 x+\mu_3\pm\delta\right).&
\eas A computer programming (\eg \cite{GazorMoazeni}) gives rise to the normalized equation \eqref{HopfNormalForm}.
\end{proof}

\begin{thm}[Bautin bifurcations from \(E^*_\pm\)]\label{THmBautin} Let \(9{\mu_3}^2\geq 32b_0 \mu_2\) and \(\zeta_\pm={\scalebox{0.7}{$\mathscr{O}$}}(\kappa).\) Then for \(a_1<0,\) there are one supercritical and one subcritical Hopf controller sets estimated by
\ba\label{PrimaryHopfs}
&T^{Sup}_{H\pm}:=\left\{(\mu_0,\mu_1,\mu_2,\mu_3)\,|\, \zeta_+=0, b_0<0\right\}\; \hbox{ and }\nonumber&\\& \;T^{Sub}_{H\pm}:=\left\{(\mu_0,\mu_1,\mu_2,\mu_3)\,|\,\zeta_-=0, b_0>0\right\}.&
\ea When controller coefficients cross \(T^{Sup_1}_{H+}\) given by \eqref{Th} and \(\zeta_+={\scalebox{0.7}{$\mathscr{O}$}}(\kappa),\) one stable limit cycle \(\mathscr{C}^1_+\) bifurcates from \(E^*_+.\) As for \(T^{Sub_1}_{H-}\) when \(\zeta_-= {\scalebox{0.7}{$\mathscr{O}$}}(\kappa),\) the bifurcation causes an unstable local limit cycle \(\mathscr{C}^1_-\) encircling \(E^*_-.\) Both of these limit cycles are considered as tertiary limit cycles. Two simultaneous limit cycles surrounding \(E^*_+\) (or \(E^*_-\)) do not appear when \(a_1<0\) and naturally, saddle-nodes of limit cycles does not occur in this case. For \(a_1>0\) and \(\zeta_\pm={\scalebox{0.7}{$\mathscr{O}$}}(\kappa),\) two subcritical and supercritical Hopf controller sets occur through the estimated manifolds
\begin{eqnarray}\label{Prima2+Hopfs}
&T^{Sub_1}_{H\pm}=\left\{(\mu_0,\mu_1,\mu_2,\mu_3) \big|\, 26214 a_1\delta\zeta_{\pm}-b_0\kappa^2=0,\; \pm b_0\zeta_{\pm}>0, b_0<0, \mbox{ and }\,  a_1>0 \right\},&\\\nonumber
&T^{Sup_1}_{H\pm}=\left\{(\mu_0,\mu_1,\mu_2,\mu_3) \big|\, 26214 a_1\delta\zeta_{\pm}-b_0\kappa^2=0,\;  \pm b_0\zeta_{\pm}>0, b_0>0, \mbox{ and }\,  a_1>0 \right\},&
\end{eqnarray} and
\begin{eqnarray}\label{SecondaryHopfs}
&T^{Sup_2}_{H\pm}=\left\{(\mu_0,\mu_1,\mu_2,\mu_3) \big|\, 26214 a_1\delta\zeta_{\pm}-b_0\kappa^2>0,\;  \zeta_{-}=0, b_0<0, \mbox{ and }\,  a_1>0 \right\},&\\\nonumber
&T^{Sub_2}_{H\pm}=\left\{(\mu_0,\mu_1,\mu_2,\mu_3) \big|\, 26214 a_1\delta\zeta_{\pm}-b_0\kappa^2>0,\; \zeta_{+}=0, b_0>0, \mbox{ and }\,  a_1>0 \right\}.&
\end{eqnarray} When controller coefficients are close to \(T^{Sub_1}_{H+}\) (\(T^{Sup_1}_{H-}\)) and \(a_1b_0\zeta_{+}>0\) (\(a_1b_0\zeta_{-}<0\)), we have only one tertiary limit cycle \(\mathscr{C}^1_+\) (\(\mathscr{C}^1_-\)) encircling \(E^*_+\) (\(E^*_-\), respectively). However, a second small limit cycle \(\mathscr{C}^2_+\) (\(\mathscr{C}^2_-\)) bifurcates from \(E^*_+\) (\(E^*_-\)) as soon as control coefficients cross \(T^{Sub_2}_{H+}\) (\(T^{Sup_2}_{H-}\)) and \(a_1b_0\zeta_{+}\) (\(a_1b_0\zeta_{-}\)) becomes negative (positive). Here, we have two pairs of limit cycles (\(\mathscr{C}^1_\pm,\) \(\mathscr{C}^2_\pm\)) surrounding \(E^*_\pm\), where \(\mathscr{C}^2_\pm\) lives inside \(\mathscr{C}^1_\pm\). Two more estimated critical controller sets follow
\be\label{SNLT}
T^\pm_{SNL}:= \{(\mu_0,\mu_1,\mu_2,\mu_3) \big|\, 26214 a_1\delta\zeta_{\pm}-b_0\kappa^2=0, \pm b_0\zeta_\pm<0, a_1>0\},
\ee where the two limit cycles (\(\mathscr{C}^1_\pm,\) \(\mathscr{C}^2_\pm\)) coalesce and disappear as a saddle-node bifurcation of limit cycles.
\end{thm}
\begin{proof}
Recall the 3-jet normal form amplitude equation \eqref{HopfNormalForm}. Let
\begin{eqnarray*}
&p_\pm(\rho):= A\rho^2+ B\rho+C_\pm, \;\hbox{ where }\; A:=- 64 a_1{b_0},\qquad B:=\frac{1}{2}b_0\kappa,\qquad C_\pm:= \pm256 \delta \zeta_\pm. &
\end{eqnarray*} Positive roots of \(p_\pm\) correspond with the limit cycles bifurcated from \(E^*_\pm.\)  We first remark that \(\delta\) is always non-negative. Since \(\kappa>0,\) \(\frac{-B}{A}=\frac{\kappa}{128 a_1}\) is always negative for \(a_1<0\) and at most one limit cycle can bifurcate. In the case of \(a_1<0\), we have one limit cycle for \(\frac{C_\pm}{A}=\mp\frac{4\delta \zeta_{\pm}}{a_1b_0}<0,\) and no positive root for \({ \emph{sign}}(\frac{C_\pm}{A})={ \emph{sign}}(\pm b_0\zeta_{\pm})>0.\) When the limit cycle exists, \ie \(\pm b_0\zeta_{\pm}<0,\) the limit cycle is asymptotically stable for \(\pm \zeta_\pm>0.\)
Therefore, \(\zeta_{\pm}=0\) is a critical controller set for the appearance of a limit cycle and it is supercritical when \(b_0<0.\) This critical controller set is subcritical for positive values of \(b_0\). These arguments justify \(T^{Sup}_{H\pm}\) and \(T^{Sub}_{H\pm}\) in \eqref{PrimaryHopfs}.

Let \(a_1>0.\) Hence,  \(\frac{-B}{A}>0.\) Assume that \(B^2-4AC_\pm=\pm\frac{b_0\left(102\times 257 a_1\delta\zeta_{\pm}-b_0\kappa^2\right)}{4}>0.\)
When \(\frac{C_\pm}{A}<0,\) the polynomial \(p_\pm\) has exactly one positive root while for \(\frac{C_\pm}{A}>0,\) we have two positive roots for \(p_\pm.\) Therefore, \({ \emph{sign}}(\frac{C_\pm}{A})= { \emph{sign}}(\mp b_0\zeta_{\pm})<0\) and \(B^2-4AC_\pm=0\) is a critical controller manifold where one limit cycle bifurcates from \(E^*_\pm.\) The bifurcated limit cycle is asymptotically stable when \(C_\pm>0\); \ie \(b_0>0.\) Thus, \(\pm b_0\zeta_{\pm}>0\) and \(B^2-4AC_\pm=0\) gives rise to a supercritical Hopf controller set \(T^{Sup_1}_{H\pm}\) in \eqref{Prima2+Hopfs}, where a small limit cycle bifurcates from \(E^*_\pm.\) Furthermore, \(\frac{C_\pm}{A}=0\) and \(B^2-4AC_\pm>0\) is another critical controller set where one small limit cycles is born in the interior of an already existed larger limit cycle. The new small bifurcated limit cycle is asymptotically stable when \(\mp b_0\zeta_{\pm}>0\) and \(\pm\zeta_\pm>0.\) This is equivalent with \(b_0<0\). These conditions determine the supercritical controller set \(T^{Sup_2}_{H\pm}\) in equation \eqref{SecondaryHopfs}. Thus, the argument for subcritical controller manifold \(T^{Sub_2}_{H\pm}\)  for \(b_0>0\) is similar. For \(\frac{C_\pm}{A}>0,\) we have two different cases:
\begin{enumerate}
  \item We have two positive roots for \(B^2-4AC_\pm>0.\)
  \item There is no limit cycle for \(B^2-4AC_\pm<0.\)
\end{enumerate}
Hence, when \(\pm b_0\zeta_\pm<0\) and \(a_1>0\) hold, \(B^2-4AC_\pm=0\) is a saddle-node controller set \(T^\pm_{SNL}\) of limit cycles in \eqref{SNLT}, where two limit cycles coalesce and disappear.
\end{proof}

\begin{figure}[t!]
\begin{center}
\subfigure[\( a_1= b_0=1, \mu_1=-0.02, \mu_2=-0.06\)\label{fig1}]
{\includegraphics[width=.4\columnwidth,height=.27\columnwidth]{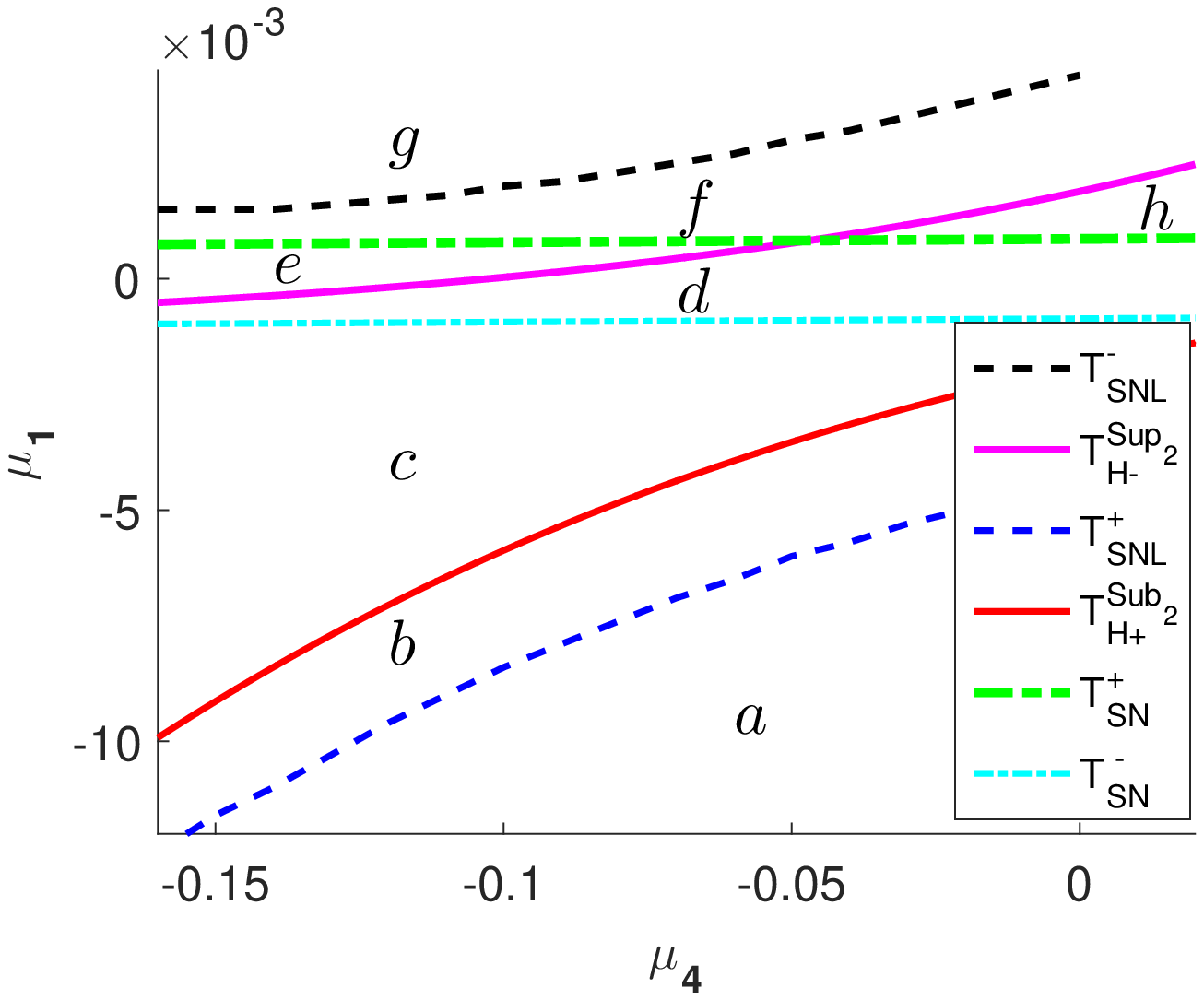}}\;
\subfigure[\( a_1= b_0=1, \mu_0=0.001, \mu_3=0.1\)\label{fig2}]
{\includegraphics[width=.4\columnwidth,height=.27\columnwidth]{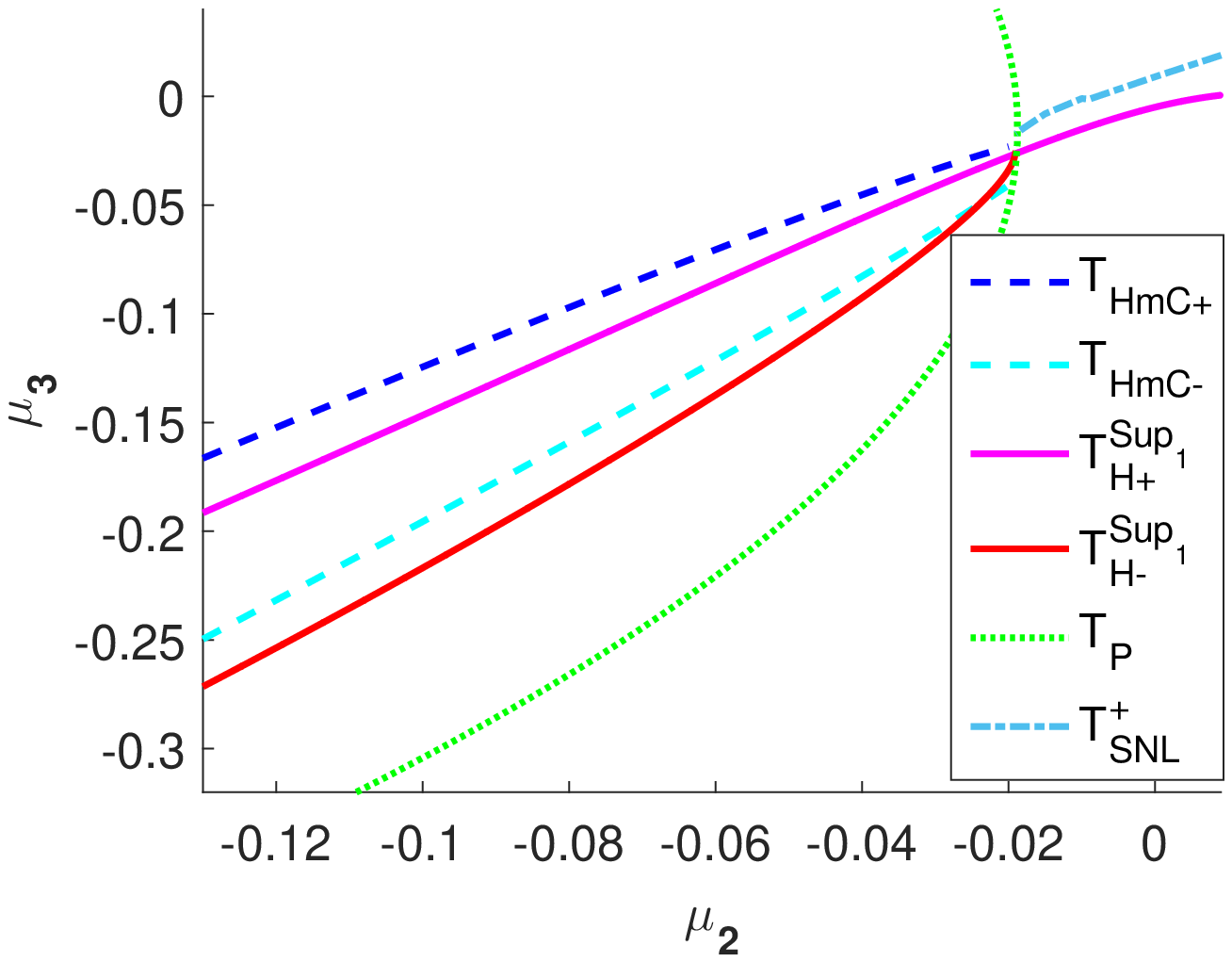}}\;
\subfigure[\(a_1=b_0=1,\) \(\mu_0=.001,\) \(\mu_1=-.1\) \(\mu_2=-.122,\) \(\mu_3=.1\)\label{fig23}]
{\includegraphics[width=.4\columnwidth,height=.27\columnwidth]{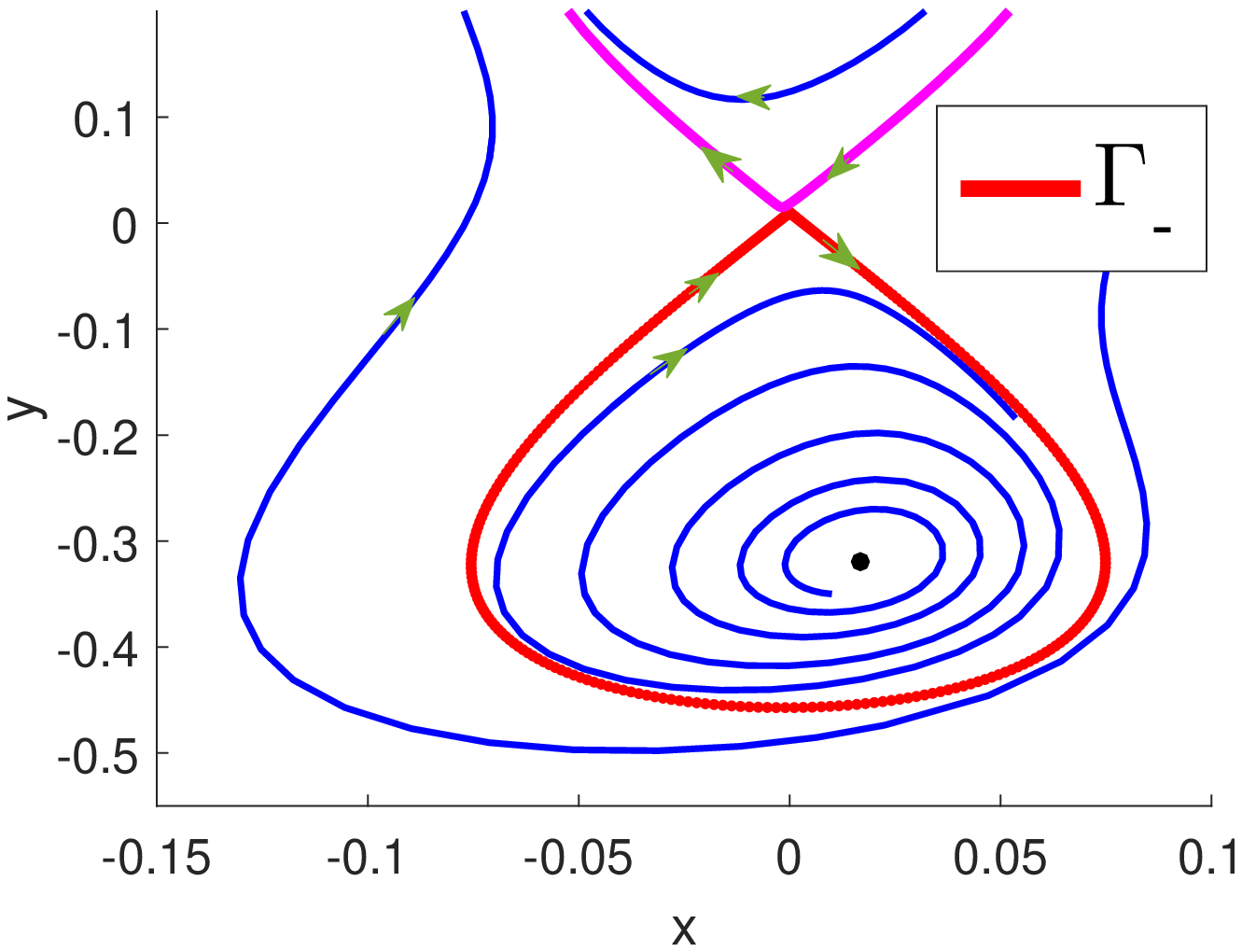}}\;
\subfigure[\(a_1= -b_0=-1,\) \(\mu_0=0,\) \(\mu_1=.01,\) \(\mu_2= .00394,\) \(\mu_3=.01\)\label{fig24}]
{\includegraphics[width=.4\columnwidth,height=.27\columnwidth]{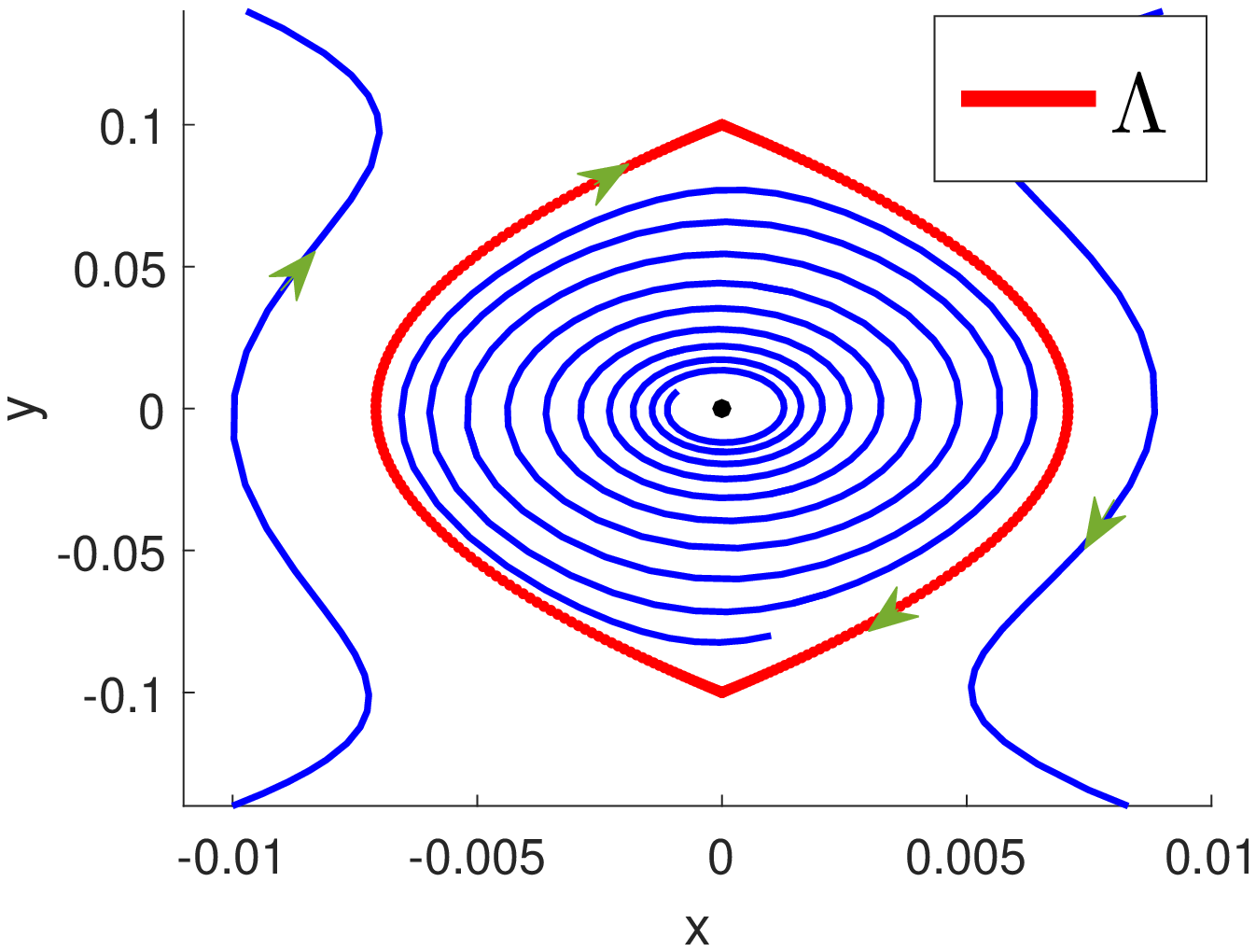}}
\end{center}
\vspace{-0.250 in}
\caption{Estimated controller sets \ref{fig1}-\ref{fig2} for \eqref{Eq03}.
Estimated \(\Gamma_-\) in \ref{fig23} and heteroclinic \(\Lambda\) in \ref{fig24}. }\label{nu4pnu1}
\vspace{-0.100 in}
\end{figure}
\begin{rem}[Basin of attraction for stabilization approach via bifurcated stable limit cycles]\label{rem3.8}
Bautin bifurcation described in Theorem \ref{THmBautin} includes supercritical and subcritical bifurcations of tertiary limit cycles \(\mathscr{C}^1_\pm\) and \(\mathscr{C}^2_\pm.\) They can be used to stabilize the system when the limit cycle is stable. However, it is important to notice about their basin of attractions. When \(\mathscr{C}^2_\pm\) is stable, the basin of attraction is the interior of \(\mathscr{C}^1_\pm\setminus\{E^*_\pm\}\). For example, Figure \ref{3Figb} illustrates the stable limit cycle \(\mathscr{C}^2_+.\) Its basin of attraction is the interior of \(\mathscr{C}^1_+\) (except \(E^*_+\)). The basin of attraction for the stable cases of \(\mathscr{C}^1_\pm\) includes the region encircled by \(\mathscr{C}^1_\pm\) and \(\mathscr{C}^2_\pm.\) However, the complete description for basin of attraction for the exterior of \(\mathscr{C}^1_\pm\) depends on the dynamics of the system. For instance, Figure \ref{Fig5d} demonstrates stable limit cycle \(\mathscr{C}^1_+\) living inside unstable limit cycle \(\mathscr{C}_0\). The basin of attraction for \(\mathscr{C}^1_+\) consists of the region between the stable manifolds (red and blue curves in Figure \ref{Fig5d}) of the origin (excluding the equilibria). This region includes the interior of \(\mathscr{C}^1_+.\)
\end{rem}

The limit cycle \(\mathscr{C}_0\) grows in size to collide with either of the secondary equilibria as controller coefficients go away from the corresponding Hopf controller set \eqref{SNr2s2nu10}, \ie \(\frac{\mu_2}{b_0}\) decreases. The limit cycle collides with either \(E_+\) or \(E_-\) giving rise to homoclinic cycles \(\Lambda_\pm\) or simultaneously collides with both of them. The latter leads to a heteroclinic cycle \(\Lambda\); see Figure \ref{fig24}. The following two theorems deal with heteroclinic and homoclinic bifurcations. The limit cycle \(\mathscr{C}_0\) may alternatively collide with the origin and disappear through a homoclinic bifurcation; \eg see Figure \ref{Fig5e} where limit cycle \(\mathscr{C}_0\) collides with the saddle and disappear as in Figure \ref{Fig5f}. Estimated controller set is then given by \(T_{HmC}:= \left\{(\mu_0,\mu_1, \mu_2,\mu_3)|\, \mu_2=\frac{8b_0}{5a_1}\mu_1\right\}.\) 

\begin{thm}[Heteroclinic cycle \(\Lambda\) when \(a_1 < 0\)]\label{HeterThm} Let \(|\mu_0|= {\scalebox{0.7}{$\mathscr{O}$}}(|\mu_1|^2)\) and \(a_1 < 0.\) Then, there is a heteroclinic cycle \(\Lambda.\) This connects the equilibrium \(E_+\) with the saddle \(E_-\). The corresponding heteroclinic bifurcation occurs at the heteroclinic controller manifold approximated by
\be\label{Hetero}
T_{HtC}:= \left\{(\mu_0,\mu_1, \mu_2,\mu_3)|\, \mu_2=\frac{2b_0}{5a_1} \mu_1+ \frac{9 }{16}\mu_3\sqrt{\mu_1}- \frac{9}{16}\frac{\mu_0}{\sqrt{\mu_1}}+ {\scalebox{0.7}{$\mathscr{O}$}}(||(\mu_1, \mu_3)||^\frac{3}{2})
\right\}.
\ee The most leading estimated terms for \(\Lambda\) are
\begin{eqnarray*}
&(x, y)=\left(\mu_1\frac{\sqrt{2}}{2}\sin(2\varphi)\sqrt{a_1\cos^2(2\varphi)+2+a_1}, \frac{\sqrt{\mu_1}}{\sqrt{-a_1}}\cos(2\varphi)\right)+({\scalebox{0.7}{$\mathscr{O}$}}(|\mu_1|^{\frac{3}{2}}), {\scalebox{0.7}{$\mathscr{O}$}}(|\mu_1|)). &
\end{eqnarray*}
\end{thm}

\begin{proof} Here, we use the rescaling transformations \eqref{Rescaling} and \eqref{EQ1} when \(\gamma_1:=0,\) \(\gamma_{02}:=0,\) \(\gamma_2:=1,\) \(\gamma_{11}:=\gamma_{12}:=\gamma_{13}:=0,\) and \(\gamma_{34}:=1.\) The unperturbed heteroclinic orbit connects the two saddles \(\big(0, \pm \frac{1}{\sqrt{-a_1}}\big).\) We assume that equations \eqref{2.20} hold.  Similar to \cite[Equation 2.11]{algabaSIADS}, we have
\begin{eqnarray*}
&y_0(0)=-\frac{1}{\sqrt{-a_1}}= p_0+q_0, y_0(\frac{\pi}{2})= \frac{1}{\sqrt{-a_1}}= q_0-p_0, \hbox{ and thus, }&\\
&\tilde{y}_0=p_0\cos(2\varphi)+q_0=\frac{1}{\sqrt{-a_1}}\cos(2\varphi), q_0=0.&
\end{eqnarray*} We compute the unperturbed heteroclinic orbit via \(H(\tilde{x}_0,\tilde{y}_0)= H(0, \pm (-a_1)^{\frac{-1}{2}})\) where
\(H(\tilde{x},\tilde{y})=\frac{1}{2}\tilde{x}^2-\frac{1}{2}\tilde{y}^2+\frac{1}{4}a_1\tilde{y}^4.\) Therefore, \(\tilde{x}_0=\frac{\sqrt{2}}{2}\sin(2\varphi)\sqrt{a_1\cos^2(2\varphi)+2+a_1}.\) The first-order terms for \(i=1, 2, 3\) in \(\phi\dot{x}\) follow equations \eqref{phi} and equations \eqref{intHom} hold. We need the first-order terms in \(\phi\dot{y}\) given by (see \cite[Equations 2.22a and 2.22b]{algabaSIADS})
\begin{eqnarray}\label{phiyd}
&\phi_0 y_{11}'-\gamma_{21} y_0-{y_0}^2\gamma_{31}-b_0 {y_0}^3+\phi_{11} y_0'=0, \phi_0 y_{1i}'-\gamma_{3(i+1)} y_0-{y_0}^2\gamma_{4(i+1)}+\phi_{1i}y_0'=0,&
\end{eqnarray} for \(i=2, 3.\) We evaluate equations \eqref{phi} and \eqref{phiyd} at \(\varphi=0, \frac{\pi}{2},\) while equations \eqref{intHom} are computed at \(\varphi= \pi/2.\) These lead to fifteen linear equations and \(\gamma_{21}=\frac{2b_0}{5a_1}+\frac{9}{16}\gamma_{31},\) \(\gamma_{32}=\frac{16 }{9}\gamma_{22},\) and \(\gamma_{33}=\frac{16 }{9}\gamma_{23}.\) Thus, transition varieties \eqref{Hetero} are derived by substitution of these values into the rescaling transformation for \(\mu_2.\)
\end{proof}

\begin{figure}[t!]
\begin{center}
\subfigure[There is only the source equilibrium \(E^*_+.\) \label{Fig2a}]
{\includegraphics[width=.24\columnwidth,height=.2\columnwidth]{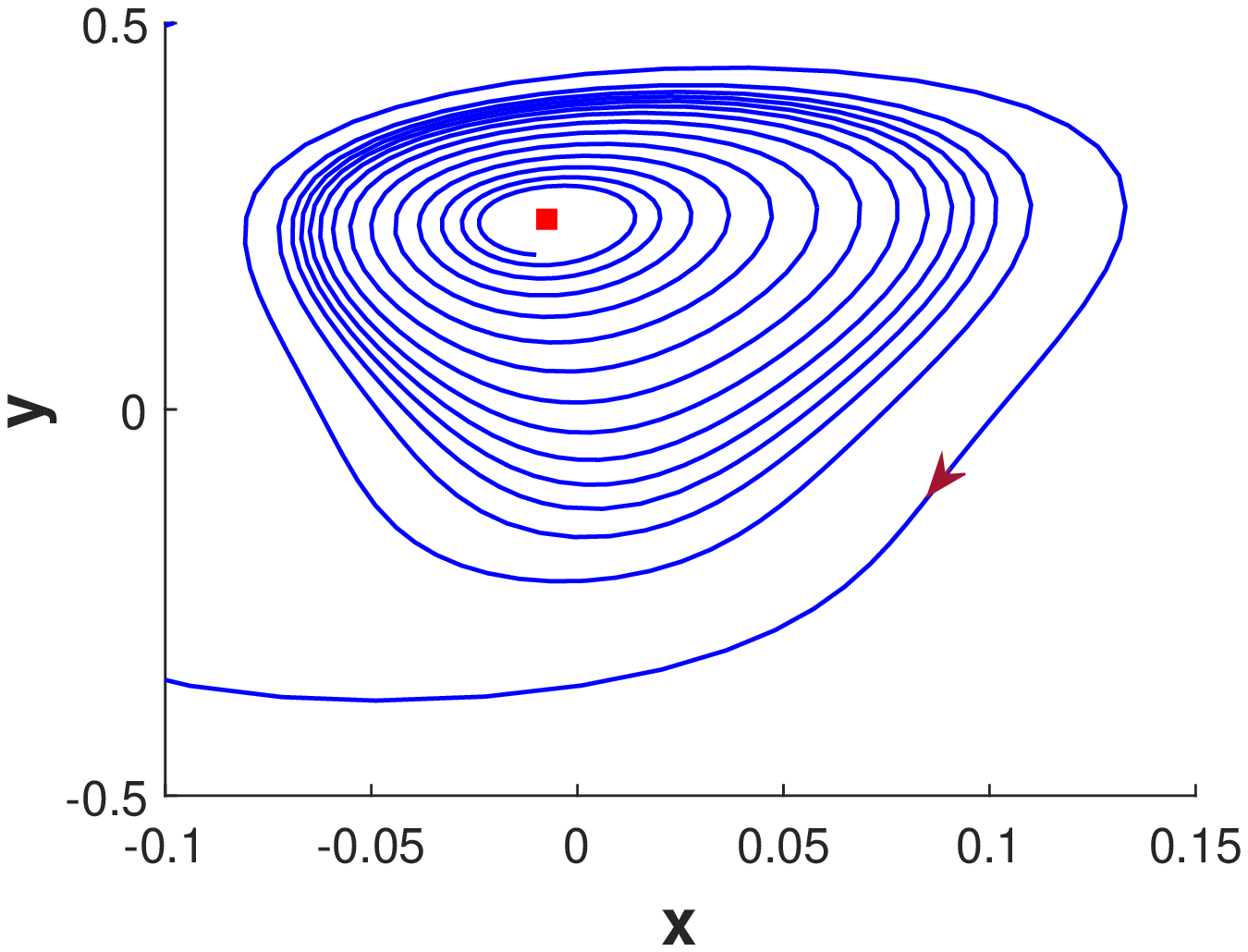}}\,
\subfigure[Stable limit cycle \(\mathscr{C}^2_+,\) unstable limit cycle \(\mathscr{C}^1_+,\) and a source equilibrium \(E^*_+\)\label{3Figb}]
{\includegraphics[width=.24\columnwidth,height=.2\columnwidth]{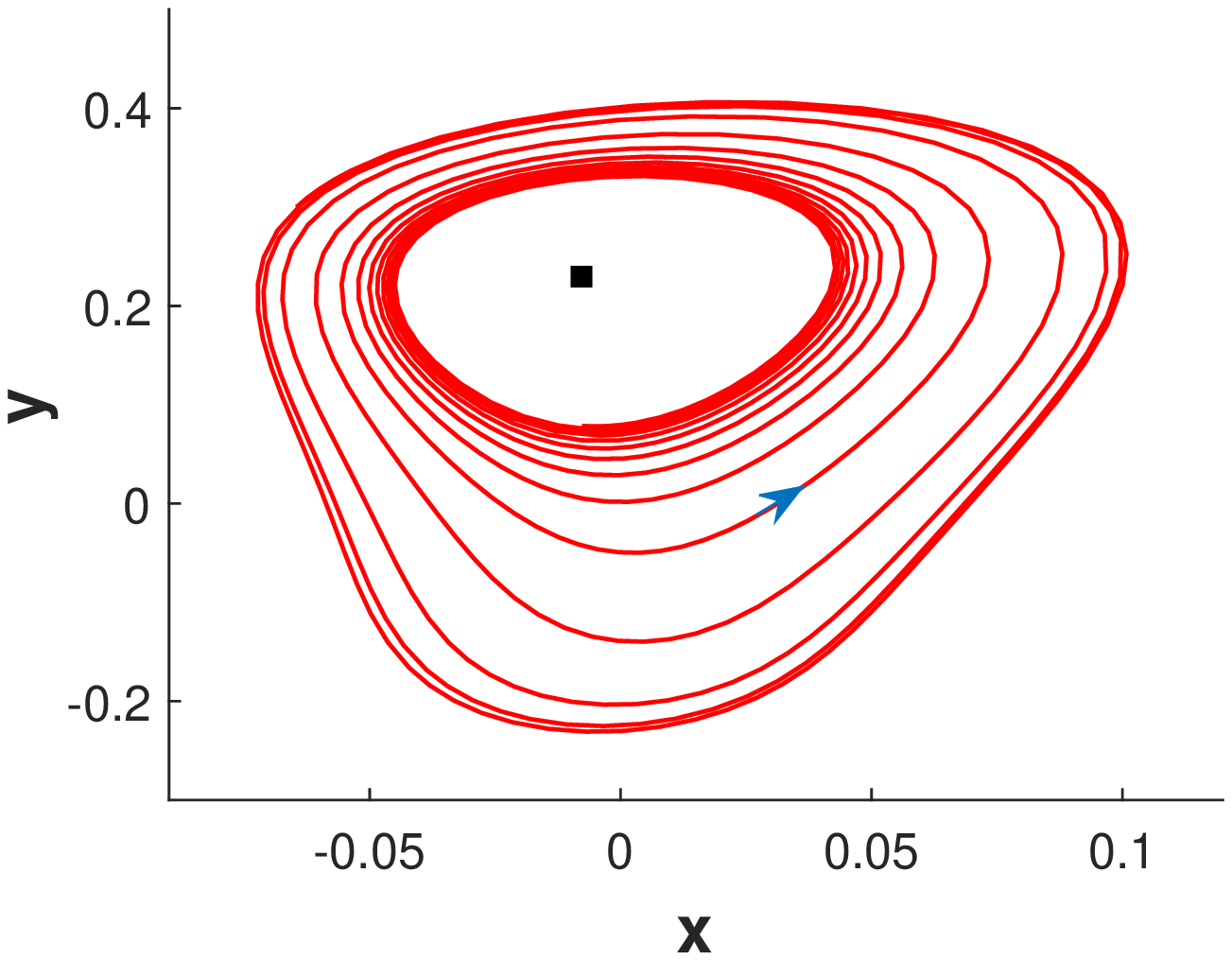}}\,
\subfigure[There is an unstable cycle \(\mathscr{C}^1_+\) and stable equilibrium \(E^*_+\).\label{3Figcc}]
{\includegraphics[width=.24\columnwidth,height=.2\columnwidth]{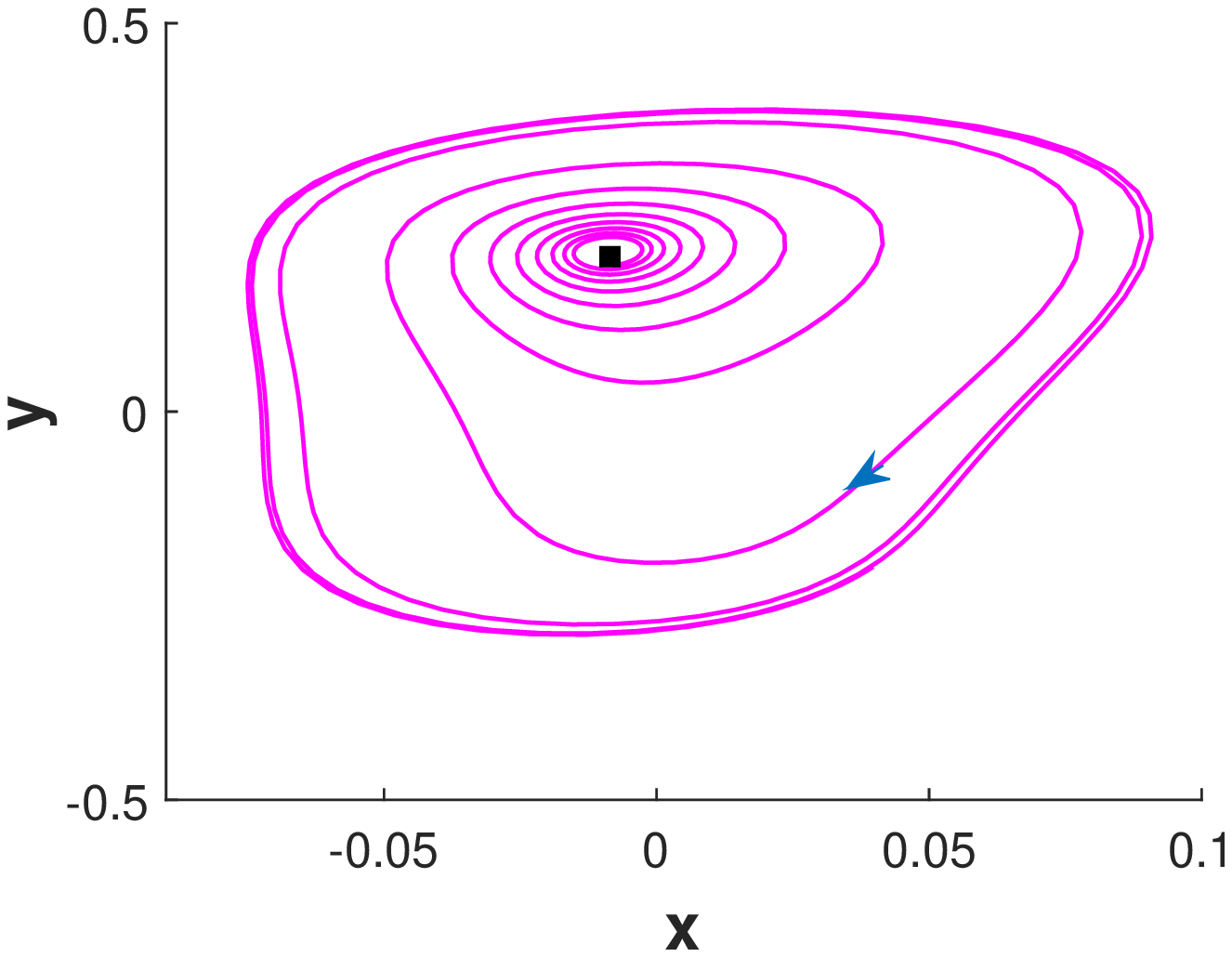}}\,
\subfigure[Unstable limit cycle \(\mathscr{C}^1_+\), spiral sinks \(E^*_\pm,\) and the primary saddle.]
{\includegraphics[width=.24\columnwidth,height=.2\columnwidth]{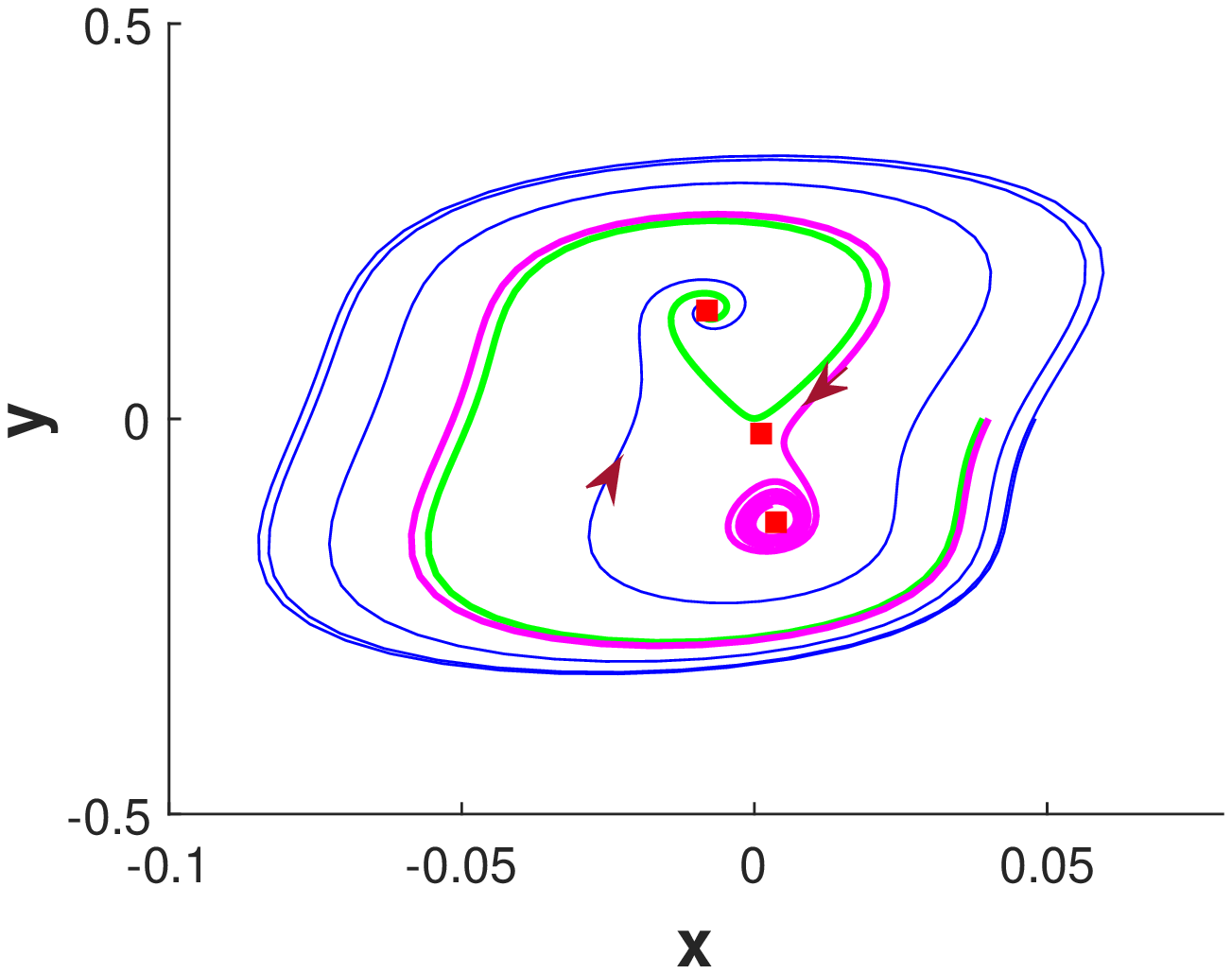}}
\subfigure[An unstable large limit cycle \(\mathscr{C}^1_+,\) a stable limit cycle \(\mathscr{C}^2_-,\)
primary saddle, sink \(E^*_+\) and source \(E^*_-.\)]
{\includegraphics[width=.24\columnwidth,height=.2\columnwidth]{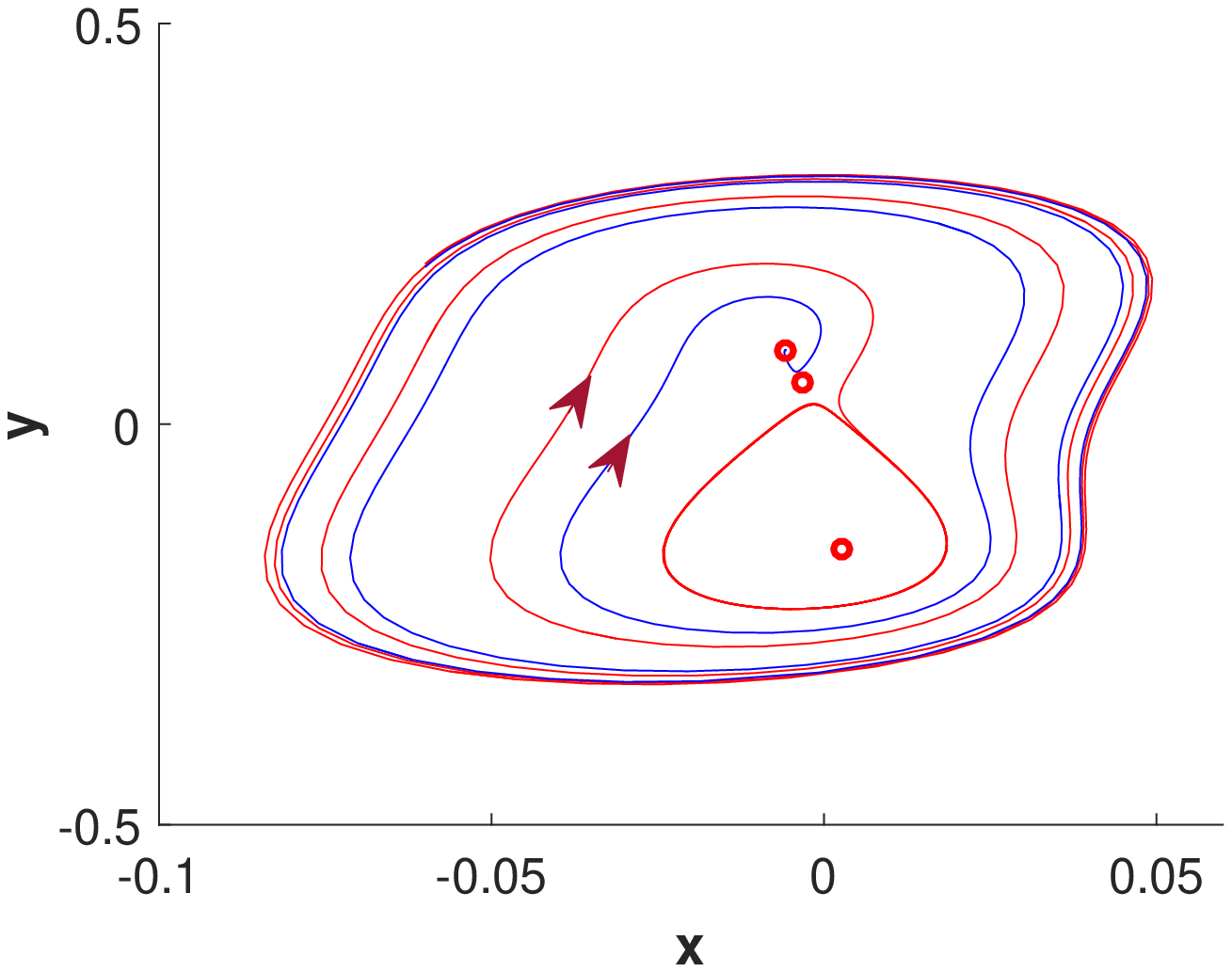}}\,
\subfigure[Two limit cycles \(\mathscr{C}^1_+\) and \(\mathscr{C}^2_-\) surround unstable \(E^*_-.\) Here, \(\mathscr{C}^1_+\) plays the role of \(\mathscr{C}^1_-\) due to disappearance of \(E^*_+.\) \label{3Figf}]
{\includegraphics[width=.24\columnwidth,height=.2\columnwidth]{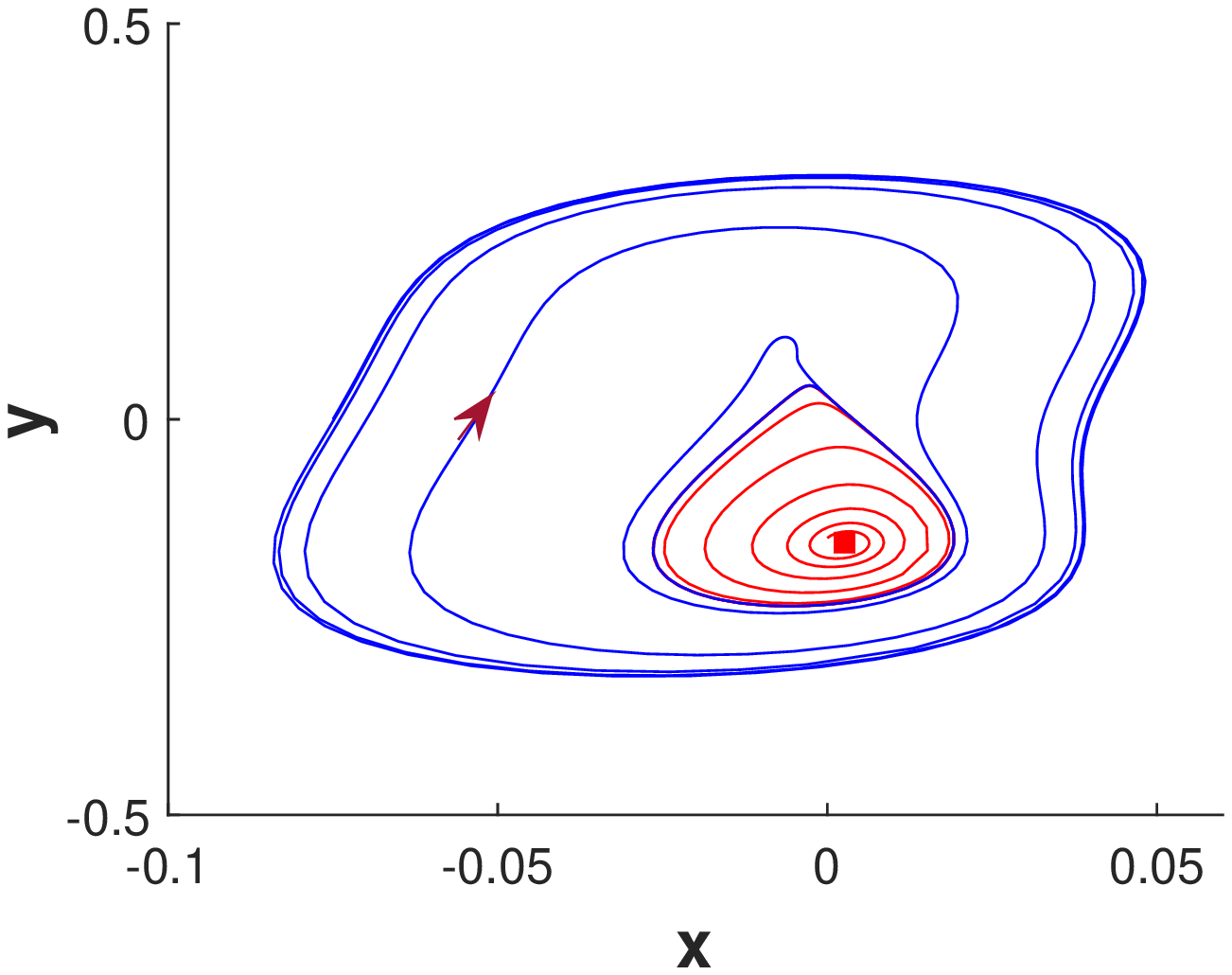}}\,
\subfigure[Asymptotically unstable equilibrium \(E^*_-.\)\label{Fig2h}]
{\includegraphics[width=.24\columnwidth,height=.2\columnwidth]{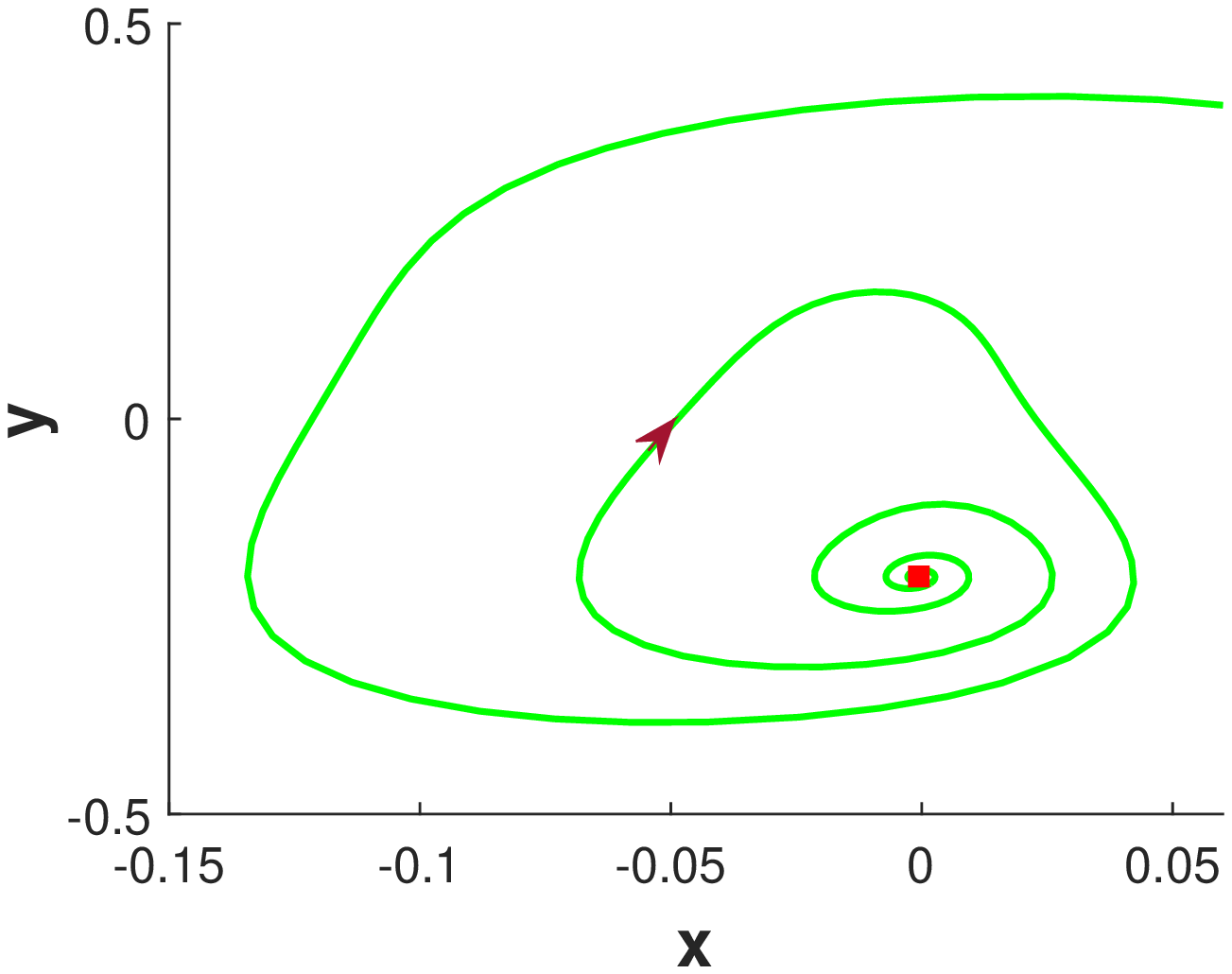}}\,
\subfigure[Spiral sink \(E^*_-\) and unstable limit cycle \(\mathscr{C}^1_+.\) Here, \(E^*_+\) disappeared at \(T^+_{SN}\) and
\(\mathscr{C}^1_+\) plays the role of \(\mathscr{C}^1_-.\)\label{Fig2j}]
{\includegraphics[width=.24\columnwidth,height=.2\columnwidth]{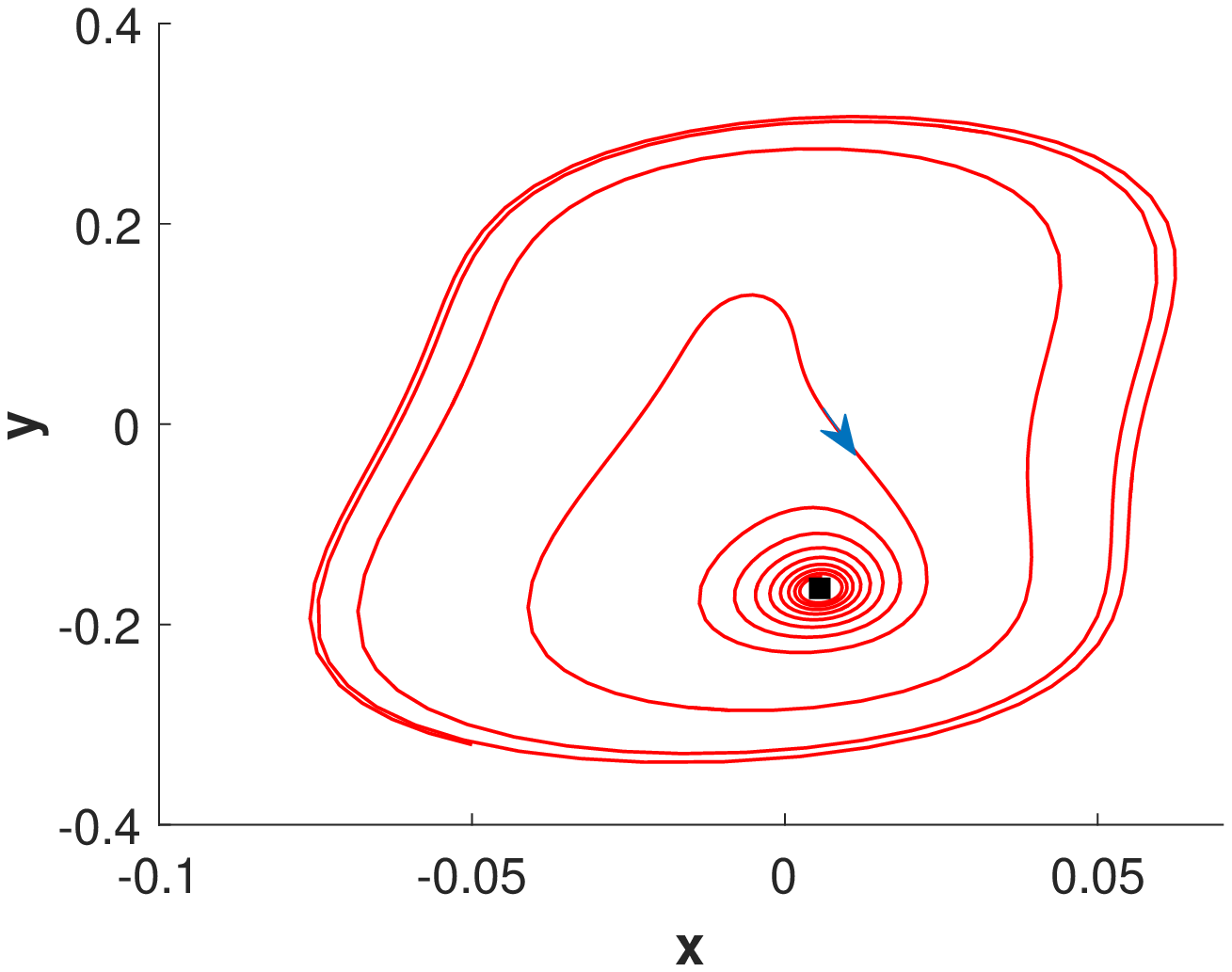}}
\end{center}
\vspace{-0.250 in}
\caption{Numerical phase portraits \ref{Fig2a}-\ref{Fig2h} are associated with regions \(a-h\) in \ref{fig1} and \eqref{Eq03}. \label{FIG3}}
\vspace{-0.100 in}
\end{figure}
\begin{thm}[Homoclinic controller manifolds for \(\Lambda_\pm\)]\label{HomE} Let \(a_1<0.\) There are two homoclinic cycles  \(\Lambda_\pm\) connecting the stable and unstable manifolds of \(E_\pm\) at homoclinic controller manifolds estimated by
\begin{footnotesize}
\begin{eqnarray}\label{pnH}
&T^{\pm}_{HmC}:=\left\{(\mu_0,\mu_1, \mu_2,\mu_3)|\, \mu_1= \frac{10^{\frac{2}{3}}}{\left(-a_1\right)^{\frac{5}{2}}}{\mu_0}^{\frac{2}{3}}-\frac{49.19204541 }{\left(-a_1\right)^{\frac{5}{2}}} b_0 \mu_0-\frac{8.203865604}{10^{\frac{-2}{3}}\left(-a_1\right)^{\frac{11}{6}}}\mu_2{\mu_0}^\frac{1}{3}+\frac{4.355675048}{10^{\frac{-2}{3}}\left(-a_1\right)^3} \mu_3{\mu_0}^{\frac{2}{3}}\right\}.&
\end{eqnarray}
\end{footnotesize} The homoclinic \(\Lambda_+\) occurs when \(\mu_0>0\) while \(\Lambda_-\) corresponds with negative values of \(\mu_0\). The leading estimated terms for \((x, y)\)-coordinates of the homoclinic cycles \(\Lambda_\pm\) are
\begin{eqnarray*}
&\mp0.7157063998\cos(2\varphi)\mp 0.2299428741 \qquad\qquad\hbox{ and }&\\
& \mp 0.3622053022\sqrt{2}\sin(\varphi)\sin(2\varphi)\sqrt{cos({2\varphi})+2.28512403471548}, \quad\hbox{ for } \varphi\in [0, \pi],&
\end{eqnarray*} respectively. This is useful for the management of its nearby oscillating dynamics.
\end{thm}
\begin{proof} We first use transformations \(x=\left(-a_1\right)^{\frac{3}{2}} \hat{x},\) \(y=\left( -a_1\right)^{\frac{1}{2}} y, \) and time rescaling \(t=-\frac{1}{a_1} \tau\) to change the coefficient \(a_1\) to \(-1\). 
Then, we have
\begin{eqnarray}
&\dot{\hat{x}}=\left(-a_1\right)^{\frac{-5}{2}}\mu_0+\left(-a_1\right)^{-2}\mu_1\hat{y}+\left(-a_1\right)^{-1}\mu_2\hat{x}-\hat{y}^3+\left(-a_1\right)^{\frac{-1}{2}}\mu_3 \hat{x}\hat{y}+b_0\hat{x}\hat{y}^2, &\\\nonumber
&\dot{\hat{y}}=-\hat{x}+\left(-a_1\right)^{-1}\mu_2\hat{y}+\left(-a_1\right)^{\frac{-1}{2}}\mu_3 \hat{y}^2+b_0\hat{y}^3.&
\end{eqnarray} Let \(\mu_0^{\ast}:=\left(-a_1\right)^{\frac{-5}{2}}\mu_0,\) \(\mu_1^{\ast}:=\left(-a_1\right)^{-2}\mu_1,\) \(\mu_2^{\ast}:=\left(-a_1\right)^{-1}\mu_2, \) \(\mu_3^{\ast}:=\left(-a_1\right)^{\frac{-1}{2}}\mu_3.\) Next, we replace \(\mu_i^{\ast}\) with \(\mu_i\) and \(\hat{y}\) with \(\tilde{y}\) for simplicity. Now apply the rescaling transformations \eqref{Rescaling} and expansion \eqref{EQ1} when
\begin{equation}
\gamma_1=\pm 0.1,  \gamma_{01}=\gamma_{02}=0,  \gamma_2=-1, \gamma_{11}=\gamma_{13}=0, \gamma_{12}=1,  \gamma_{21}=\gamma_{22}=0,  \gamma_{34}=0.
\end{equation} Recall that the unperturbed system is Hamiltonian and it holds a homoclinic cycle \(\Lambda_+\) for \(\gamma_1>0.\) This connect the stable and unstable manifolds of the saddle \(E_+.\) The homoclinic cycle \(\Lambda_-\) happens for \(\gamma_1<0\) and corresponds with \(E_-\).
Following the proof of Theorem \ref{Hom0}, we apply equations \eqref{2.20}, \((\tilde{y}_0(0), \tilde{y}_0(\frac{\pi}{2}))=(p_0+q_0, q_0-p_0),\) \(H(\tilde{x}_0(\frac{\pi}{2}), \tilde{y}_0(\frac{\pi}{2}))= H(0, p_0+q_0),\) and \(\frac{\partial H}{\partial \tilde{y}}(0, p_0+q_0)=0\) to obtain
\begin{eqnarray*}
&\tilde{y}_0=\mp0.7157063998\cos(2\varphi)\mp 0.2299428741,\; p_0=\mp0.7157063998, \; q_0=\mp 0.2299428741.&
\end{eqnarray*} where \(\tilde{x}_0\) is obtained from \(H(\tilde{x}_0,\tilde{y}_0)=\frac{1}{2}\tilde{x_0}^2-\frac{1}{2}\tilde{y_{0}}^2+\frac{1}{4}a_1\tilde{y_{0}}^4\pm\frac{1}{10}\tilde y_{0}=H(0, p_0+q_0).\)
Here, equations \eqref{phi} and \eqref{intHom} hold and we evaluate them at \(\varphi=\pi\) and \(\varphi=\pi, \pi/2\), respectively. We obtain
\(\gamma_{23}=\pm  0.2464356892 \gamma_{33}\) and \(\gamma_{31} := \pm 1.129378222 b_0.\) A substitution into the rescaling transformations completes the proof.
\end{proof}

\begin{rem}[Equilibria \(E_\pm\) are initially encircled by limit cycles \(\mathscr{C}^1_\pm\) and \(\mathscr{C}^2_\pm,\) respectively]
When the limit cycles \(\mathscr{C}^1_\pm\) are bifurcated, \(\mathscr{C}^1_+\) and \(\mathscr{C}^2_+\) encircle the secondary equilibrium \(E_+\) while \(\mathscr{C}^1_-\) and \(\mathscr{C}^2_+\) surround \(E_-\). However, steady-state bifurcations may lead to disappearances and appearances of \(E_\pm\) inside these limit cycles. For instance, Figure \ref{Fig2j} illustrates unstable limit cycle \(\mathscr{C}^1_+\) encircling spiral sink \(E_-.\) This is because \(\mathscr{C}^1_+\) and \(\mathscr{C}^2_+\) are initially bifurcated through a saddle-node bifurcation of limit cycles and surrounded \(E_+\) for controller coefficients from region \(b\) in \ref{fig1}. The limit cycle \(\mathscr{C}^2_+\) disappears through a subcritical Hopf bifurcation in Figure \ref{3Figcc} and next, two equilibria (\(E_-\) and a saddle) are born through a saddle-node bifurcation at \(T^-_{SN}.\) Then, \(E_+\) and the saddle point disappear at controller manifold \(T^+_{SN}.\) This leaves \(E_-\) as the only equilibrium living inside the limit cycle \(\mathscr{C}^1_+.\)
\end{rem}

Now we briefly discuss the differences between symmetry breaking bifurcations and symmetry preserving bifurcations in the following remark.
\begin{rem}[Symmetry preserving bifurcations versus symmetry breaking bifurcations] \label{SymBrkPrs}
There are some essential differences between \(\mathbb{Z}_2\)-equivariant bifurcations of Bogdanov-Takens singularity and its symmetry-breaking bifurcations. At the start of the analysis, we observe the asymmetric bifurcations through the appearances of secondary equilibria. Next, each of these may undergo asymmetric Hopf singularity where two tertiary limit cycles \(\mathscr{C}^1_\pm\) come to existence; \(\mathscr{C}^1_+\) surrounds \(E_+\) while \(\mathscr{C}^1_+\) encircles \(E_-\). Bifurcated limit cycles \(\mathscr{C}^1_\pm\) simultaneously collide with the origin and construct a double homoclinic cycle for the equivariant cases. The symmetry breaking, instead, causes these to construct two different {\it quaternary homoclinic cycles} \(\Gamma_\pm\). We have, of course, been naturally looked for these bifurcations.
However, the symmetry-breaking parameters unexpectedly allow the system to experience two different Bautin bifurcations from \(E_\pm.\) These are the alternatives to Hopf bifurcations from \(E_\pm\) for the symmetric bifurcations. For Bautin bifurcations, there are an ordered subcritical and supercritical bifurcations of limit cycles leading to two simultaneous limit cycles \(\mathscr{C}^1_+\) and \(\mathscr{C}^2_+\). Limit cycle \(\mathscr{C}^2_+\) has smaller amplitude than \(\mathscr{C}^1_+\) so that the smaller limit cycle \(\mathscr{C}^2_+\) lives in the interior of \(\mathscr{C}^1_+\). These limit cycles first merge to a single {\it bistable limit cycle} and then, disappear when controller coefficients cross the {\it saddle-node controller manifolds} of {\it limit cycles}. This is different from the saddle-node bifurcation of limit cycles in \cite[Lemma 5.7]{GazorSadriSicon}. A limit cycle \(\mathscr{C}_0\) bifurcates from the primary equilibrium. For the equivariant bifurcation cases, its symmetric feature causes only a heteroclinic cycle \(\Lambda,\) \ie it simultaneously collides with the secondary equilibria \(E_+\) and \(E_-.\) For symmetry-breaking bifurcations, the bifurcated limit cycle \(\mathscr{C}_0\) collides with either \(E_+\) or \(E_-.\) Thus, we have two different homoclinic cycles \(\Lambda_\pm\). The heteroclinic cycle \(\Lambda\) is still possible for symmetry-breaking cases by appropriately tuning the controller coefficients; see Theorem \ref{HeterThm}.
\end{rem}

\section{Nonlinear controllers for linearly uncontrollable cases}\label{SecBifCont}
\begin{figure}[t!]
\begin{center}
{\includegraphics[width=.4\columnwidth,height=.22\columnwidth]{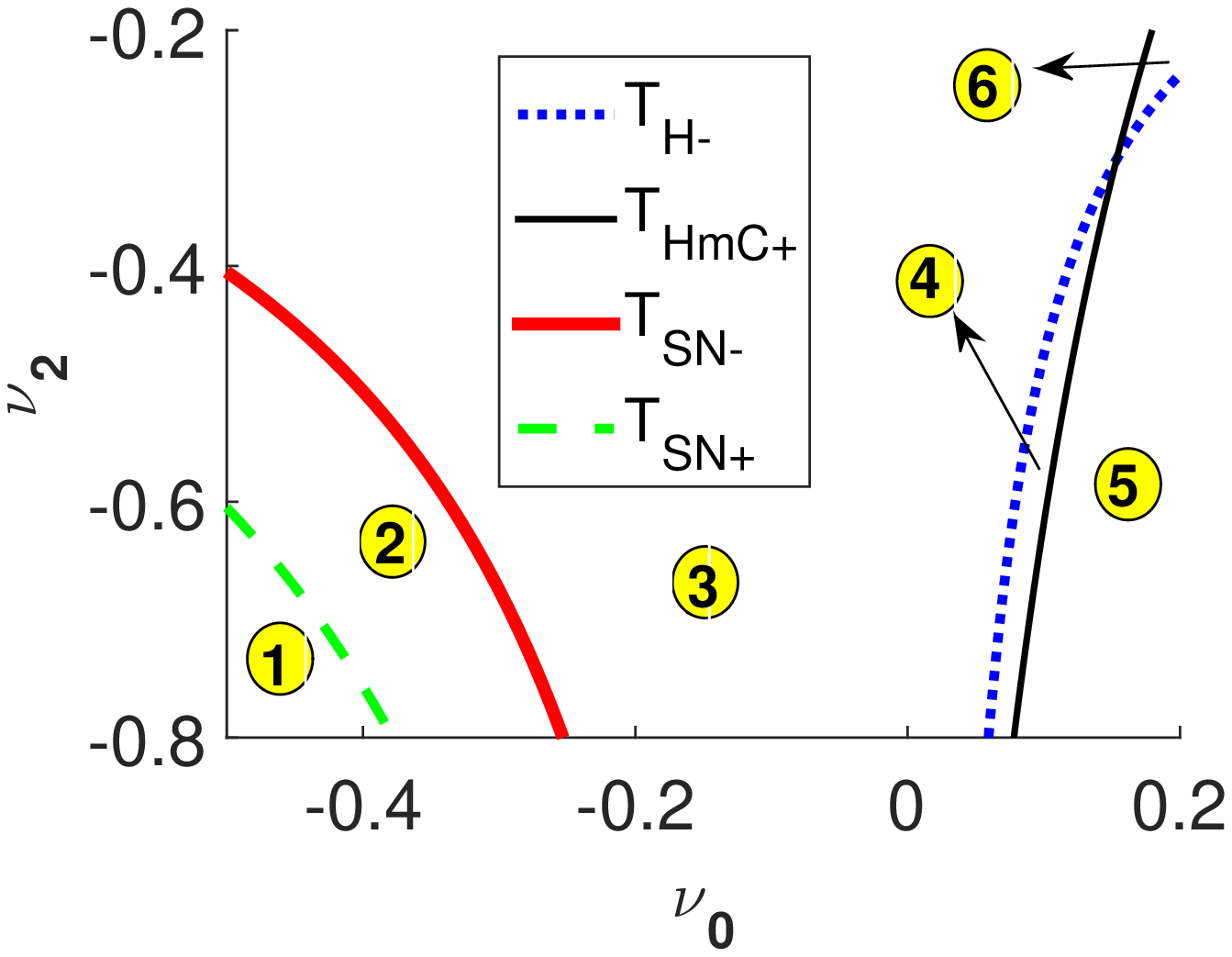}
\includegraphics[width=.4\columnwidth,height=.22\columnwidth]{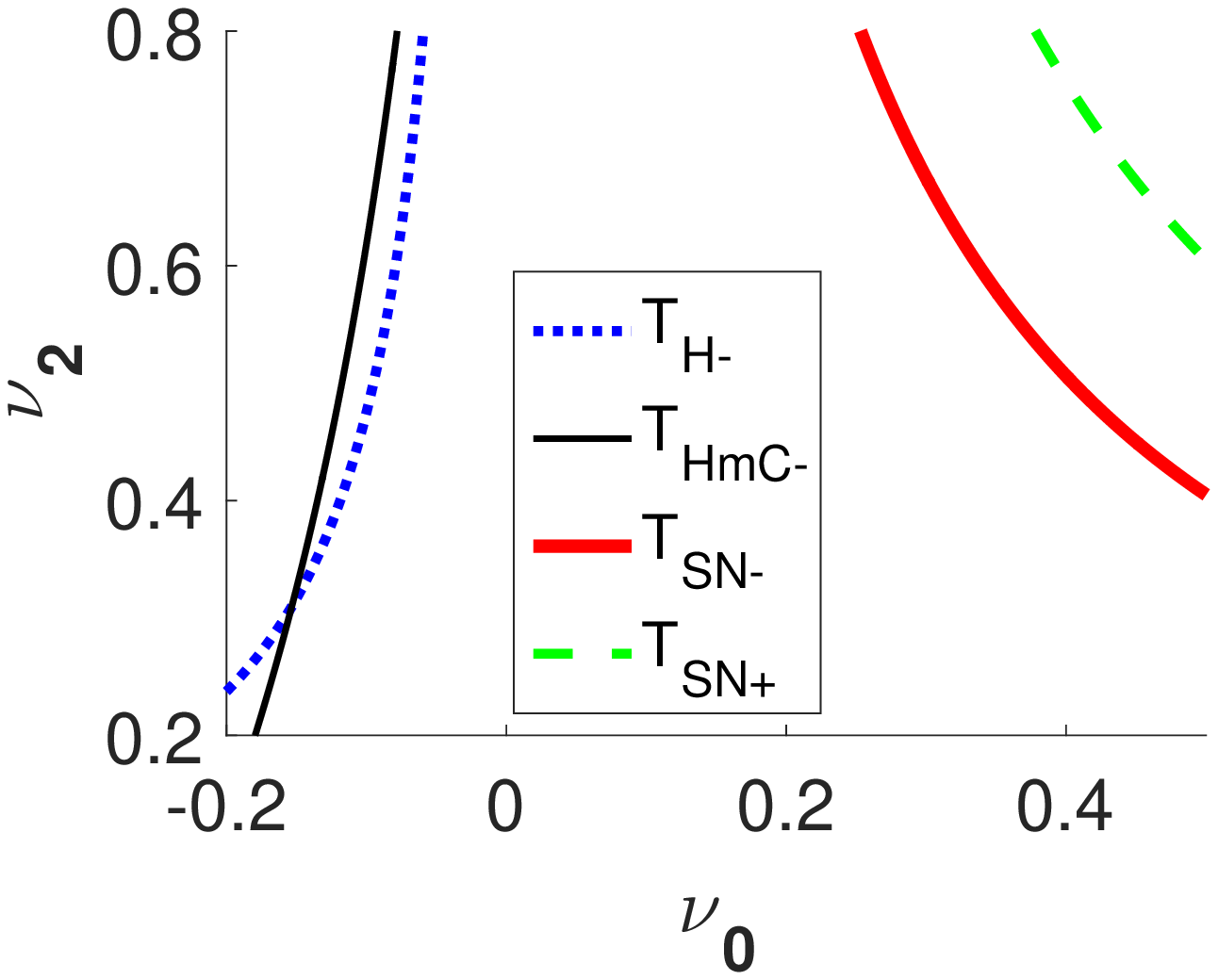}}
\end{center}
\vspace{-0.30 in}
\caption{Estimated bifurcation controller sets of system \eqref{BfControlr2s2} for the uncontrollable linearization case} \label{fIg61}
\vspace{-0.10 in}
\end{figure}

In this section we consider bifurcation control problem for a different differential system type \eqref{Eq01} with a single input controller whose linearization at the origin is either controllable or it is linearly uncontrollable. Since local symmetry-breaking bifurcations of \(\mathbb{Z}_2\)-equivariant control systems of type \eqref{Eq02} are determined by their cubic Taylor expansion, we consider cubic truncated controlled differential system \eqref{Eq02}, \ie
\begin{equation}\label{BfControlr2s2}
\dot{x}:=c_1 x^3+c_2 y^3+c_3 x y^2+c_4 y x^2+u_1,\;\; \dot{y}:=-x+c_5 x^3+c_6 y^3+c_7 x y^2+c_8 y x^2+u_2.
\end{equation} Here, we only treat single state-feedback input controller systems; \ie we choose either \(u_1:=0\) or \(u_2:= 0.\) When \(u_1\neq 0\), the linearization of system \eqref{BfControlr2s2} at the origin is linearly controllable; \ie it satisfies Kalman controllability rank condition. For the case \(u_2\neq 0,\) the linearized system has an uncontrollable mode. Many nonlinear techniques from nonlinear control theory fails for systems with uncontrollable modes. In this section, we show that the available controller coefficients can be exploited to enforce rich lists of bifurcation scenarios for both controllable and uncontrollable cases of system \eqref{BfControlr2s2}; see Figures \ref{ContrCase} and \ref{FIg6}.


\begin{thm}[When linearization satisfies Kalman condition] Let \(u_2:=0,\) \(u_1:=\upsilon_1 x+{{\upsilon_2}}y+ {{\upsilon_3}}x y,\) and \(\upsilon_i\) stand for controller coefficients. Then, controlled system \eqref{BfControlr2s2} admits a Hopf controller manifold, a pitchfork and a heteroclinic controller manifold given by
\(T_{H}= \{(\upsilon_1, {{\upsilon_2}}, {{\upsilon_3}})|\, \upsilon_1=0\},\) \(T_{P}= \{(\upsilon_1, {{\upsilon_2}}, {{\upsilon_3}})|\, {{\upsilon_2}}=0\},\) and
\begin{small}
\begin{eqnarray*}
&T_{HtC}=\Big\{(\upsilon_1, {{\upsilon_2}}, {{\upsilon_3}})\big|\,\frac{\upsilon_1}{2}-\frac{c_3+3c_6}{10c_2}{{\upsilon_2}}+
\frac{120c_7+120c_4-90c_8-270c_1}{720}\upsilon_1{{\upsilon_2}}+
\frac{(c_3+3c-6)\left(9c_1+3c_8-4c_4-4c_7\right)}{60c_2} {{{\upsilon_2}}}^2=0\Big\}.&
\end{eqnarray*}
\end{small}
Estimated homoclinic controller manifolds \(T_{HmC}\) and \(T_{HmC_{\pm}}\) in the controller coefficient space \((\upsilon_1, \upsilon_2, \upsilon_3)\) are
\begin{small}
\begin{eqnarray}\nonumber
&\frac{\upsilon_1}{2}+\frac{77{c_3}^2+213{c_6}^2+80c_2c_7+222c_3c_6}{320 c_2(c3+3c_6){\upsilon_1}^{-2}}+\frac{{c_3}^3-7c_6{c_3}^{2}-21{c_6}^{2}c_3
+8c_3c_2c_4+8c_3c_2c_7}{30{c_2}^3{{{\upsilon_2}}}^{-2}}+\frac{135{c_6}^3
+360c_1{c_2}^2+120{c_2}^2c_8-120c_2c_4c_6-408c_2c_6c_7}{240{c_2}^2(c_3+3c_6){\upsilon_1}^{-1}{{{\upsilon_2}}}^{-1}}
&\\\nonumber&- \frac{2{{\upsilon_2}}}{5c_2}(c_3+3c_6)+\frac{27{c_6}^3-72c_1{c_2}^{2}-24{c_2}^{2}c_8+24c_2c_6c_4+24c_2c_6c_7
}{30{c_2}^3}{{{\upsilon_2}}}^{2}+\frac{87c_3{c_6}^2-59{c_3}^3-163c_6{c_3}^2-40c_2c_3c_4-136c_2c_3c_7}{240{c_2}^2(c_3+3c_6)} \upsilon_1{{\upsilon_2}}=0,&
\\\nonumber
&\frac{\upsilon_1}{2}-\frac{2{{\upsilon_2}}}{5 c_2}\left( c_3+3c_6\right)\mp\big(\frac{3\sqrt{2}\pi}{32}+\frac{\sqrt{2} \pi \left(21c_6+13c_3\right)}{480 c_2}\upsilon_1\big) {{\upsilon_3}}\sqrt {-{{\upsilon_2}}} \mp\frac{(450c_3c_6+891{c_6}^2-189{c_3}^2-360c_2c_7) \sqrt{2}\pi}{7680{c_2}^2}{{\upsilon_3}} {{\upsilon_2}}\sqrt {-{{\upsilon_2}}}=0.&
\end{eqnarray}
\end{small}
\end{thm}
\bpr
Using near-identity changes of coordinates and a {\sc Maple} program, the system \eqref{BfControlr2s2} for \(u_2:=0\) and \(u_1:=\upsilon_0+ \upsilon_1 x+{{\upsilon_2}}y+ {{\upsilon_3}}x y\) can be transformed into \eqref{Eq03} where \(\mu_1= {{\upsilon_2}}+{\scalebox{0.7}{$\mathscr{O}$}}(||\upsilon||^2), \mu_2=\frac{\upsilon_1}{2}+{\scalebox{0.7}{$\mathscr{O}$}}(||\upsilon||^2), \)
\begin{scriptsize}
\begin{eqnarray}\nonumber
&\mu_0=\upsilon_0+\frac{288{c_2}^2c_8+864c_1{c_2}^2+77{c_3}^2c_6+231c_3{c_6}^2-288c_2c_4c_6
-96c_2c_3c_4-288c_2c_6c_7-96c_2c_3c_7-11{c_3}^3-297{c_6}^3}{96 {c_2}^2(c_3+3c_6)}\upsilon_0{{\upsilon_2}}&\\\nonumber&+\frac{10c_3c_6+48c_2c_7+23{c_3}^2-33{c_6}^2}{32c_2(c_3+3c_6)}\upsilon_0\upsilon_1,&\\\label{Controlable}
&\mu_3 = \frac{1442c_1{c_2}^2+212c_3{c_6}^2+72{c_3}^2c_6+482{c_2}^2c_8-482c_2c_6c_7
-162c_2c_3c_7-482c_2c_4c_6-162c_2c_3c_4-272{c_6}^3-{c_3}^3}{144{c_2}^2}\upsilon_0+\frac{{{\upsilon_3}}}{3}+{\scalebox{0.7}{$\mathscr{O}$}}(||\upsilon||^2).&
\end{eqnarray}
\end{scriptsize} This implies that the controlled differential system \eqref{BfControlr2s2} is fully unfolded. Here, we let \(\nu_0\) for simplicity. The controller curve \({T_{H_{\pm}}}\) are derived by substitution in \eqref{eta0}.
The controller sets for \eqref{BfControlr2s2} when \(\upsilon_0=0\) are derived from equations \eqref{Prop1},
\eqref{HmCpm}, and \eqref{Hetero} as claimed.
\epr

\begin{figure}[t!]
\begin{center}
\subfigure[The equilibrium is a spiral source.\label{Fig10(a)}]
{\includegraphics[width=.28\columnwidth,height=.22\columnwidth]{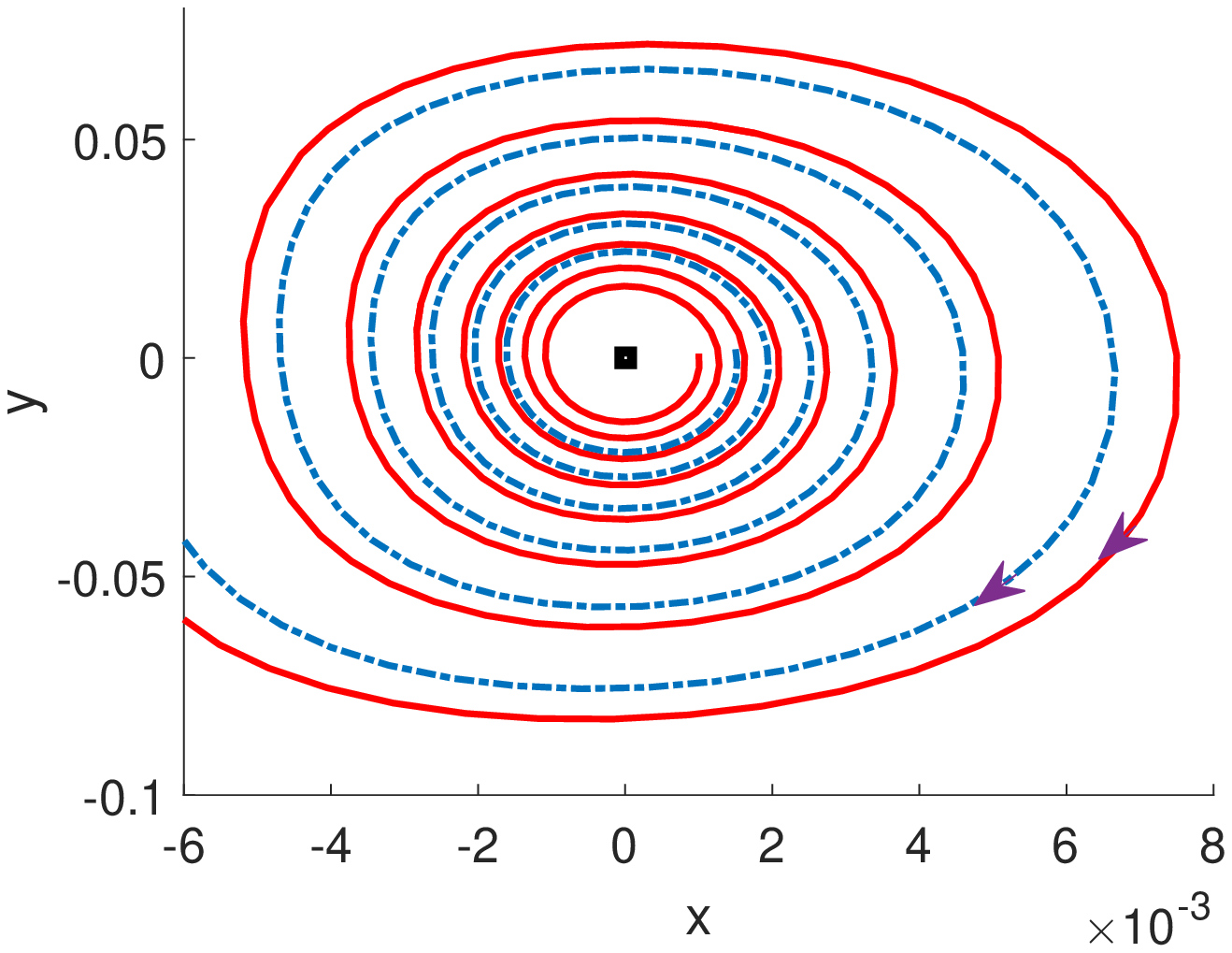}}\,
\subfigure[Unstable limit cycle \(\mathscr{C}_0\) encircles the spiral sink. \label{Fig10(b)}]
{\includegraphics[width=.32\columnwidth,height=.22\columnwidth]{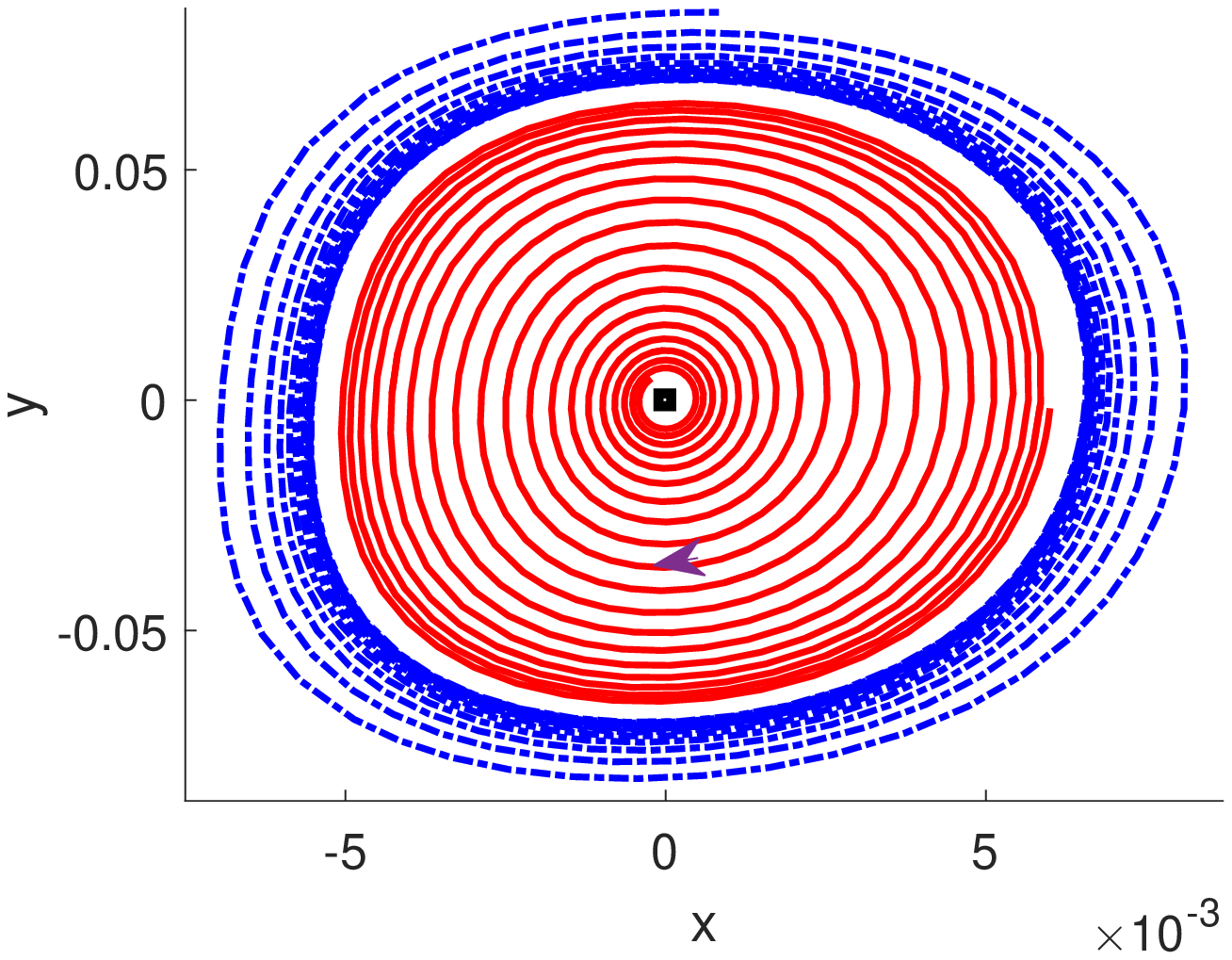}}\,
\subfigure[Unstable \(\mathscr{C}_0\), spiral sinks \(E_\pm\), saddle origin.]
{\includegraphics[width=.34\columnwidth,height=.22\columnwidth]{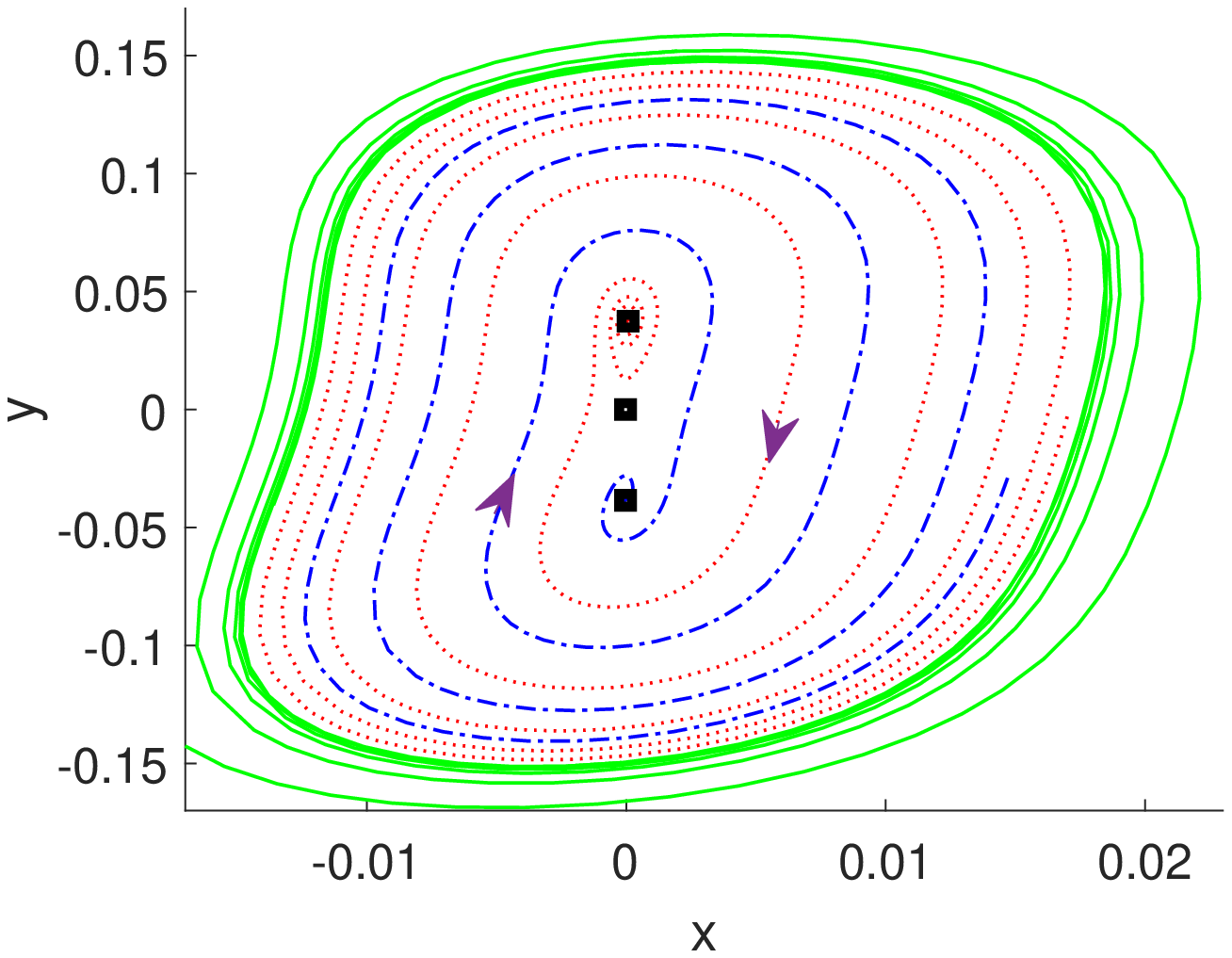}}\,
\subfigure[Sink \(E_-\), source \(E_+\), origin, unstable \(\mathscr{C}_0,\) and stable \(\mathscr{C}^1_+\).\label{Fig5d}]
{\includegraphics[width=.32\columnwidth,height=.22\columnwidth]{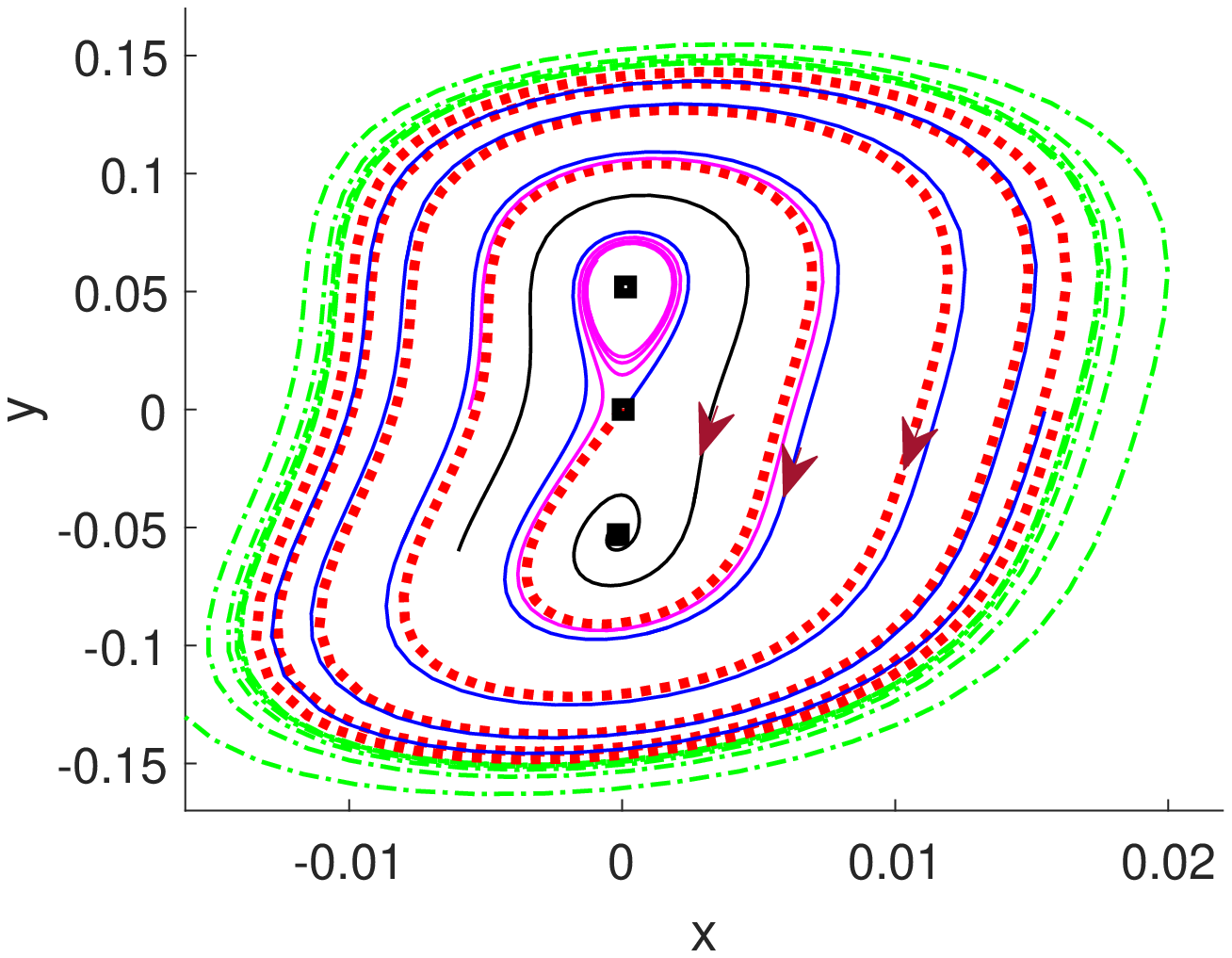}}\,
\subfigure[Source \(E_+\), spiral sink \(E_-,\) origin, and unstable \(\mathscr{C}_0.\)\label{Fig5e}]
{\includegraphics[width=.32\columnwidth,height=.22\columnwidth]{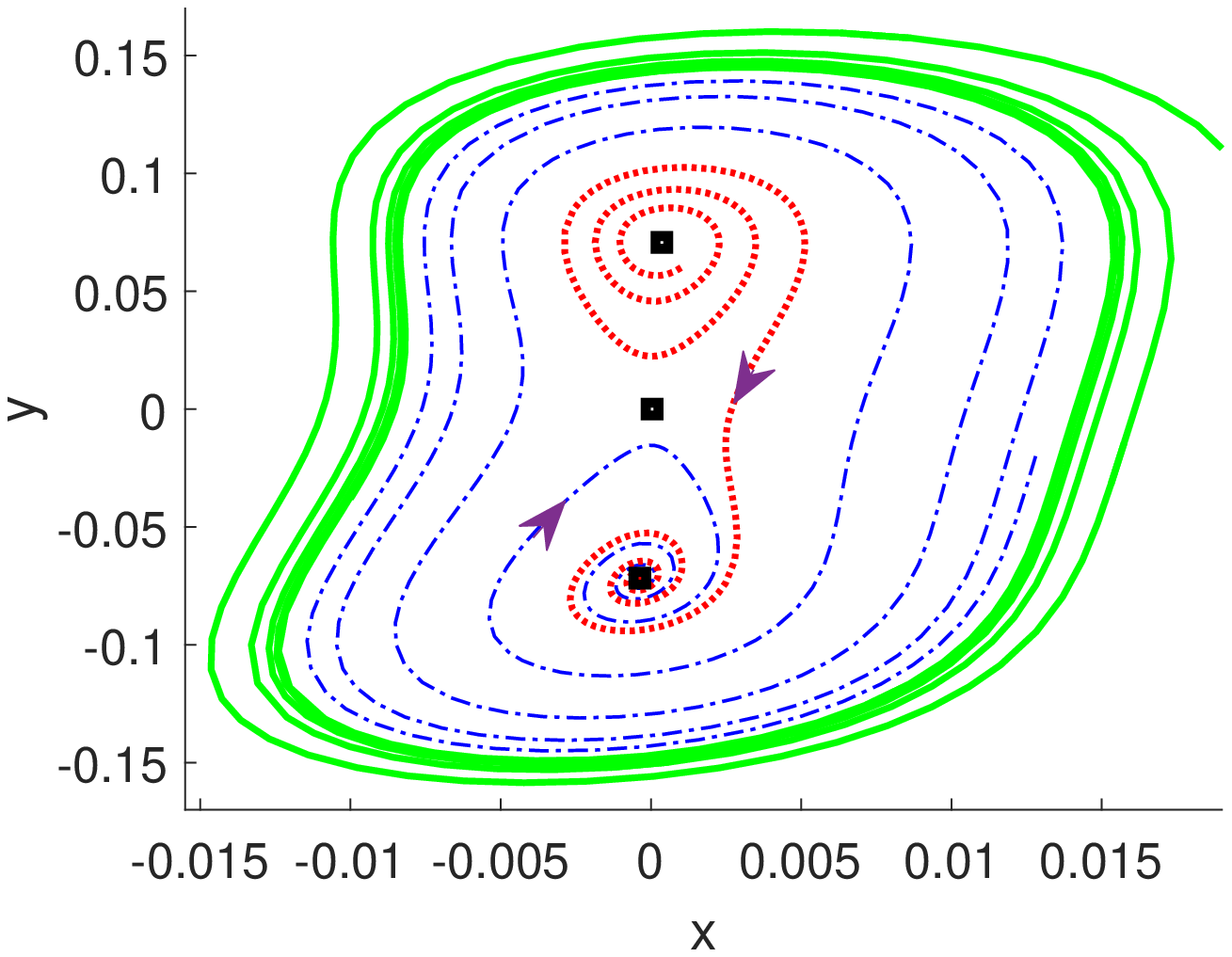}}\,
\subfigure[There are primary saddle, source \(E_+\) and sink \(E_-\).\label{Fig5f}]
{\includegraphics[width=.32\columnwidth,height=.22\columnwidth]{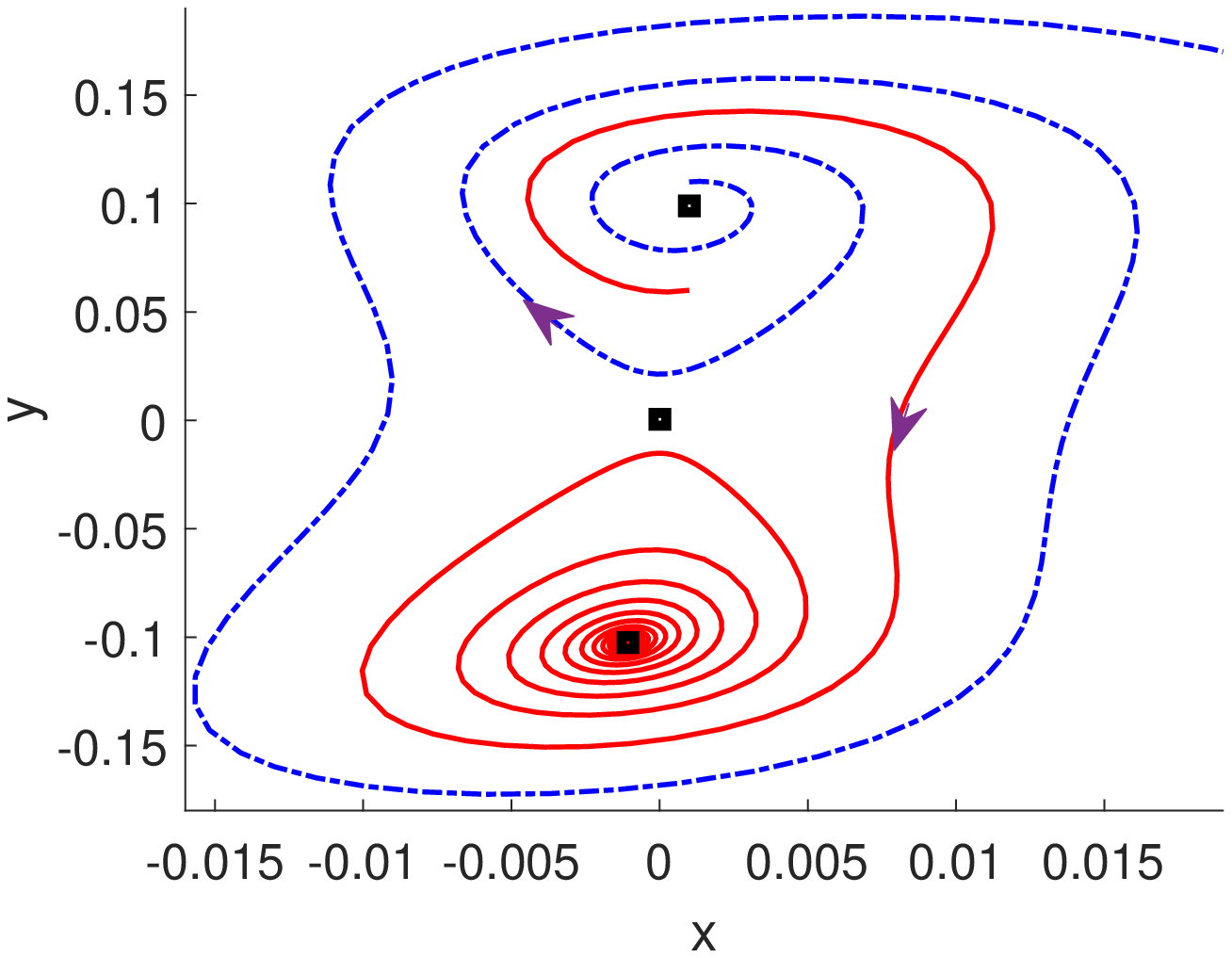}}\,
\subfigure[Two spiral sources \(E_\pm\), primary saddle, stable \(\mathscr{C}^1_-\)\label{Fig10(g)}]
{\includegraphics[width=.32\columnwidth,height=.22\columnwidth]{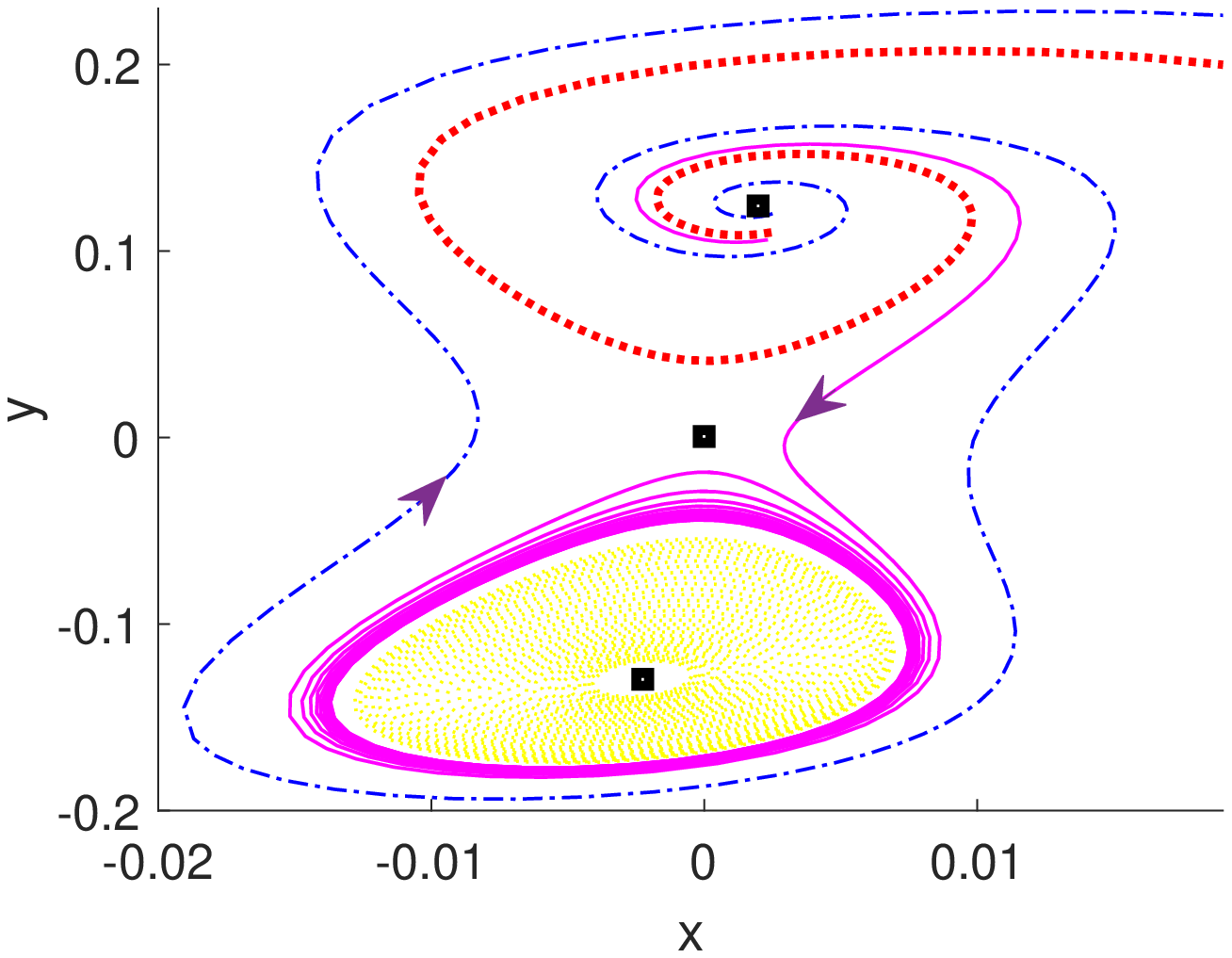}}\,
\subfigure[There are two spiral sources \(E_\pm\) and the primary saddle.\label{Fig10(h)}]
{\includegraphics[width=.32\columnwidth,height=.22\columnwidth]{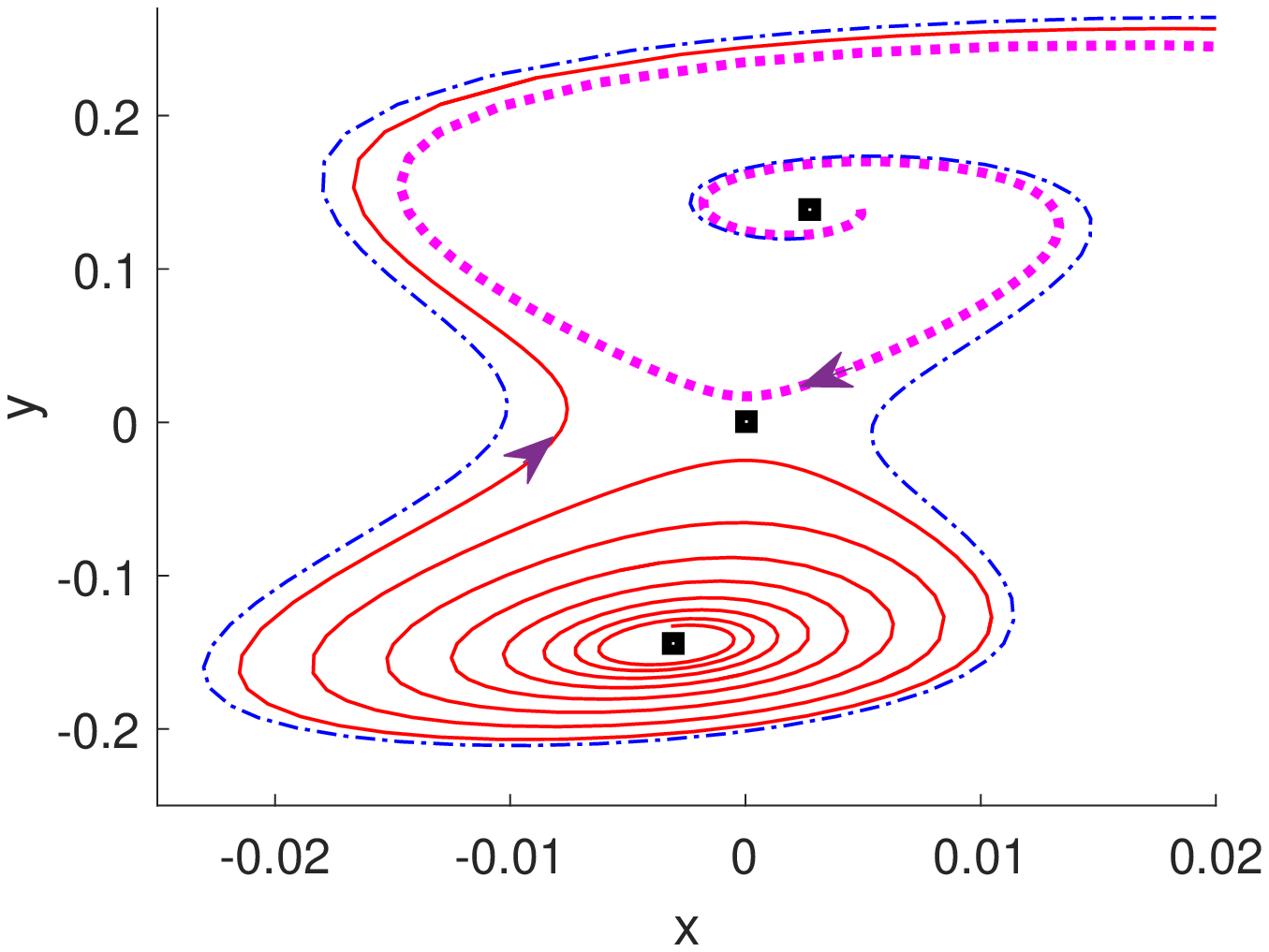}}\,
\subfigure[Three equilibria and unstable limit cycle \(\mathscr{C}^1_-\).\label{Fig10(k)}]
{\includegraphics[width=.32\columnwidth,height=.22\columnwidth]{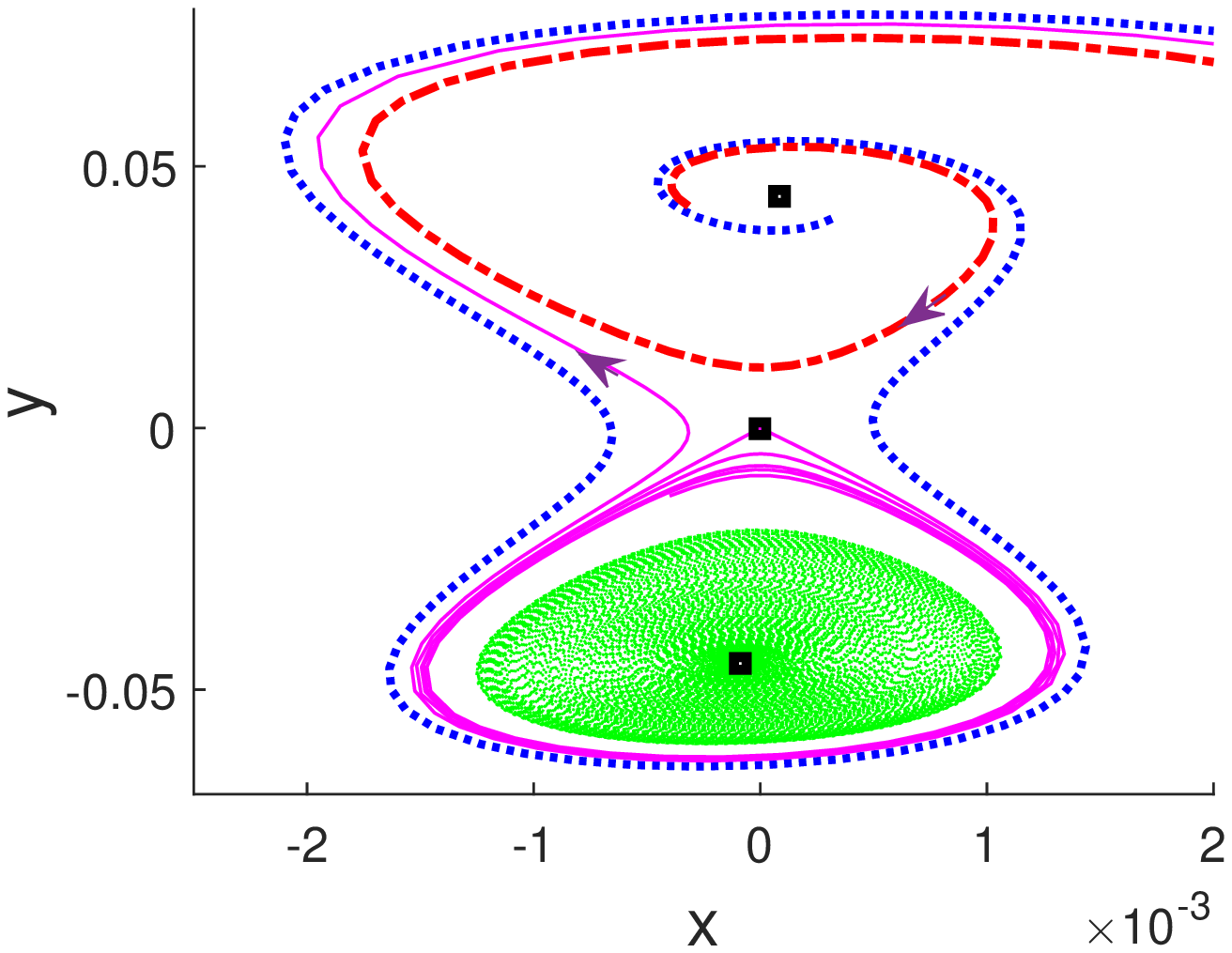}}\,
\end{center}
\vspace{-0.300 in}
\caption{Controlled phase portraits for linearly controllable case of \eqref{BfControlr2s2}. Figures \ref{Fig10(a)}-\ref{Fig10(k)} are associated with control coefficients chosen from regions (a)-(i) in Figure \ref{Fig9(a)}. \label{ContrCase}}
\vspace{-0.100 in}
\end{figure}

In order to illustrate the numerical \(\mathbb{Z}_2\)-breaking controller bifurcation varieties, we choose \(c_i:=1,\, i=1\ldots 8,\) the controller input \({{\upsilon_3}}:=\pm0.3,\) and obtain Figures \ref{Fig9(a)}-\ref{Fig9(c)}. These numerical controller manifolds are highly accurate over the plotted intervals. Next, input pairs (\(\upsilon_1, {{\upsilon_2}}\)) for values \((0.005, 0.005), (-0.005, 0.005),\) \((-0.025, -0.0014),\) \((-0.025, -0.0027),\) \((-0.025, -0.005),\) \((-0.025, -0.01),\) \((-0.025, -0.016),\) and \((-0.025, -0.02)\) are chosen from each connected region labeled (a)-(h) in Figure \ref{Fig9(a)}. We depict the numerical controlled phase portraits in Figures \ref{Fig10(a)}-\ref{Fig10(h)}, respectively.

\begin{thm}[Linearly uncontrollable cases]\label{UnconThm} Let \(u_1:=0\) and \(u_2:=\nu_0+\nu_1 y+\nu_2y^2.\) There are two saddle node controller manifolds \(T_{SN{\pm}}\), two Hopf controller sets \(T_{H{\pm}}\) and two homoclinic controller manifolds \(T_{HmC{{\pm}}}\). For simplicity, let \(c_1=-1,\) \(c_2=1,\) \(c_3=1,\) \(c_4=-1,\) \(c_5=-1,\) \(c_6=1,\) \(c_7=1,\) and \(c_8=-1.\) Then, these controller manifolds follow
\begin{eqnarray*}
&T_{SN{\pm}}=\left\{(\nu_0,\nu_1,\nu_2)|\, 4\nu_0\nu_2=1\pm2\nu_1+{\nu_1}^2 \right\},\quad\quad
T_{H{\pm}}=\left\{(\nu_0,\nu_1,\nu_2)|\, \nu_1=\pm1\mp\sqrt{1+4\nu_0\nu_2} \right\},&\\
&\hbox{ and } \qquad T_{HmC{{\pm}}}=\left\{(\nu_0,\nu_1,\nu_2)|\, \frac{1}{2}\nu_1+{\frac {16}{45}}{\nu_0}^{2}+{\frac {37}{80}}{
\nu_1}^{2}-\frac{1}{3}\nu_0\nu_2\mp{\frac{\sqrt {2}\pi\left(16{\nu_0}^{2}+3{\nu_1}^{2}\right) \left(4\nu_0-3\nu_2\right) }{32\sqrt {48{\nu_0}^{2}+9{\nu_1}
^{2}}}}=0 \right\}. &
\end{eqnarray*}
\end{thm}
\bpr
Here we deal with bifurcation control of system \eqref{BfControlr2s2} when Equations \eqref{Controlable} is now replaced with \(\mu_0=0,\) \(\mu_2=\frac{3c_2c_4-c_3^2}{3c_2}\nu_0^2,\) \(\mu_3=\frac{2(3c_2c_4+3c_7c_2-c_3^2-3c_3c_6)}{9c_2}\nu_0+\frac{2}{3}\nu_5,\) and \(\mu_1\) estimated by
\begin{scriptsize}
\begin{eqnarray*}\label{UnControlable}
&\mu_1=\frac{32 \left( 3c_6+c_3 \right)  \left({c_3}^3+3c_6{c_3}^2-6c_2c_3c_4-6c_2c_3c_7+27c_1{c_2}^2+9{c_2}^{2}c_8 \right) {\nu_0}^2+9c_2 \left(64c_2c_4+80c_2c_7-23{c_3}^2
-58c_3c_6+81{c_6}^2 \right) {\nu_1}^2}{576{c_2}^2(3c_6+c_3)}
+\frac{288{c_2}\nu_1-192c_3\nu_0\nu_2}{576{c_2}}.&
\end{eqnarray*}
\end{scriptsize} The claims are then followed from Theorems \ref{Thm22}, \ref{thm3}, and \ref{Hom0}. Replacing the values for \(c_i\)-s, system \eqref{BfControlr2s2} has four equilibria given by
\begin{footnotesize}
\begin{eqnarray*}
&(x_1^{\pm}, y_1^{\pm})=\left(\frac{1+\nu_1\pm\sqrt{1+2\nu_1+\nu_1^2-4\nu_0\nu_2}}{2\nu_2},\frac{1+\nu_1\pm\sqrt{1+2\nu_1+\nu_1^2-4\nu_0\nu_2}}{2\nu_2}\right),&\\&
(x_2^{\pm}, y_2^{\pm})=\left(\frac{1-\nu_1\pm\sqrt{1-2\nu_1+\nu_1^2-4\nu_0\nu_2}}{2\nu_2}, \frac{1-\nu_1\pm\sqrt{1-2\nu_1+\nu_1^2-4\nu_0\nu_2}}{2\nu_2}\right).&
\end{eqnarray*}
\end{footnotesize} The estimated controller sets follow equations \eqref{nu1eq}, \eqref{eta0}, and \eqref{HmCpm}.
\epr

We take \(\nu_1=0.1\) and depict critical controller manifolds in terms of controller coefficients (\(\nu_0, \nu_2\)) in Figures \ref{fIg61}. We choose \((-0.7, -0.7),\) \((-0.4, -0.6),\) \((0.1, -0.2),\) \((0.1, -0.5),\) \((0.1, -0.66),\) and \((0.3, -0.1)\)
for \((\nu_0, \nu_2)\) from regions \tikz[baseline=(char.base)]{\node[shape=circle,fill=yellow!90,scale=0.8,draw=black!100,inner sep=2pt] (char) {1};}-\tikz[baseline=(char.base)]{\node[shape=circle,fill=yellow!90,scale=0.8,draw=black!100,inner sep=2pt] (char) {6};} in the first figure of \ref{fIg61}. Then, the numerical phase portraits are illustrated in Figures \ref{UncFig1}-\ref{UncFig5}, respectively.
System \eqref{BfControlr2s2} has no equilibrium for controller coefficients choices from region \tikz[baseline=(char.base)]{\node[shape=circle,scale=0.8,fill=yellow!90,draw=black!100,inner sep=2pt] (char) {1};}. Two small equilibria are bifurcated via a fold bifurcation at \(T_{SN{+}}\); see Figure \ref{UncFig2}. Another saddle-node bifurcation occurs at \(T_{SN{-}}\) and two new equilibria are born. Thus, there are four equilibria for controller coefficient choices from region \tikz[baseline=(char.base)]{\node[shape=circle,fill=yellow!90,scale=0.8,draw=black!100,inner sep=2pt] (char) {3};}. As for controller coefficients of region \tikz[baseline=(char.base)]{\node[shape=circle,fill=yellow!90,scale=0.8,draw=black!100,inner sep=2pt] (char) {4};}, a subcritical Hopf bifurcation gives rise to an unstable limit cycle; see Figure \ref{UncFig4}. The limit cycle disappears at homoclinic controller set \(T_{HmC{{+}}}\) and therefore, there is no limit cycle for coefficients from region \tikz[baseline=(char.base)]{\node[shape=circle,fill=yellow!90,scale=0.8,draw=black!100,inner sep=2pt] (char) {5};}. Control choices associated with  region \tikz[baseline=(char.base)]{\node[shape=circle,fill=yellow!90,scale=0.8,draw=black!100,inner sep=2pt] (char) {6};} lead to a stable limit cycle.

\begin{rem}[Subcritical and supercritical type switching for bifurcation varieties] \label{SwitchingSub}
Linear stability analysis provides a local approach and its neighborhood validity is essential for practical life problems. The neighborhood validity can be very small if nonlinear terms give rise to a subcritical Hopf bifurcation. The basin of attraction for the asymptotically stable equilibrium in subcritical cases is merely the interior of the small bifurcated unstable limit cycle. Therefore, the {\it linear stability of this kind} fails in many control engineering applications. {\it Primary supercritical Hopf bifurcation} has been widely considered as a {\it safe controller design}; \eg see \cite{Kang04}. This is due to its expected large basin of attraction for the bifurcated asymptotically stable limit cycle.

Figure \ref{fIg61} illustrates that controller coefficient choices from region \tikz[baseline=(char.base)]{\node[shape=circle,fill=yellow!90,scale=0.8,draw=black!100,inner sep=2pt] (char) {4};} correspond with unstable limit cycle in Figure \ref{UncFig4} while choices from region \tikz[baseline=(char.base)]{\node[shape=circle,fill=yellow!90,scale=0.8,draw=black!100,inner sep=2pt] (char) {6};} cause stable limit cycle in Figure \ref{UncFig6}. The same phenomenon occurs for regions \(g\) and \(i\) in Figure \ref{Fig9(a)}. We have stable limit cycle \(\mathscr{C}^1_-\) in Figure \ref{Fig10(g)} while \(\mathscr{C}^1_-\) is unstable in Figure \ref{Fig10(k)}. They play a potential role for stabilizing the system. Controller coefficients from regions \tikz[baseline=(char.base)]{\node[shape=circle,fill=yellow!90,scale=0.8,draw=black!100,inner sep=2pt] (char) {6};} and \(i\) from \ref{Fig10(k)} and \ref{Fig9(a)} fail the restrictions in Theorem \ref{thm3}. These imply the existence of Bautin bifurcation in Theorem \ref{THmBautin}.

\end{rem}

\begin{figure}[t!]
\begin{center}
\subfigure[\({{\upsilon_3}}:=0.3\), \(a_1=b_0=1\)\label{Fig9(a)}]
{\includegraphics[width=.25\columnwidth,height=.22\columnwidth]{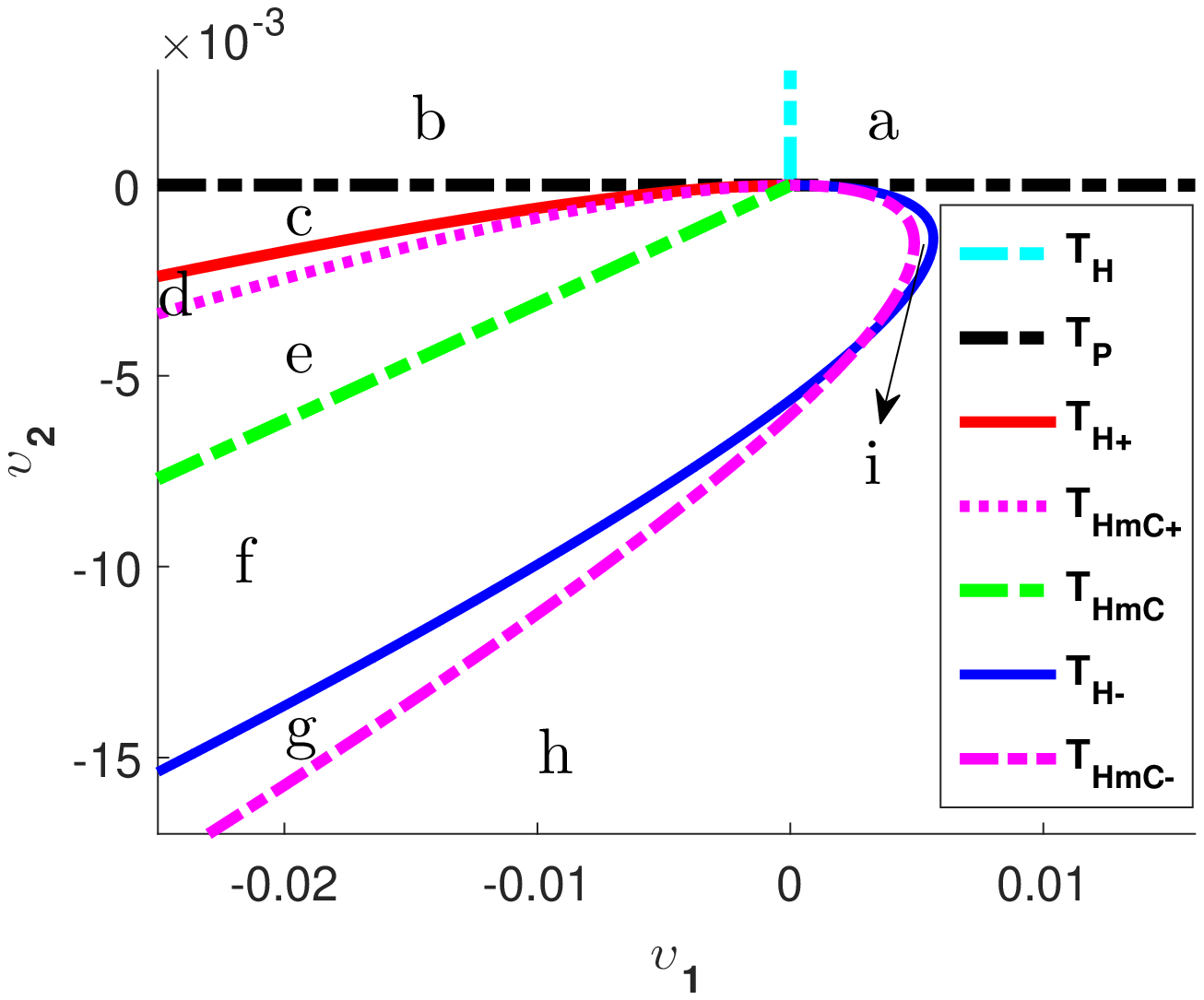}}
\subfigure[\({{\upsilon_3}}:=-0.3,\) \(a_1= b_0=1\) \label{Fig9(c)}]
{\includegraphics[width=.25\columnwidth,height=.22\columnwidth]{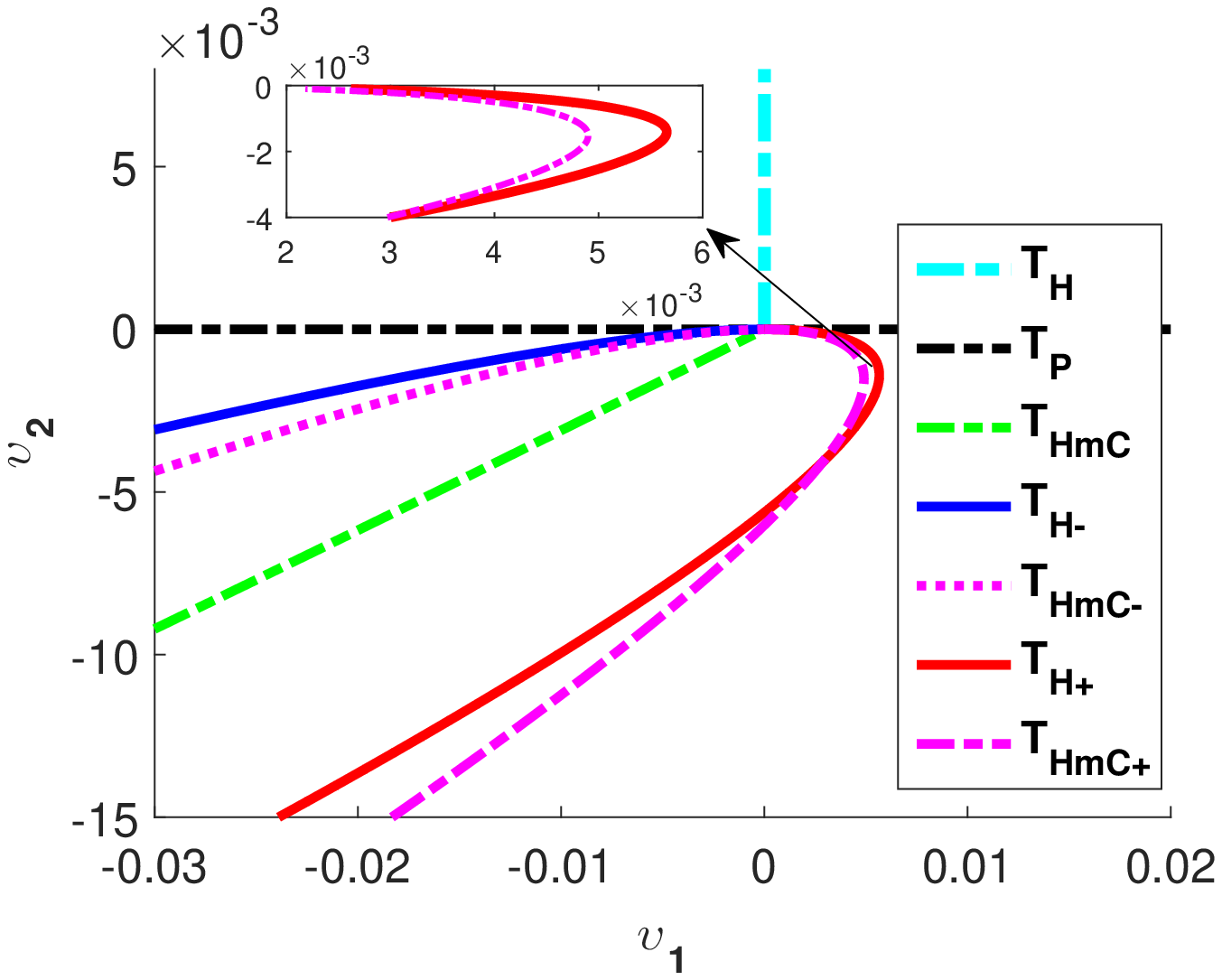}}
\subfigure[\({\upsilon_3}:=\pm 0.3, a_1= b_0=-1\)\label{Fig9(b)}]{\includegraphics[width=.23\columnwidth,height=.22\columnwidth]{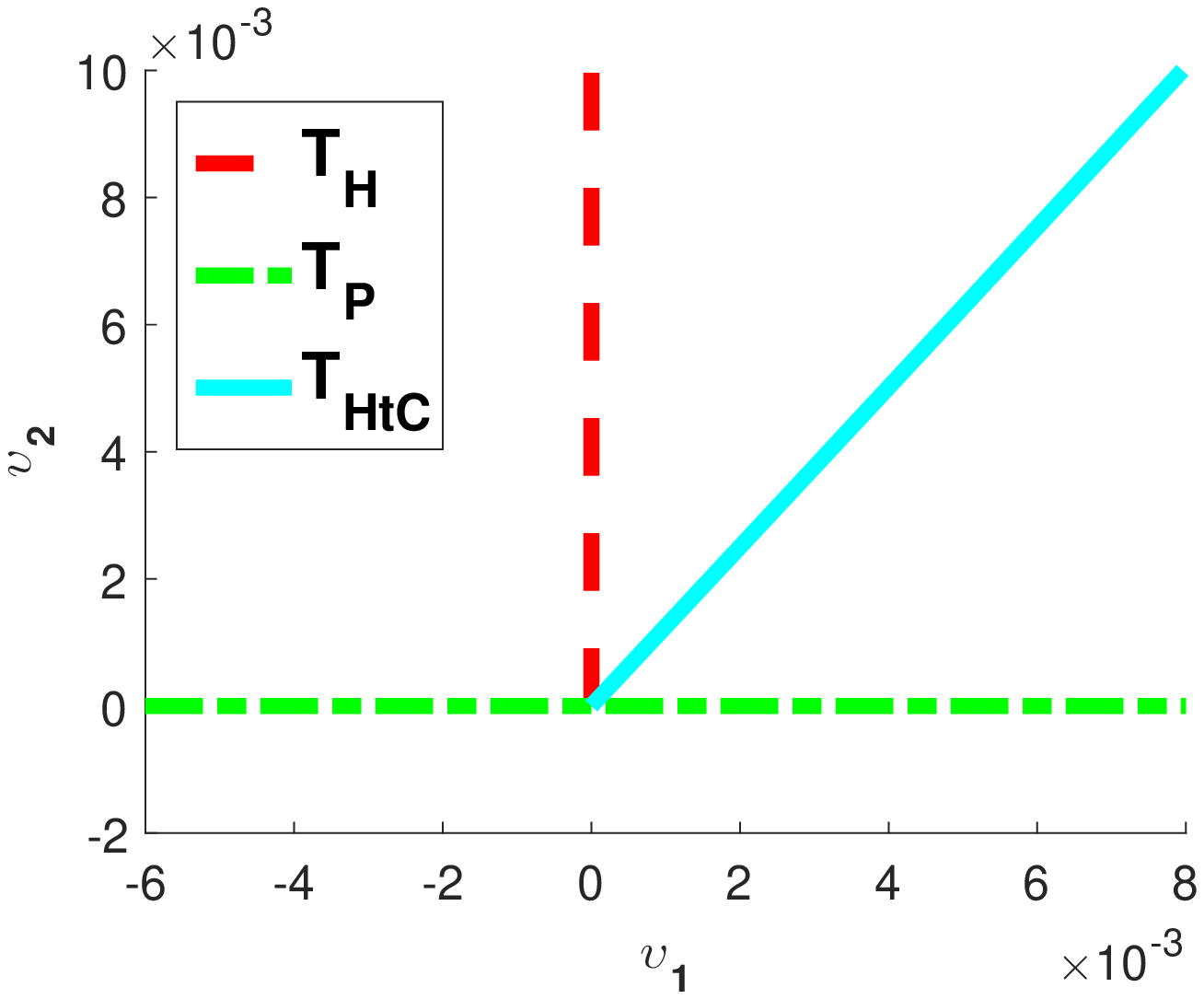}}
\subfigure[\({{\upsilon_3}}:=\pm 0.3,\) \(b_0=- a_1=1\)\label{Fig9(d)}]{\includegraphics[width=.24\columnwidth,height=.22\columnwidth]{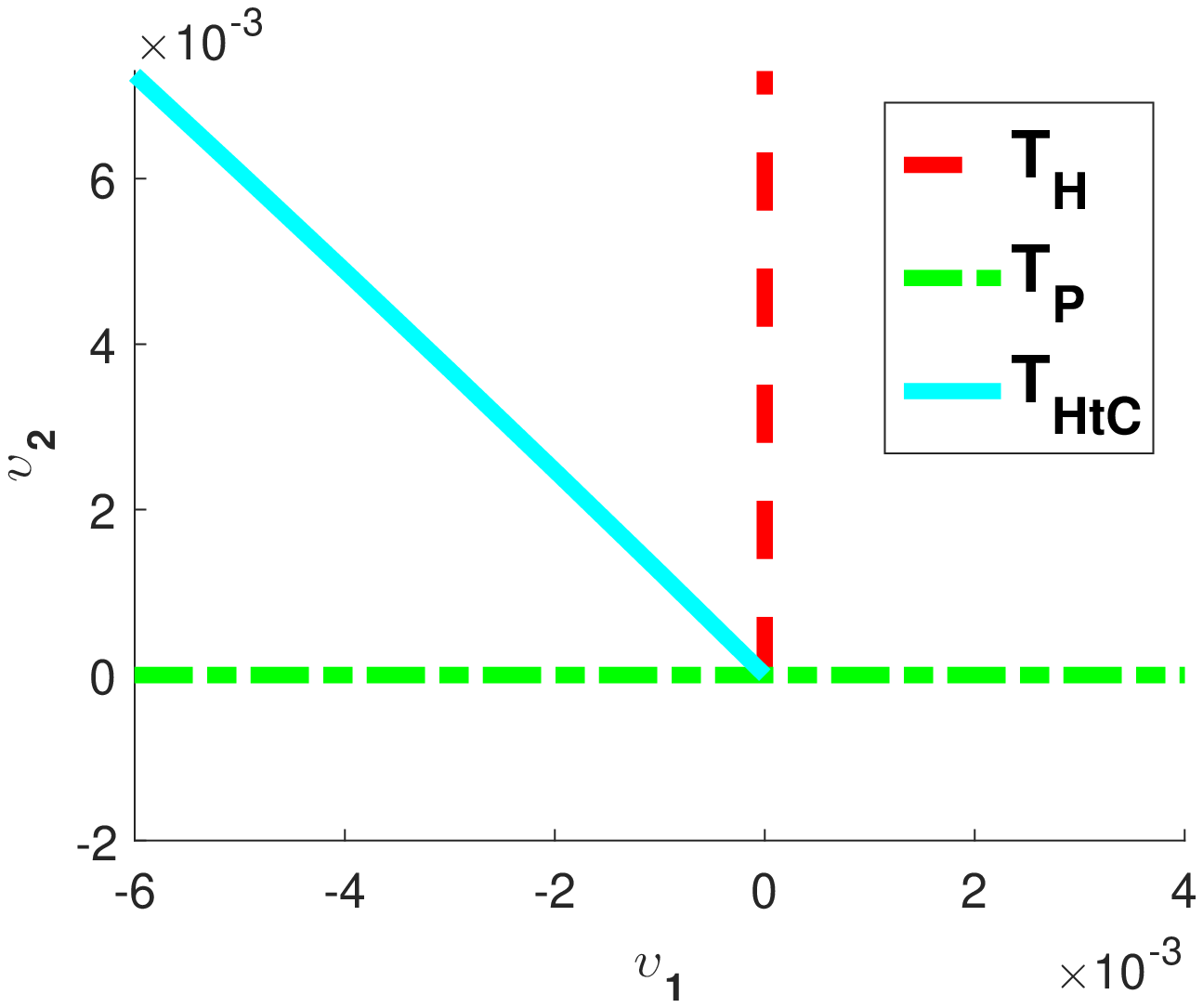}}
\end{center}
\vspace{-0.280 in}
\caption{Symmetry-breaking controller set for linearly controllable case of \eqref{BfControlr2s2}, \ie \(u_2:= 0.\) } \label{Fig9}
\vspace{-0.100 in}
\end{figure}

\begin{figure}[t!]
\begin{center}
\subfigure[No equilibrium\label{UncFig1}]
{\includegraphics[width=.3\columnwidth,height=.15\columnwidth]{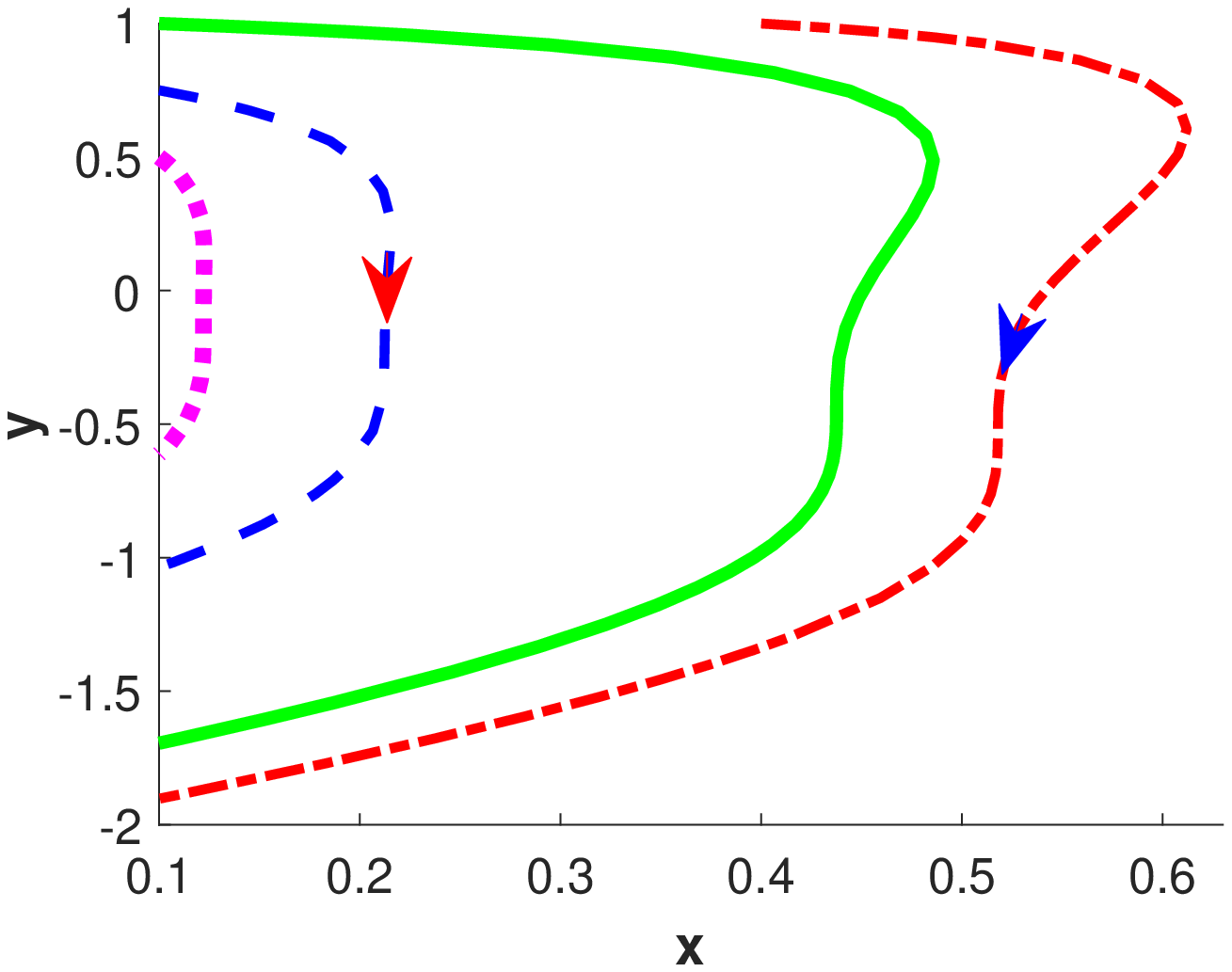}}
\subfigure[Two equilibrium\label{UncFig2}]
{\includegraphics[width=.3\columnwidth,height=.15\columnwidth]{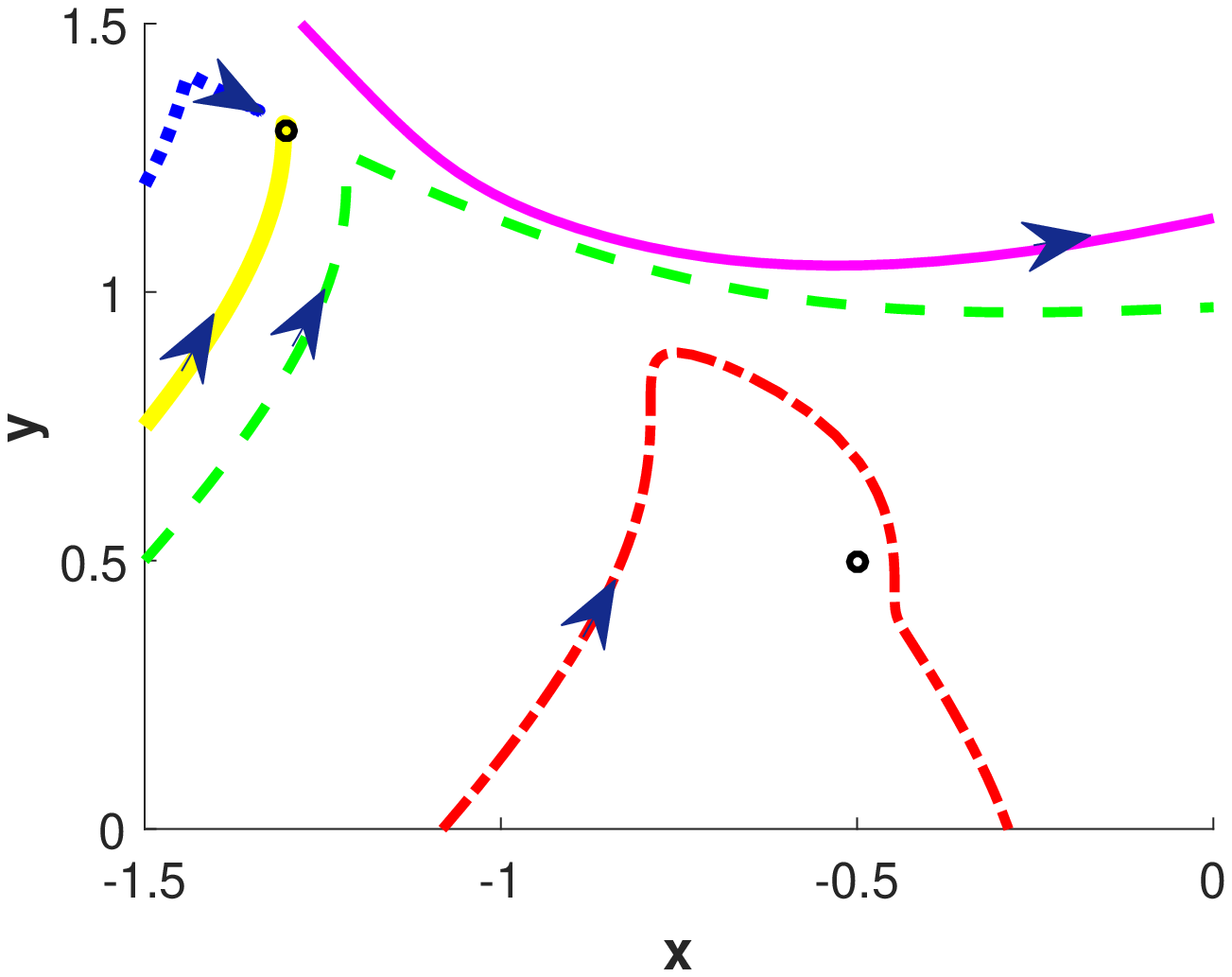}}
\subfigure[No stable equilibrium]
{\includegraphics[width=.3\columnwidth,height=.15\columnwidth]{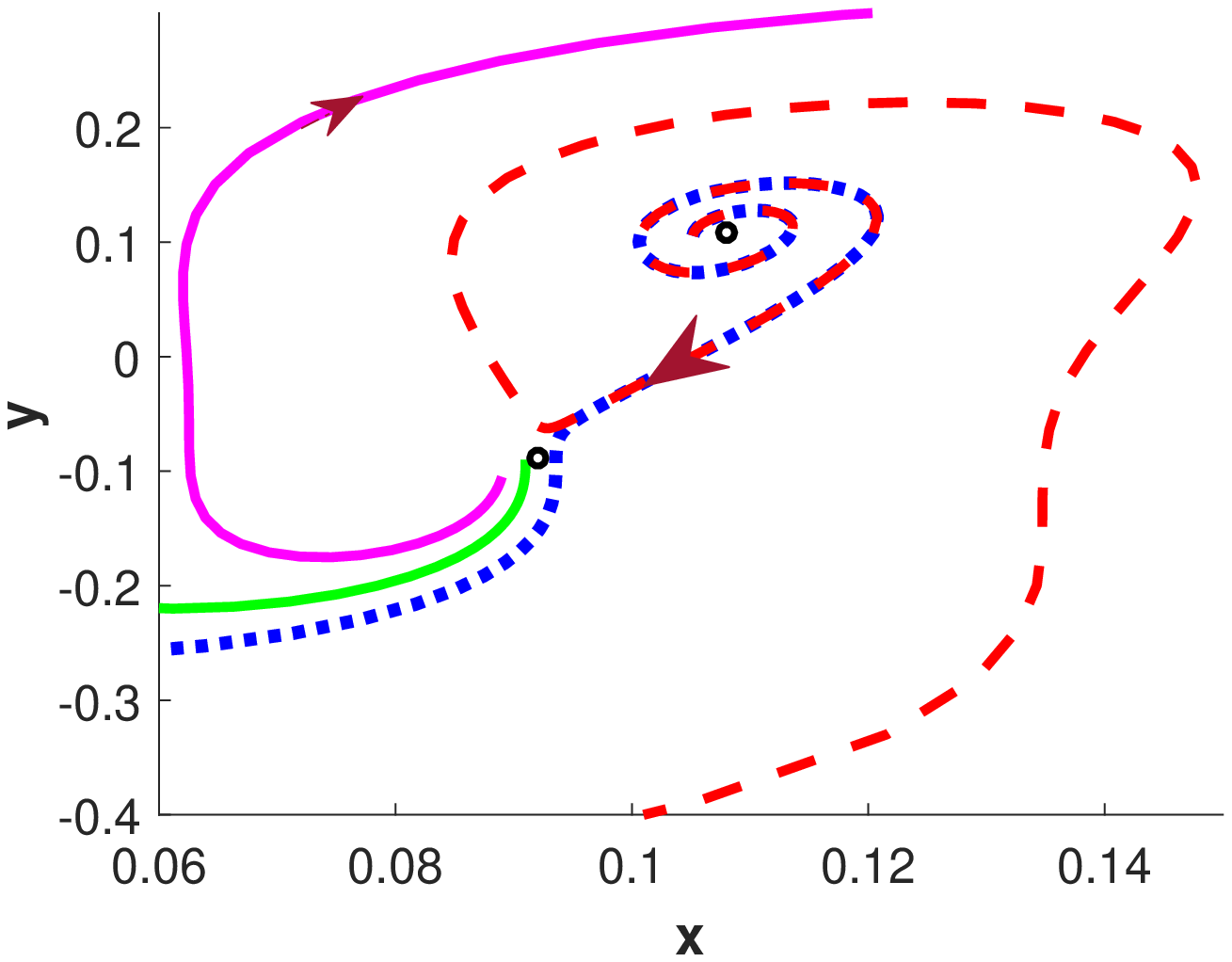}}
\subfigure[An unstable limit cycle\label{UncFig4}]
{\includegraphics[width=.3\columnwidth,height=.15\columnwidth]{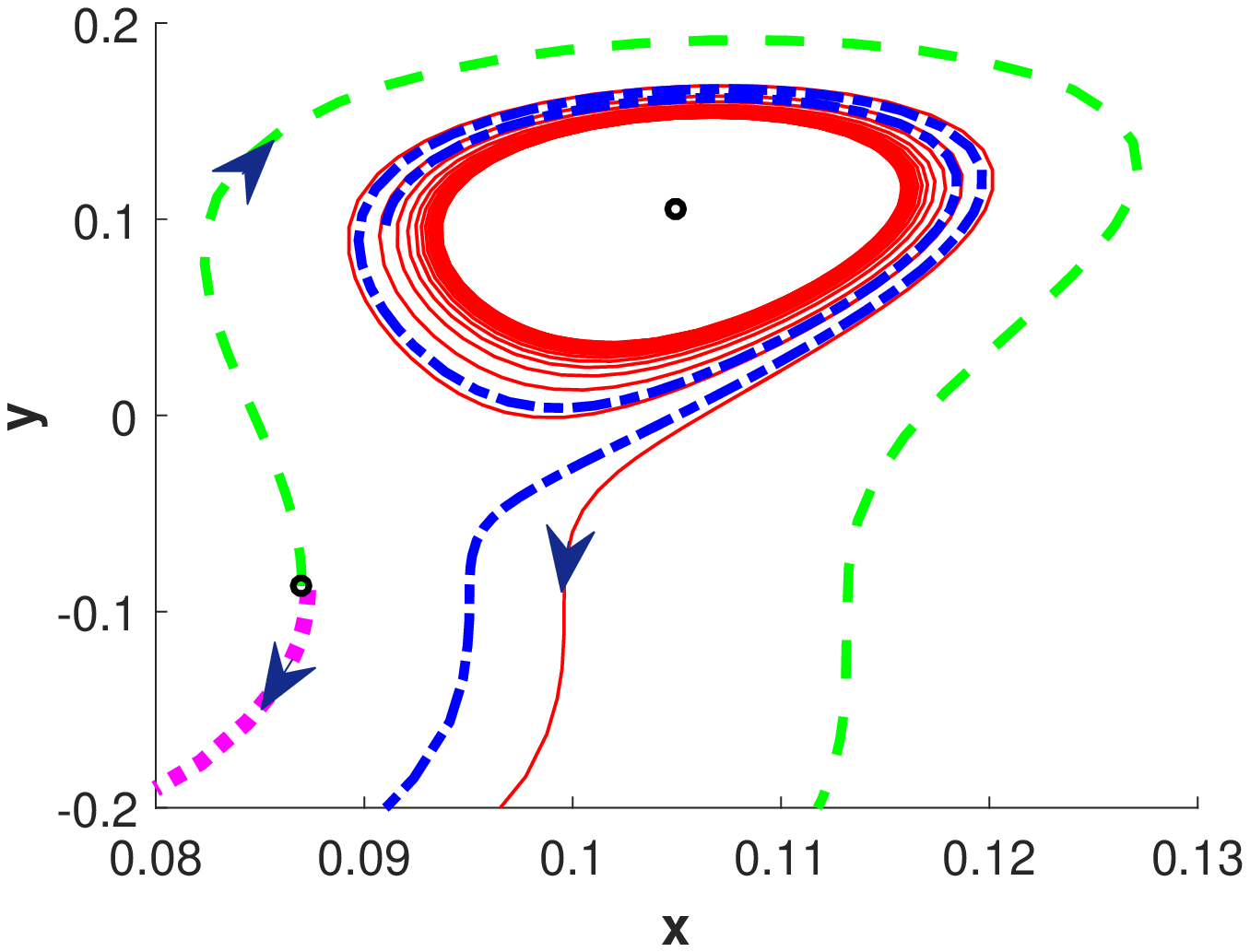}}
\subfigure[One stable equilibrium\label{UncFig5}]
{\includegraphics[width=.3\columnwidth,height=.15\columnwidth]{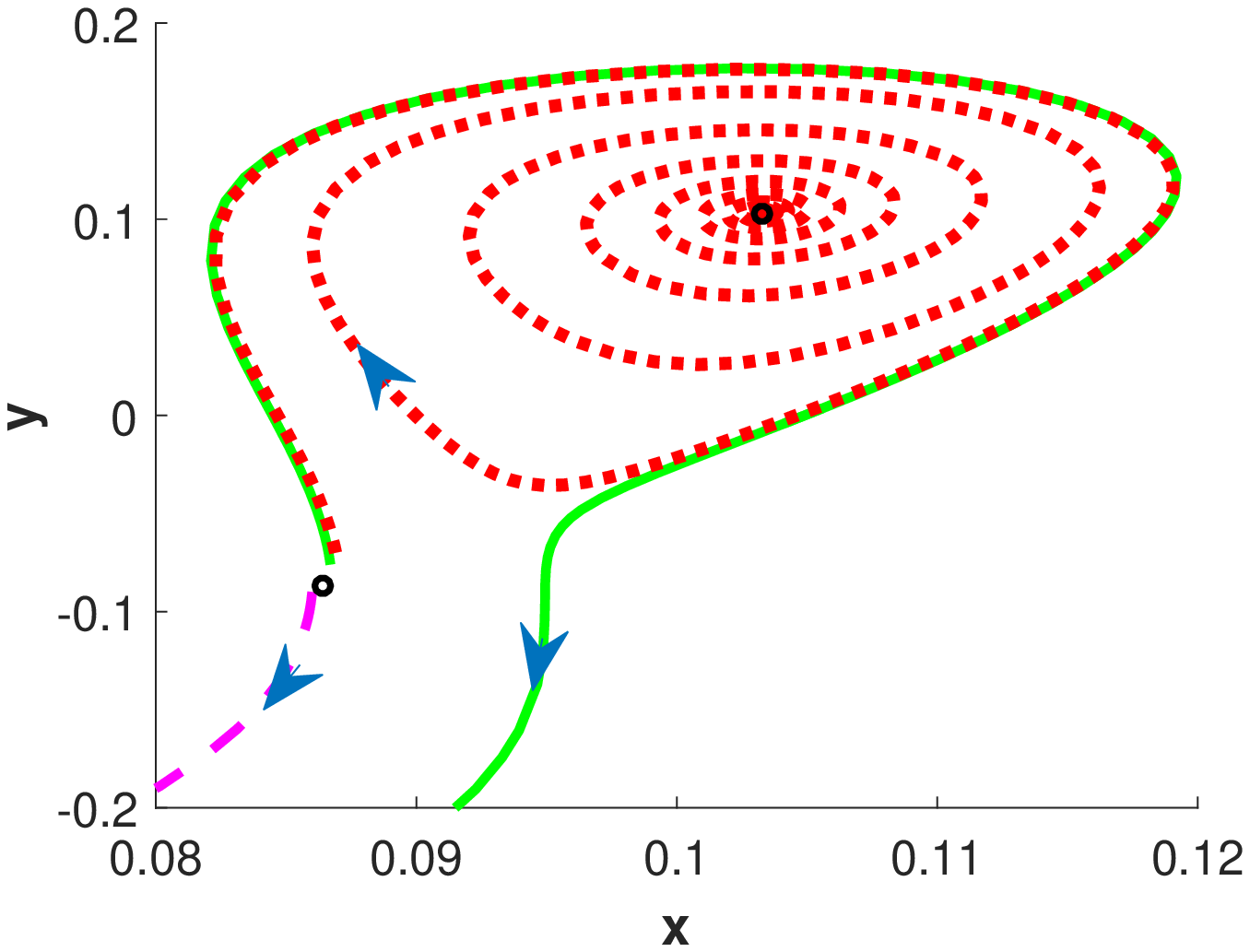}}
\subfigure[A stable limit cycle\label{UncFig6}]
{\includegraphics[width=.3\columnwidth,height=.15\columnwidth]{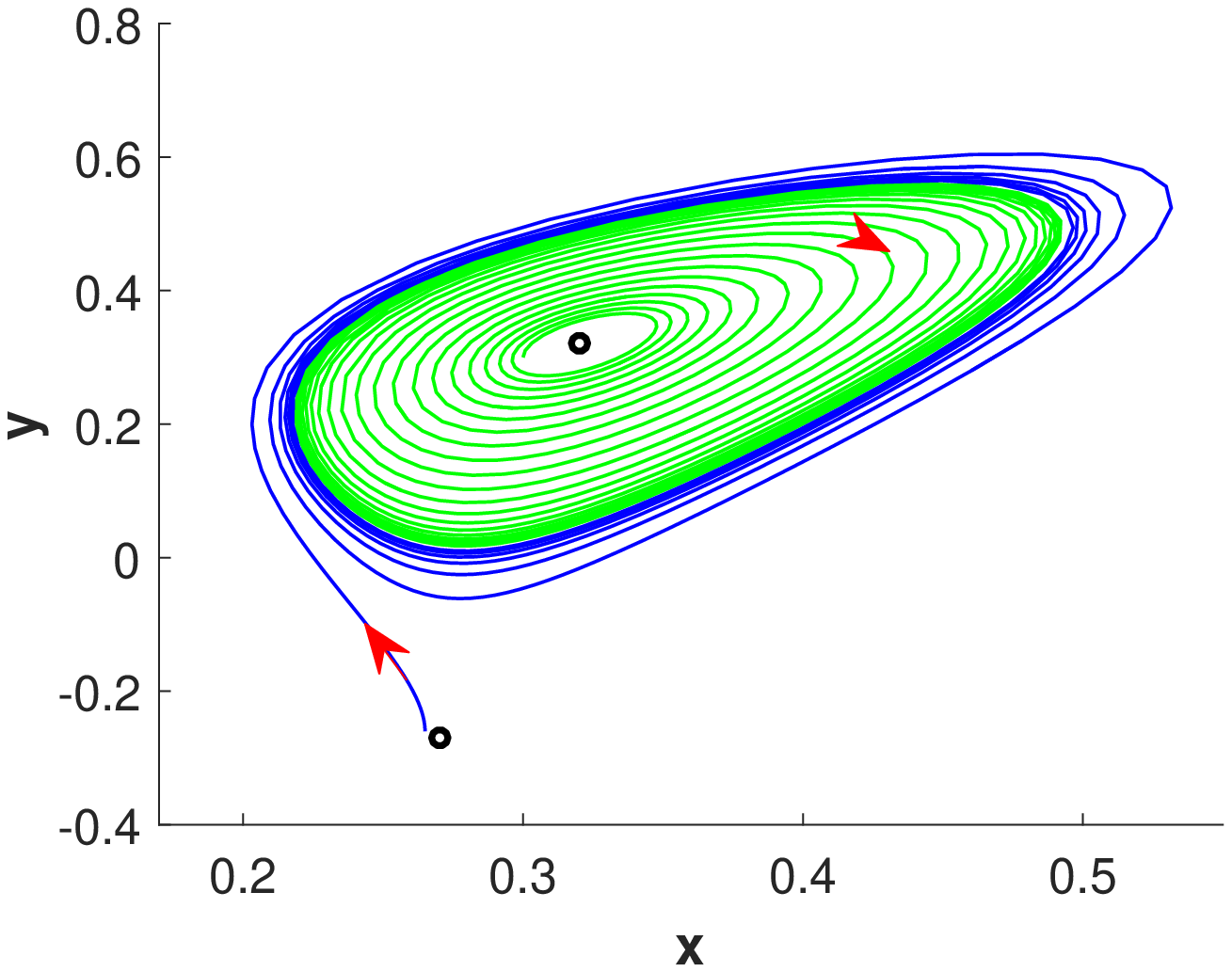}}
\end{center}
\vspace{-0.280 in}
\caption{Controlled phase portraits \ref{UncFig1}-\ref{UncFig5} corresponding with regions 1-5 partitioned by the estimated controller sets in Figure \ref{fIg61} for uncontrollable linearization case of \eqref{BfControlr2s2}}\label{FIg6}
\vspace{-0.100 in}
\end{figure}

\section{Chua system }\label{sec4}

The main goal in this section is to illustrate how controlled Chua system experiences a rich list of bifurcation scenarios by appropriately choosing small controller coefficients.
There have been an extensive literature on the dynamics study of Chua circuit systems; \eg see \cite{Puebla,YangZhao,Zhao}. For example, 
Zhao  et. al. \cite{Zhao} investigated Chua systems for Hopf and generalized Hopf bifurcations. Yang and Zhao \cite{YangZhao} considered the modifed Chua's circuit system with a delayed feedback: a delayed system undergoes Hopf and Hopf-zero bifurcations and stabilized the system through either a stable periodic orbit or a stable equilibrium. Puebla  et. al.  \cite{Puebla} implemented a linear PI compensator for a tracking control. Chua circuit is an electrical circuit and experiences Bogdanov-Takens, Hopf-zero and Hopf bifurcations.
Consider the controlled Chua system
\ba\label{chua}
&\dot{x}=\alpha(y-a x^3- c x),\quad \dot{y}= x-y+z,\quad \dot{z}=-\beta y+u, \quad u=\nu_0+ \nu_1 x+\nu_2 y+\nu_3 x y,&
\ea where \(u\) stands for a linear state-feedback input with small gain parameters \(\nu_i\) for \(i=0, 1, 2, 3\). State variables \(x\) and \(y\) represent the voltage of the capacitor and \(z\) is the electric current in the inductor. The uncontrolled system \((\nu_i=0 \hbox{ for } i\leq 4)\) is \(\mathbb{Z}_2\)-equivariant whose symmetry is given by the reflection \((x,y,z)\longrightarrow (-x,-y,-z)\). Uncontrolled Chua circuit system has three meaningful equilibria: the origin and \(e_{\pm}: \left(x_{\pm}, y, z_{\pm} \right)=\left(\pm \sqrt{-\frac{c}{a}}, 0, \mp \sqrt{-\frac{c}{a}}\right).\) The origin experiences Hopf and Hopf-zero bifurcations when \((\beta=-\alpha (c-1)(1+\alpha c), 0<c<1)\) and \((\beta=0, \alpha=-\frac{1}{c}, 0<c<1)\), respectively. The equilibria \(e_{\pm}\) undergo Hopf and Hopf-zero bifurcations for \((\beta=\alpha (2c+1) (1-2\alpha c), -\frac{1}{2}<c<0)\) and \((\beta=0, \alpha=\frac{1}{2c}, -\frac{1}{2}<c<0)\). These four singularities are not treated here; see \cite{YangZhao,Puebla}. We remark that our Hopf bifurcations here are caused by controller coefficient choices and thus, they are essentially different from Hopf singular cases caused by the system's choices for \(\alpha, \beta\) and \(c\).

\begin{prop}\label{ChuaProp}
When \(c=0,\) the origin is the only equilibrium of the system. System \eqref{chua} has a \(\mathbb{Z}_2\)-equivariant Bogdanov-Takens singularity at the origin for \(\beta=\alpha\) and \(c=0\).
\end{prop}
\begin{proof}
Jacobian matrix associated with \eqref{chua} at \((0, 0)\) for \(u=0\) is \(J:=[-\alpha c, \alpha, 0; 1 , -1 ,1; 0, -\beta, 0 ].\)
Eigenvalues for \(c=0\) are \(0\) and \(-\frac{1\pm\sqrt{1+4\alpha-4\beta}}{2}\). Thus for \(\alpha=\beta,\) we have a none semi-simple double zero eigenvalue.
\end{proof}

\begin{thm}[Controlled bifurcations] Consider controlled Chua differential system \eqref{chua} for \(\beta=\alpha\) and \(c=0\). Then by varying small controller coefficients \(\nu_i\), the controlled system undergoes a pitchfork bifurcation, three Hopf bifurcations and two homoclinic bifurcations.
\end{thm}
\bpr By Proposition \ref{ChuaProp}, there is a two-dimensional invariant center manifold \(\mathscr{M}\) for Chua system \eqref{chua}.
The reduction of \eqref{chua} on center manifold \(\mathscr{M}\) and then, primary shift of coordinates give rise to \(\dot{x}= v_1+ a\alpha(\alpha-1)^3 x^3+3 a\alpha^2 (\alpha-1)^2 x^2y+a\alpha^4y^3\) and \(\dot{y}= -x+\nu_3 x y+v_2,\) for
\begin{scriptsize}
\begin{eqnarray*}\nonumber
&\scalebox{0.9}{${v_1:=(3a\nu_0^2-6a\alpha\nu_0^2-\alpha^2\nu_1^2-\nu_1\alpha-\nu_3\nu_0+\nu_1\nu_2\alpha+3\nu_0^2
a\alpha^4-6a\alpha^3\nu_0^2+\nu_3\nu_0\alpha+9a\alpha^2\nu_0^2)y
+3a\alpha^2(-\alpha\nu_0+\nu_0\alpha^2+\nu_0) y^2}$}&\\\nonumber
&\scalebox{0.9}{${+(\alpha-1)(3a\alpha^3\nu_0 -6a\alpha^2\nu_0+6\nu_0 a\alpha-3a\nu_0 +\nu_3)x^2+\alpha(6a\alpha^3\nu_0-12a\alpha^2\nu_0+\nu_3-6\nu_0 a+12\nu_0 a\alpha) x y+3 a\alpha^3(\alpha-1) xy^2}$}&\\\nonumber
&\scalebox{0.95}{${+\frac{2\nu_3\nu_0+9a\alpha\nu_0^2-3a\nu_0^2-15a\alpha^2\nu_0^2+15a\alpha^3\nu_0^2-\nu_2^2\alpha+\nu_2\alpha-\nu_1\nu_2\alpha-2\alpha^3\nu_1^2-\nu_1\alpha^2+3\nu_1\nu_2\alpha^2
+3a\alpha^5\nu_0^2-9\nu_0^2a\alpha^4+2\nu_3\nu_0\alpha^2-3\nu_3\nu_0\alpha+\alpha^2\nu_1^2+\nu_1\alpha}{\alpha}x,}$}&
\\\nonumber
&\scalebox{0.95}{${v_2:=\frac{\nu_3(\alpha-1)}{\alpha} x^2-\frac{\alpha^2\nu_1^2-2\alpha^3\nu_1^2+3\alpha^2\nu_1\nu_2-\alpha\nu_1\nu_2-\alpha\nu_2^2
+2\alpha^2\nu_3\nu_0-3\alpha\nu_3\nu_0+2\nu_3\nu_0-\alpha^2\nu_1+\alpha\nu_1
+\alpha\nu_2}{\alpha^2}x-\frac{\alpha^2\nu_1^2-\nu_3\nu_0\alpha+\nu_3\nu_0-\nu_1\nu_2\alpha+\nu_1\alpha}{\alpha} y.}$}&
\end{eqnarray*}
\end{scriptsize} The first level normal form is obtained through a finite sequence of flow time-one maps generated by the initial value problems \(\dot{x}\frac{\partial}{\partial x}+ \dot{y}\frac{\partial}{\partial y}= T_i,\) \(x(0, \nu, X, Y)= X,\) and \( y(0, \nu, X, Y)= Y\) for \(i=1, 2, 3.\) Further, we assign each homogenous monomial vector field with a grade equal to its degree minus one plus two times degree of its parameters, \eg \(d(x^2y^3\nu_1{\nu_2}^3)= 2+3+2+6=13.\) The image of homological operator \(L^k: \LST_k\rightarrow \LST_k\) defined by \(L^k(w):= [w, -x \frac{\partial}{\partial y}]\) can be used to simplify the system. Hence, we consider the Lie brackets given by
\begin{eqnarray}\nonumber
&\left[y \frac{\partial}{\partial x},-x \frac{\partial}{\partial y}\right]=x \frac{\partial}{\partial x}-y \frac{\partial}{\partial y}, \left[\frac{1}{2}x \frac{\partial}{\partial x}-\frac{1}{2}y \frac{\partial}{\partial y},-x \frac{\partial}{\partial y}\right]=-x \frac{\partial}{\partial y}, \left[y^3 \frac{\partial}{\partial x},-x \frac{\partial}{\partial y}\right]=3x y^2 \frac{\partial}{\partial x}-y^3 \frac{\partial}{\partial y}, &\\\label{Lie1}
& \left[\frac{1}{2}x^2y \frac{\partial}{\partial x}-\frac{1}{2}xy^2 \frac{\partial}{\partial y},-x \frac{\partial}{\partial y}\right]=\frac{1}{2}x^3 \frac{\partial}{\partial x}-\frac{3}{2}x^2y \frac{\partial}{\partial y}, \left[x^2y \frac{\partial}{\partial x}+xy^2 \frac{\partial}{\partial y},-x \frac{\partial}{\partial y}\right]=x^3 \frac{\partial}{\partial x}+x^2y \frac{\partial}{\partial y}, &\\\nonumber
&\left[\frac{3}{4}xy^2 \frac{\partial}{\partial x}-\frac{1}{4}y^3 \frac{\partial}{\partial y},-x \frac{\partial}{\partial y}\right]=\frac{3}{2}x^2y \frac{\partial}{\partial x}-\frac{3}{2}xy^2 \frac{\partial}{\partial y}, \left[x y^2 \frac{\partial}{\partial x}+y^3 \frac{\partial}{\partial y},-x \frac{\partial}{\partial y}\right]=2x^2 y \frac{\partial}{\partial x}+2xy^2 \frac{\partial}{\partial y}.&
\end{eqnarray} Hence, terms of the form \(x^3 \frac{\partial}{\partial x}, x^2y \frac{\partial}{\partial y},\) \(x^2 y \frac{\partial}{\partial x}, xy^2 \frac{\partial}{\partial y}\) and parametric terms associated with \(x \frac{\partial}{\partial y}\) can be simplified from the system. However, we can choose between \(x y^2 \frac{\partial}{\partial x}\) or \(y^3 \frac{\partial}{\partial y}\) to simplify from the system. This is also true for parametric terms \(x \frac{\partial}{\partial x}\) and \(y \frac{\partial}{\partial y},\) where only one of them can be simplified from the system. Given the Lie brackets in equation \eqref{Lie1} and the grading function, we choose \(T_1\) to simplify terms of grade 2 as follows
\begin{eqnarray*}
&T_1^x=\left(\frac{1}{2}\alpha\nu_1-\nu_1-\frac{1}{2}\nu_2 \right) y-
{\frac { \left( \alpha\nu_1-\nu_1-\nu_2 \right) x}{2\alpha}}-a\alpha\, \left( \alpha-1 \right) ^3{x}^{2}y-\frac{3a{\alpha}
^{2} \left( \alpha-1 \right) ^{2}xy^{2}}{2}-\frac{3a{\alpha}^3 \left( \alpha-1 \right) y^3}{4},&\\\nonumber
&T_1^y={\frac { \left( \alpha\nu_1-\nu_1-\nu_2 \right) y}{2\alpha}}-\frac{a\alpha\, \left( \alpha-1 \right) ^3xy^{2}}{2}-\frac{a{\alpha}^{2}\left(\alpha-1\right) ^{2}y^3}{2}.&
\end{eqnarray*}
The updated system is given by \(\dot{X}\frac{\partial}{\partial X}+ \dot{Y}\frac{\partial}{\partial Y} = \exp {\rm ad}_{T_1^x\frac{\partial}{\partial x}+ T_1^y\frac{\partial}{\partial y}}(v_1\frac{\partial}{\partial x}+v_2\frac{\partial}{\partial y})\) whose grade two terms are associated with
\begin{eqnarray*}
&\left(a\alpha^4 Y^3-\alpha\nu_1 Y-\frac{\alpha\nu_1-\nu_2}{2} X
+\frac{3a\alpha^3(\alpha-1)}{4}XY^2\right) \frac{\partial}{\partial X}+\left(\frac{3a\alpha^3(\alpha-1)}{4}Y^3-\frac{\left(\alpha\nu_1-\nu_2\right)}{2}Y\right)\frac{\partial}{\partial Y}&
\end{eqnarray*} For simplicity of notations, we replace \((X, Y)\) with \((x, y).\)
Due to the Lie brackets
\begin{eqnarray*}
&\left[\frac{1}{3}x^2 \frac{\partial}{\partial x}-\frac{2}{3}xy \frac{\partial}{\partial y},-x \frac{\partial}{\partial y}\right]=-x^2 \frac{\partial}{\partial y}, \left[\frac{2}{3}x y \frac{\partial}{\partial x}-\frac{1}{3}y^2 \frac{\partial}{\partial y},-x \frac{\partial}{\partial y}\right]=\frac{2}{3}x^2\frac{\partial}{\partial x}-\frac{4}{3}xy \frac{\partial}{\partial y}&\\\nonumber
&\left[x y \frac{\partial}{\partial x}+y^2 \frac{\partial}{\partial y},-x \frac{\partial}{\partial y}\right]=x^2\frac{\partial}{\partial x}+x y \frac{\partial}{\partial y}, \left[y^2 \frac{\partial}{\partial x},-x \frac{\partial}{\partial y}\right]=2x y\frac{\partial}{\partial x}-y^2 \frac{\partial}{\partial y},&
\end{eqnarray*} we can eliminate all of the terms \(x^2 \frac{\partial}{\partial y},\) \(xy \frac{\partial}{\partial y},\) \(x^2\frac{\partial}{\partial x},\) while only one of \(x y\frac{\partial}{\partial x}\) and \(y^2 \frac{\partial}{\partial y}\) can be simplified.
The calculations show that only the following terms of grade 3 remain in the system
\begin{eqnarray*}
&\Big(\frac{3\alpha^3a\nu_0\left(1+\alpha^2\right) y^2-\nu_0\left(\nu_2-2\alpha\nu_2+\nu_1
-\alpha\nu_1+2\alpha^2\nu_1\right)}{2\alpha}+\frac{\alpha \left( 3\nu_0a \left( \alpha-1 \right)  \left( 1+{\alpha}^{2} \right) +\nu_3 \right) xy}{3}\Big)\frac{\partial}{\partial x}+\frac{\alpha \left( 3a \nu_0\left( \alpha-1 \right)  \left( {\alpha}^{2}+1 \right) +\nu_3 \right) y^2}{3}\frac{\partial}{\partial y}&
\end{eqnarray*}
Let \(A^{-1}_{r-i}(y):=y^{r-i+1} \frac{\partial}{\partial x}.\) Then,
\begin{eqnarray*}
&A^{-1}_r(y+f(\nu))=\sum_{i=0}^{r+1}{r+1 \choose i} f(\nu)^i y^{r-i+1} \frac{\partial}{\partial x}. &
\end{eqnarray*} Thus, term \(\frac{3}{2}\alpha^2a(1+\alpha^2)\nu_0 y^2 \frac{\partial}{\partial x}\) can be cancelled by replacing \(y\) with \(y+f(\nu)\) where  \(f(\nu)=-\frac{1}{2}\alpha^2a(1+\alpha^2)\nu_0\). Thereby, the system is normalized to \eqref{Eq03}, where \(\mu_0=\frac{9}{8}\alpha\nu_1\nu_0-\frac{3}{32}\nu_1\nu_0-\frac{33}{32\alpha}\nu_1\nu_0
-\frac{37}{32}\nu_2\nu_0-\frac{11}{32\alpha}\nu_2\nu_0-\nu_0,\)
\begin{scriptsize}
\bas
&\mu_1=-\frac{5\alpha^2\nu_1^2}{16}-\frac{3\nu_1^2}{4}-\frac{3\alpha\nu_1^2}{16}-\frac{3\alpha\nu_1\nu_2}{16}
-\frac{5\nu_1\nu_2}{16}-\frac{\nu_2^2}{4}-\frac{1047}{5120} a\alpha^4\nu_0^2 +\frac{2583}{1280}a\alpha^3\nu_0^2
-\frac{1047}{5120}a\nu_0^2-\frac{1857}{512}a\alpha^2\nu_0^2+\frac{2583}{1280}a\alpha\nu_0^2-\alpha\nu_1,\qquad&\\
&\mu_2=-\frac{49 \alpha^2+3 \alpha+12}{64}\nu_1^2+\frac{27 \alpha^2-\alpha+6}{32 \alpha}\nu_1\nu_2-\frac{5(\alpha-1){\nu_2}^2}{64\alpha}+\frac{(63\alpha-31) (\alpha-1)\nu_3\nu_0}{64\alpha}+\frac{3a(2829\alpha^2-3226\alpha+2829)(\alpha-1)^3{\nu_0}^2}{10240\alpha}-\frac{\alpha\nu_1}{2}+\frac{\nu_2}{2},&\\
&\mu_3=\frac{1}{3}\alpha\nu_3-\frac{9}{16}\alpha^3a\nu_0+\frac{9}{16} a\alpha^2\nu_0-\frac{3}{16}a\alpha \nu_0+\frac{3}{16}a\alpha^4\nu_0, \qquad\qquad\hbox{ and }\qquad\qquad a_1=a\alpha^4,\quad b_0=\frac{3}{4}a(\alpha-1)\alpha^3.&
\eas\end{scriptsize}
For \(\nu_0=0,\) \(\mu_0=0.\) Thus, we follow Proposition \ref{Prop1}, Theorem \ref{thm3}, and Theorem \ref{Hom0} for deriving symbolic estimates of \(T_{P},\) \(T_H,\) \(T_{H\pm}\) and \(T_{HmC{{\pm}}}.\) For simplicity of the formulaes, we take  \(\nu_3=0.3,\) \(\alpha:=0.8,\) \(a:=1.\) Then, we obtain the following equations:
\begin{scriptsize}
\begin{eqnarray*}
&T_H=\left\{(\nu_1, \nu_2)\big|\, \nu_1=\frac{5}{4}\nu_2\right\}, \quad T_{HmC}=\left\{(\nu_1,\nu_2)\big|\, \nu_1=\frac {320}{1771}+\frac {35045}{28336}\nu_2-\frac {5}{28336}\sqrt{1048576+5286912\nu_2+13861929{\nu_2}^2}\right\},&\\\nonumber
&T_p=\left\{(\nu_1, \nu_2)\big|\, \nu_1=0\right\}, \quad T_{HmC_{\pm}}=\left\{(\nu_1,\nu_2)\big|\, \pm\frac {4989\sqrt{10}}{1600000}\pi \nu_1\sqrt{\nu_1}\pm\frac{21\sqrt{10}}
{160000}\pi\sqrt{\nu_1}\nu_2-\frac{16}{25}\nu_1+\frac{1}{2}\nu_2\pm\frac{9\sqrt{10}}{1000}\pi\sqrt{\nu_1}=0\right\},&\\
&T_{H_{\pm}}=\Big\{(\nu_1, \nu_2)\big|\,\frac {5293162496}{30517578125}{\nu_2}^2-\frac {67273949184}{152587890625}\nu_1\nu_2\pm\frac {978767872\sqrt{5}}{30517578125}\sqrt{\nu_1}\nu_2-\frac {1841299456}{30517578125}\nu_1&\\&
+\frac {2097152}{48828125}\nu_2\pm\frac {199753728\sqrt{5}}{30517578125}\sqrt{\nu_1}=0\Big\}.&
\end{eqnarray*}
\end{scriptsize} These are derived from equations \eqref{SNr2s2nu10},  \eqref{eta0}, and \eqref{HmCpm}, respectively.
\epr

\begin{figure}[t!]
\begin{center}
\subfigure[\(\mathbb{Z}_2\)-breaking controller sets for controlled Chua system \eqref{chua} when \(\nu_3=0.3,\) \(\nu_0=0,\) \(\alpha:=0.8,\) and \(a:=1\) \label{fig71}]
{\includegraphics[width=.3\columnwidth,height=.24\columnwidth]{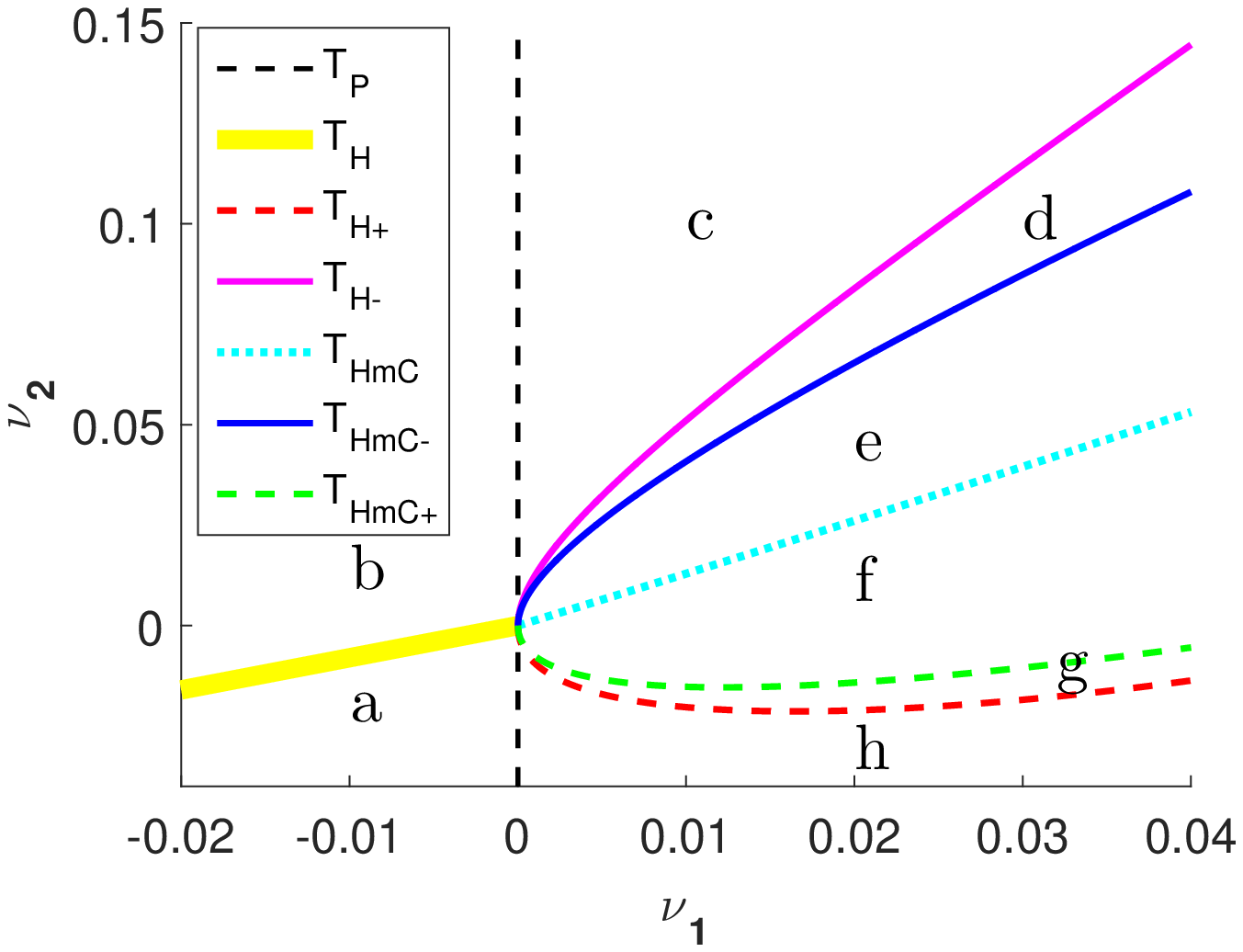}}\;
\subfigure[Trajectories and regulating controller \(u\) associated with region (a) for \(\nu_1=-0.01,\) \(\nu_2=-0.02.\) \label{Fig7a}]
{\includegraphics[width=.3\columnwidth,height=.24\columnwidth]{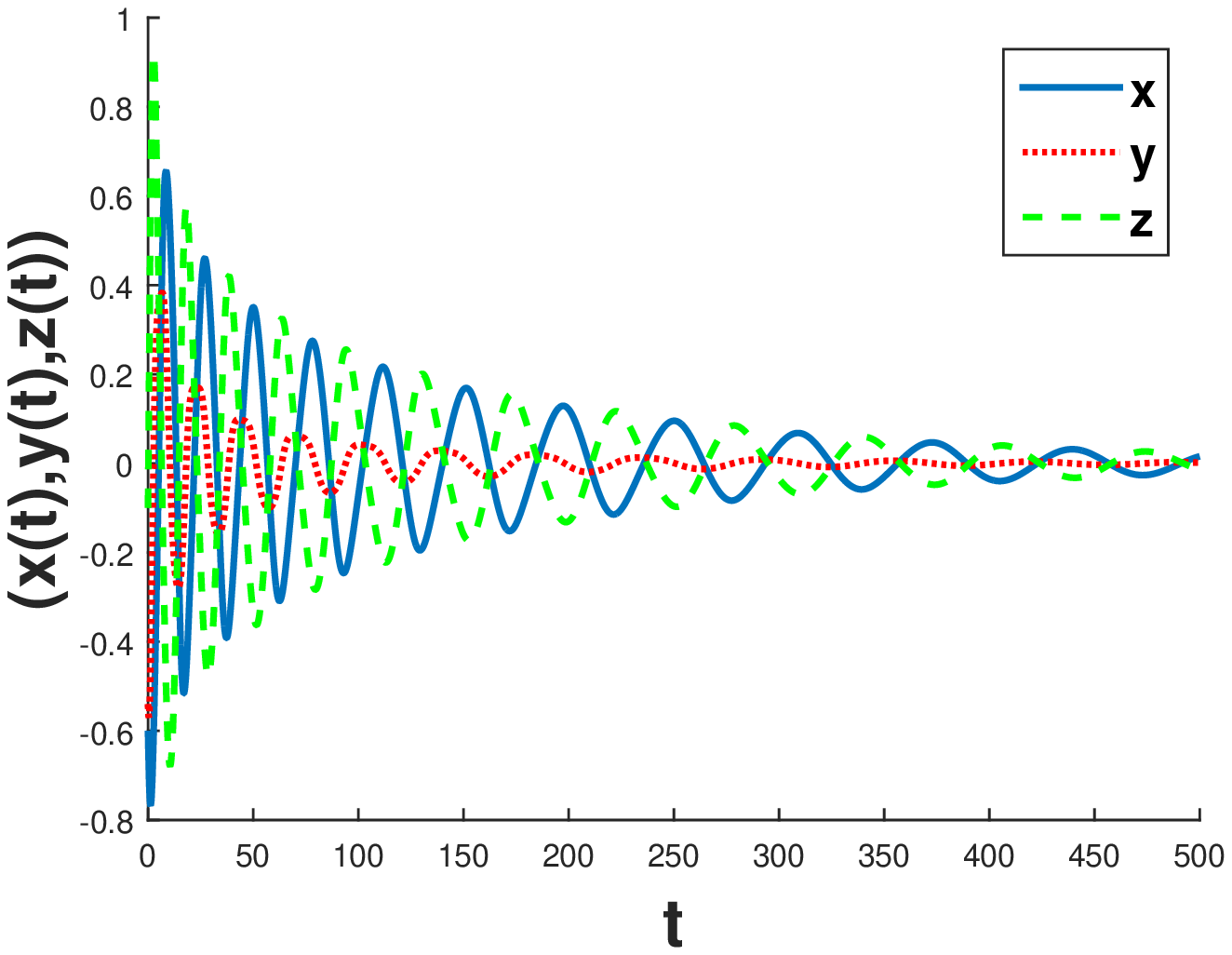}\,
\includegraphics[width=.3\columnwidth,height=.21\columnwidth]{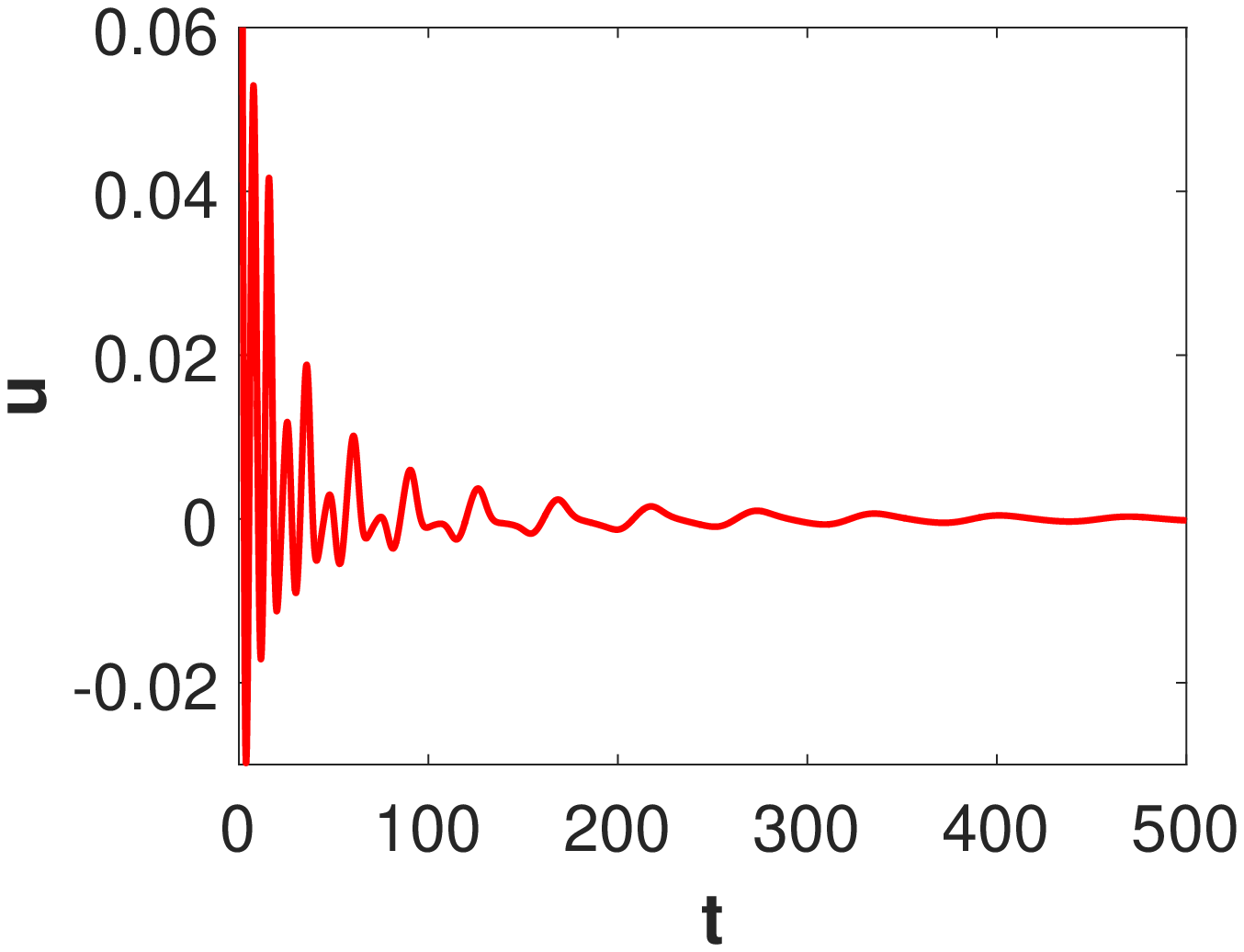}}\,
\subfigure[Trajectories and controller \(u\) for region (b) when \((\nu_1, \nu_2)=(-0.01, 0)\). \label{FigStab1b}]
{\includegraphics[width=.3\columnwidth,height=.24\columnwidth]{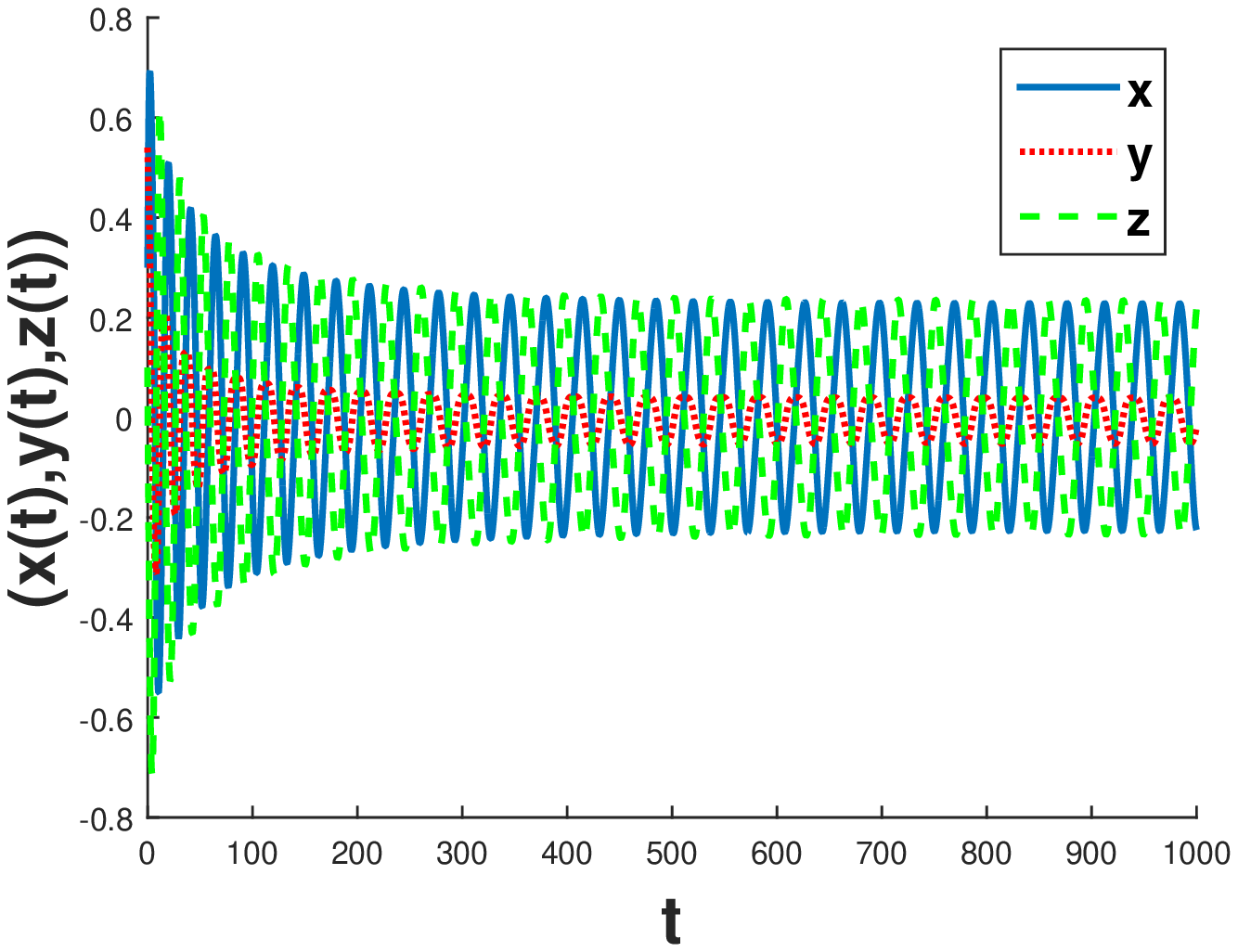}
\includegraphics[width=.3\columnwidth,height=.21\columnwidth]{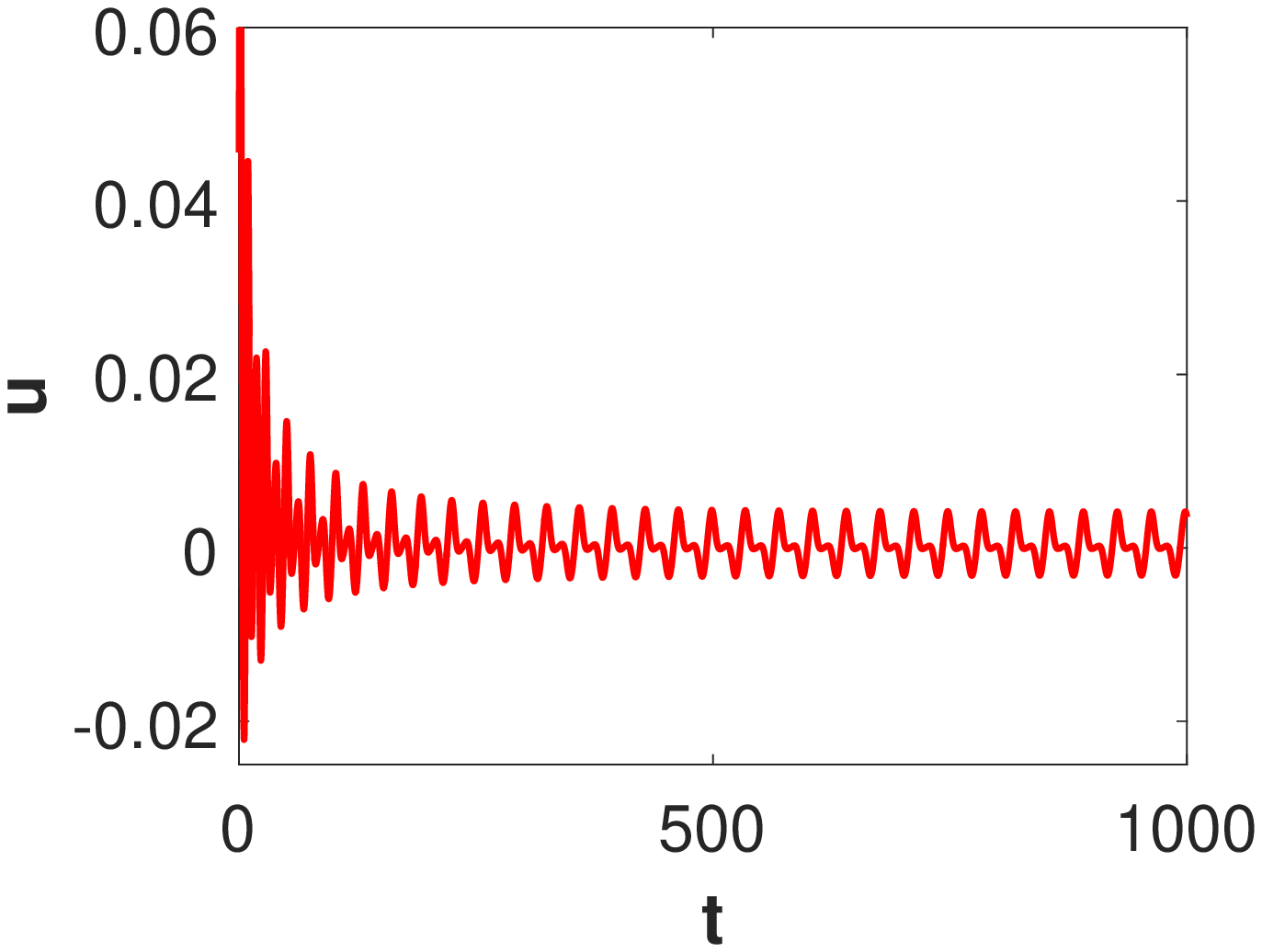}}
\subfigure[Region (b) for \((\nu_1, \nu_2)=(-0.02, 0).\)\label{FigStab2b}]
{\includegraphics[width=.3\columnwidth,height=.24\columnwidth]{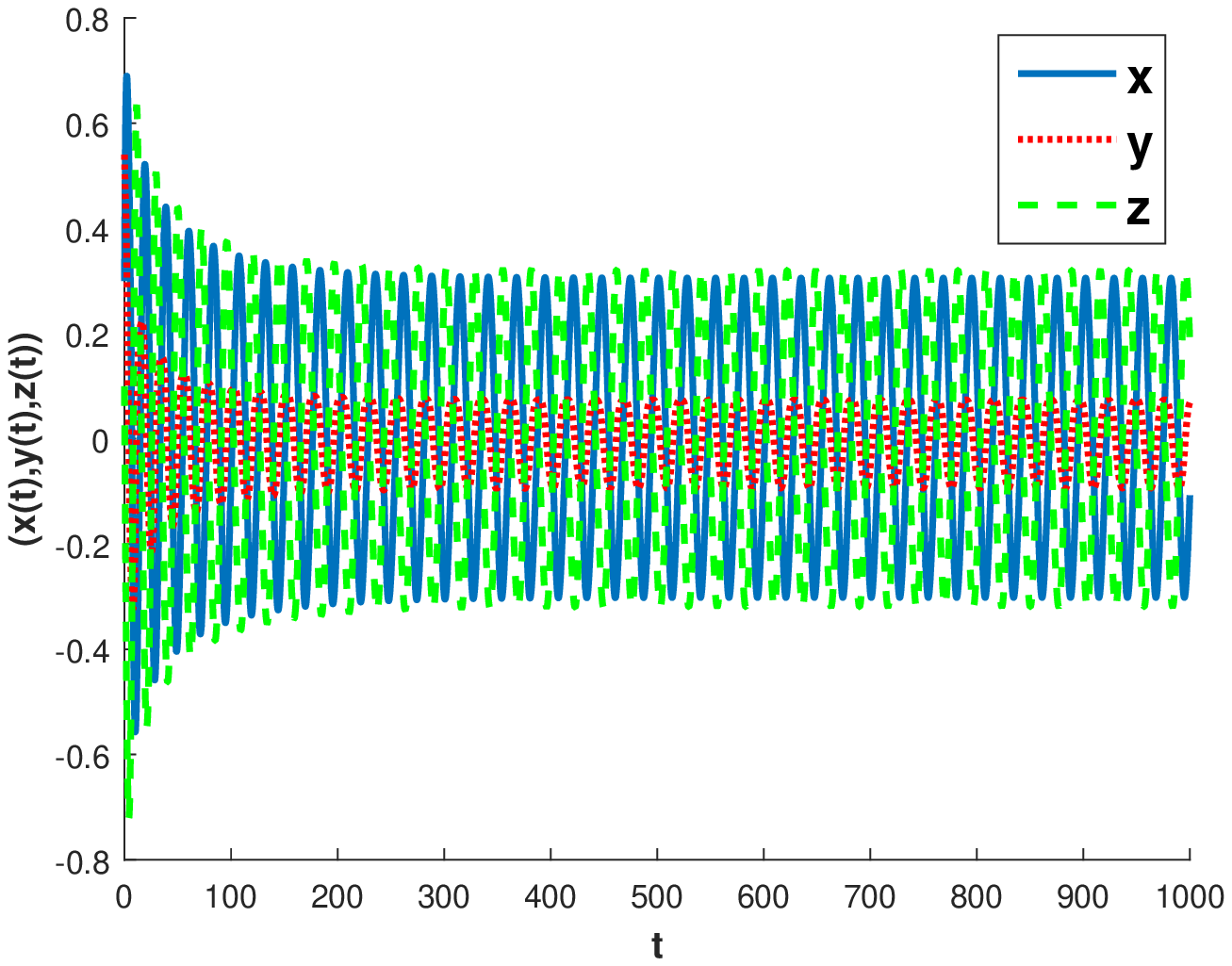}}\,
\end{center}
\vspace{-0.250 in}
\caption{Regularization of the origin using bifurcation control and stabilization via supercritical Hopf bifurcation; controller coefficients are from regions (a) and (b) in  Figure \ref{fig71} when \(\alpha:=0.8,\) \(a:=1,\) \(\nu_3=0.3,\) and \(\nu_0=0\).}\label{breaking2}
\vspace{-0.100 in}
\end{figure}

Let \(\nu_3=0.3\) and \(\nu_0=0\). Then, controller coefficient choices \((\nu_1, \nu_2)\) from regions (a)-(h) from Figure \ref{fig71} give rise to the controlled trajectories of Chua system \eqref{chua} in Figures \ref{Fig7a}-\ref{FigStab2b}, \ref{Fig7c}-\ref{Fig7f}, and \ref{Fig7g}-\ref{Fig7j}, respectively.

\subsection{Controller objectives}

Different bifurcation scenarios can be realised by appropriate choices of controller coefficients and these provide an effective approach to meet many desired control objectives. For example, we explain how these objectives can include feedback regularization, stabilization, amplitude size and frequency control for oscillations of the controlled differential Chua system.

\subsubsection*{Region (a): A feedback regularization by bifurcation control:}

Controller coefficient choices from region (a) in Figure \ref{fig71} turn the origin into a spiral sink.
Hence, the feedback control \(u\) in \eqref{chua} regularizes the origin and all trajectories converge to the origin; the origin is globally asymptotically stable. Figure \ref{Fig7a} depicts trajectories \(x(t),\) \(y(t),\) and \(z(t)\) for initial values \((-0.6 ,-0.54, -0.1)\) and controller coefficients \((\nu_1, \nu_2)=(-0.01, -0.02)\) chosen from region (a). These trajectories converge to the origin while the controller \(u\) is very small.

\subsubsection*{Region (b): Stabilization and oscillation control via a supercritical Hopf bifurcation:}

There is a stable limit cycle \(\mathscr{C}_0\) for region \(b\) and
trajectories converge to \(\mathscr{C}_0.\) By Proposition \ref{Prop1}, there is a supercritical Hopf bifurcation at the border between regions \(a\) and \(b,\) that is estimated by \(\nu_1=\frac{5}{4}\nu_2.\) This provides a stabilization approach for the system. Furthermore, the leading estimate for radius of the bifurcation limit cycle is given by \(\sqrt{\dfrac{4(\alpha\nu_1-\nu_2)}{3a(\alpha-1)\alpha^3}}\) while the leading estimate for its angular frequency is \(\sqrt{\alpha\nu_1}.\) Figures \ref{FigStab1b}-\ref{FigStab2b}  illustrate converging trajectories to the limit cycle \(\mathscr{C}_0\) from the initial values \((0.3, 0.54, 0.1)\) for controller coefficient choices \((\nu_1, \nu_2)\) from region (b). Since the angular frequency and amplitude size of oscillations are proportional to the angular frequency and radius of the bifurcation limit cycle, we can control the oscillating trajectories accordingly. Hence, an increase in \(\nu_1\) leads to an increase in angular frequency of oscillations while \(|\alpha\nu_1-\nu_2|\) is the deciding factor for the amplitude sizes of the oscillations. Therefore, an increase from \(|\nu_1|=0.01\) to  \(|\nu_1|=0.02\) gives rise to an amplification in both amplitude sizes and angular frequencies in figures \ref{FigStab1b}-\ref{FigStab2b}. Notice that the controller \(u\) in figure \ref{FigStab1b} is very small while it causes oscillations with moderate magnitudes.

\begin{figure}[h!]
\begin{center}
\subfigure[Region (c) for \((\nu_1, \nu_2)=(0.01, 0.1)\)\label{Fig7c}]
{\includegraphics[width=.24\columnwidth,height=.22\columnwidth]{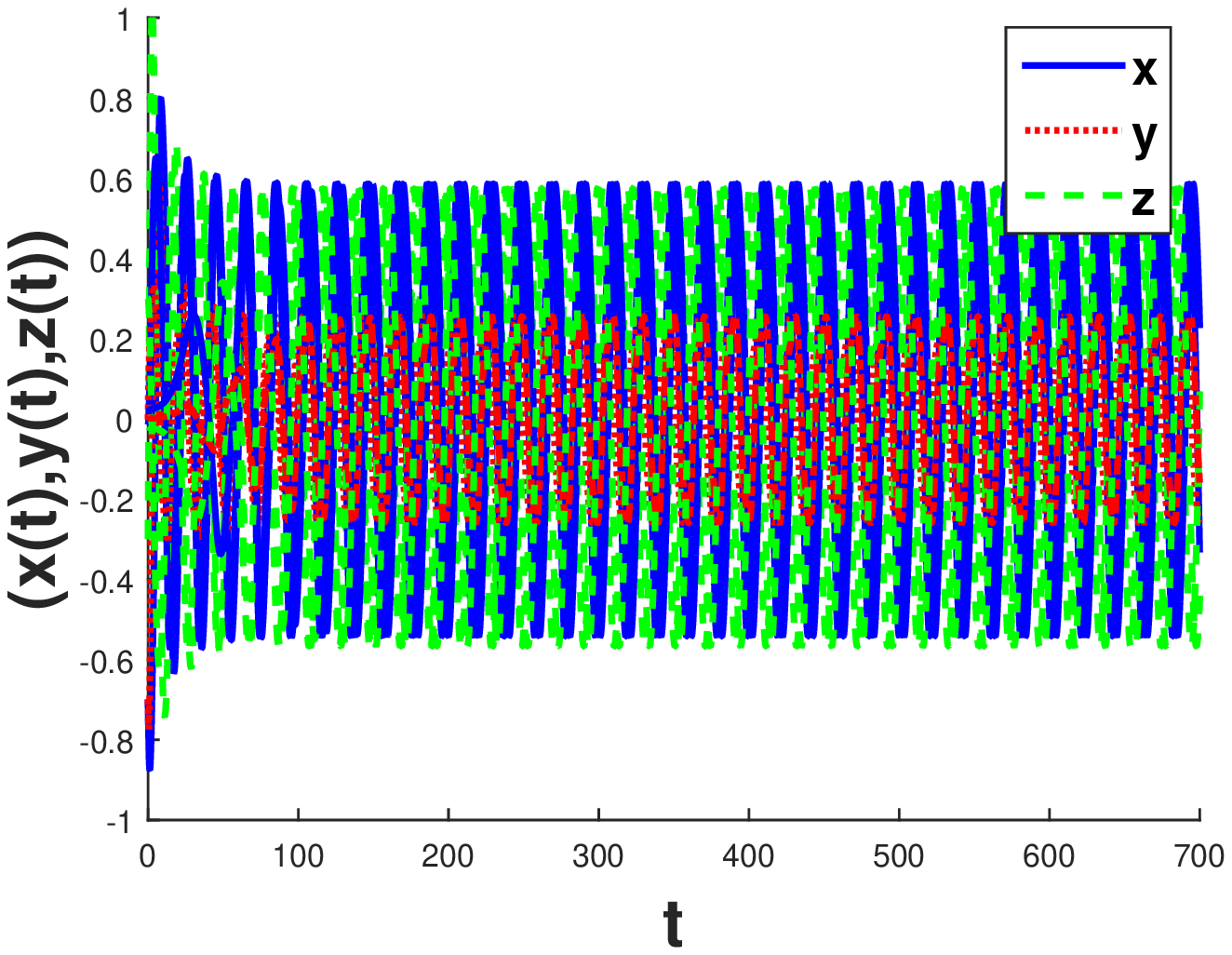}}\,
\subfigure[Region (d) for \((\nu_1, \nu_2)=(0.02, 0.068)\)\label{Fig7d}]
{\includegraphics[width=.24\columnwidth,height=.22\columnwidth]{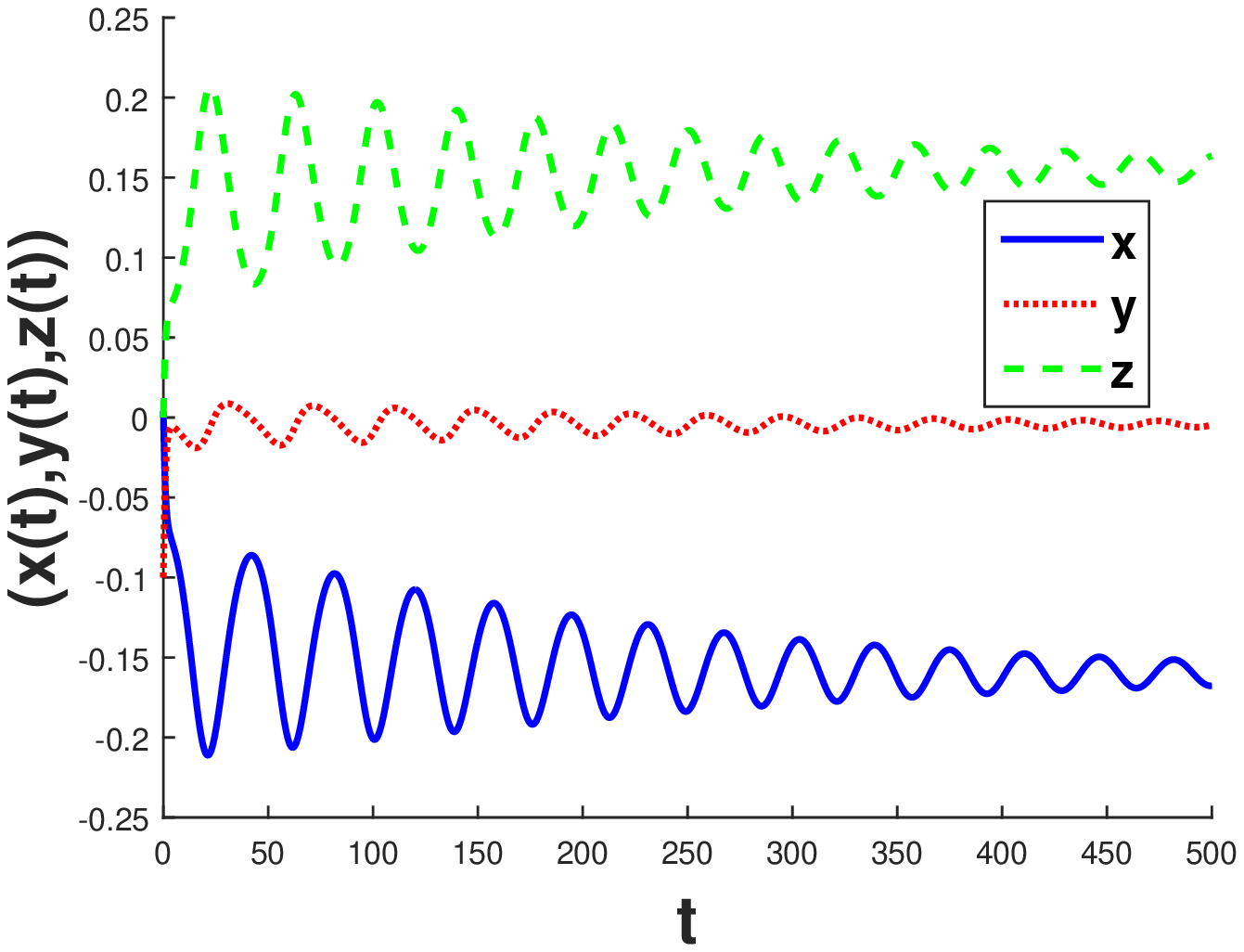}}\,
\subfigure[Region (e) for \((\nu_1, \nu_2)=(0.02, 0.03)\)\label{Fig7e}]
{\includegraphics[width=.24\columnwidth,height=.22\columnwidth]{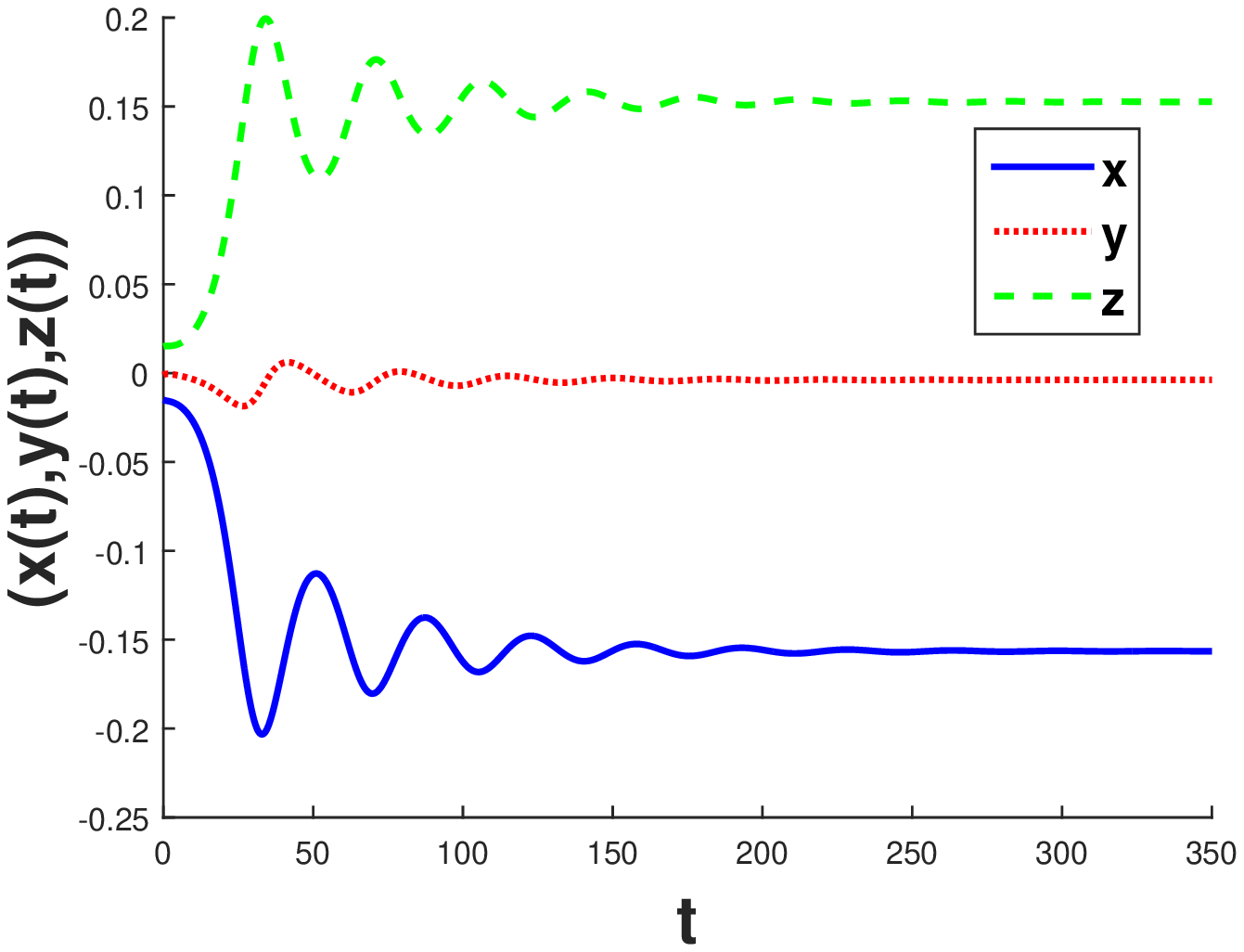}}\,
\subfigure[Region (f) for \((\nu_1, \nu_2)=(0.02, 0)\)\label{Fig7f}]
{\includegraphics[width=.23\columnwidth,height=.22\columnwidth]{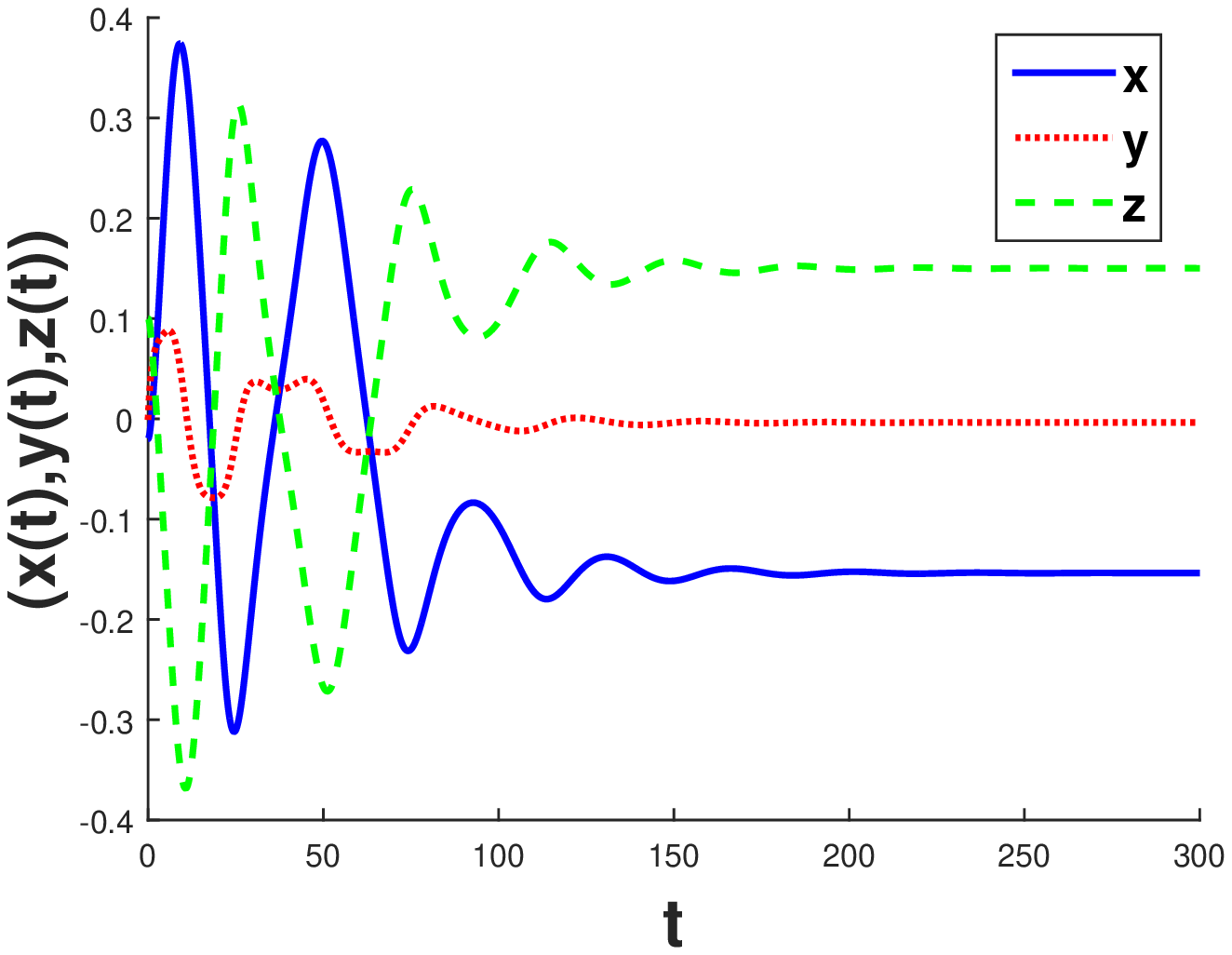}}\;\,
\end{center}
\vspace{-0.250 in}
\caption{Controlled trajectories \ref{Fig7c}-\ref{Fig7f} associated with regions (c)-(f) from Figure \ref{fig71} for \eqref{chua}, \(\alpha:=0.8,\) \(a:=1\), \(\nu_3=0.3,\) \(\nu_0=0\).}\label{breaking2}
\vspace{-0.100 in}
\end{figure}

\subsubsection*{Region (c): A pitchfork bifurcation of equilibria inside \(\mathscr{C}_0\) within center manifold \(\mathscr{M}\).}

There is a pitchfork bifurcation at border of regions (b) and (c). The system for controller coefficients from region (b) has a limit cycle \(\mathscr{C}_0.\) The limit cycle \(\mathscr{C}_0\) lives on the center manifold \(\mathscr{M}\) of Chua system and within the center manifold \(\mathscr{M}\), \(\mathscr{C}_0\) encircles the origin. As the controller coefficients cross the pitchfork variety \(T_p,\) two new equilibria \(e_\pm\) are born from the origin. Hence, for controller coefficients chosen from region (c), controlled Chua system \eqref{chua} has two spiral sources, and origin as a hyperbolic saddle. All these equilibria live inside the stable limit cycle \(\mathscr{C}_0\) within the center manifold. Trajectories in Figure \ref{Fig7c} converge to the asymptotically stable limit cycle \(\mathscr{C}_0\) for initial values \((-0.7, -0.7, -0.3)\) and \((\nu_1, \nu_2)=(0.01, 0.1).\) The oscillation frequency and amplitude size control of the oscillating trajectories are controlled by the management of the stable limit cycle \(\mathscr{C}_0\). Limit cycle \(\mathscr{C}_0\) is controllable as similar to the controller coefficient cases from region (b).

\subsubsection*{Region (d): A secondary subcritical Hopf bifurcation of \(\mathscr{C}^1_-\) inside \(\mathscr{C}_0\) within \(\mathscr{M}\).}

By Theorem \ref{thm3}, a secondary Hopf bifurcation of
limit cycles from \(e_-\) occurs inside the limit cycle \(\mathscr{C}_0\). Remark that equilibria \(e_\pm,\) the origin, and limit cycles \(\mathscr{C}^1_-\) and \(\mathscr{C}_0\) all live on the center manifold \(\mathscr{M}\). Hence for controller coefficient choices from region (d), the controlled system \eqref{chua} has a spiral source, a saddle, a spiral sink and two limit cycles. The leading estimates for angular frequency and radius of \(\mathscr{C}^1_-\) are given by \(\frac{\sqrt{2}}{4}(16\alpha\nu_1-(\nu_2-\alpha\nu_1)^2)\) and \(-\frac{8\sqrt{2}}{7a \alpha^4}\sqrt{\frac{14}{3(\alpha-1)}(\nu_2-\alpha\nu_1)-184 ((\nu_2-\alpha\nu_1)^2-4\alpha\nu_1)}-\frac{32}{7}\frac{\sqrt{4\alpha\nu_1-(\nu_2-\alpha\nu_1)^2}}{a_1},\) respectively. However, \(\mathscr{C}^1_-\) is an unstable limit cycle. Hence, amplitude size control of the oscillating dynamics is typically determined by the stable limit cycle \(\mathscr{C}_0\).
Figure \ref{Fig7d} represents controlled dynamics associated with region (d). Controlled trajectories for initial values \((0.004, -0.1,0)\) converge to \(e_{-}\) when controller coefficients \((\nu_1, \nu_2)=(0.02, 0.068)\) are taken from region (d).

\subsubsection*{Region (e): Disappearance of \(\mathscr{C}^1_-\) via homoclinic bifurcation at \(T_{HmC-}\).}

The limit cycle \(\mathscr{C}^1_-\) collides with the origin and constructs a
homoclinic cycle \({\Gamma}_-\) exactly when controller coefficients are taken from homoclinic controller set \(T_{HmC-};\) see  Theorem \ref{Hom0}. Therefore, both of \(\mathscr{C}^1_-\) and \({\Gamma}_-\) disappear when we take controller coefficients from region (e) in Figure \ref{fig71}. Hence, there exist a source, a saddle, a sink and an enlarged stable limit cycle \(\mathscr{C}_0.\) Trajectories in Figure \ref{Fig7e} for initial values \((-0.015,-0.001, 0.015)\) converge to \(e_{-}\) for \((\nu_1, \nu_2)=(0.02, 0.03)\) from region (e).

\subsubsection*{Controlled dynamics of region (f): \(\mathscr{C}_0\) disappears via homoclinic bifurcation at \(T_{HmC}\).}

The limit cycle \(\mathscr{C}_0\) collides with origin for
controller coefficients taken from transition variety \(T_{HmC}\). Hence, limit cycle \(\mathscr{C}_0\) disappears from the local dynamics of controlled Chua system when we take controller coefficients from region (f). Therefore, there are a source, a saddle, a sink and no limit cycle. Figure \ref{Fig7f} depicts a trajectory that  converges to \(e_{-}\) from initial values \((-0.02,-0.001, 0.1)\) for controllers \((\nu_1, \nu_2)=(0.02, 0)\) corresponding with region (f).  For controller coefficients from region (f), system \eqref{chua} has three equilibria and no limit cycle.

\begin{figure}[t!]
\begin{center}
\subfigure[Region (g): \(\nu_1=0.018,\) \(\nu_2=-0.016.\) Orbits converge to \(\mathscr{C}_+\) and \(e_{-}.\) \label{Fig7g}]
{\includegraphics[width=.24\columnwidth,height=.22\columnwidth]{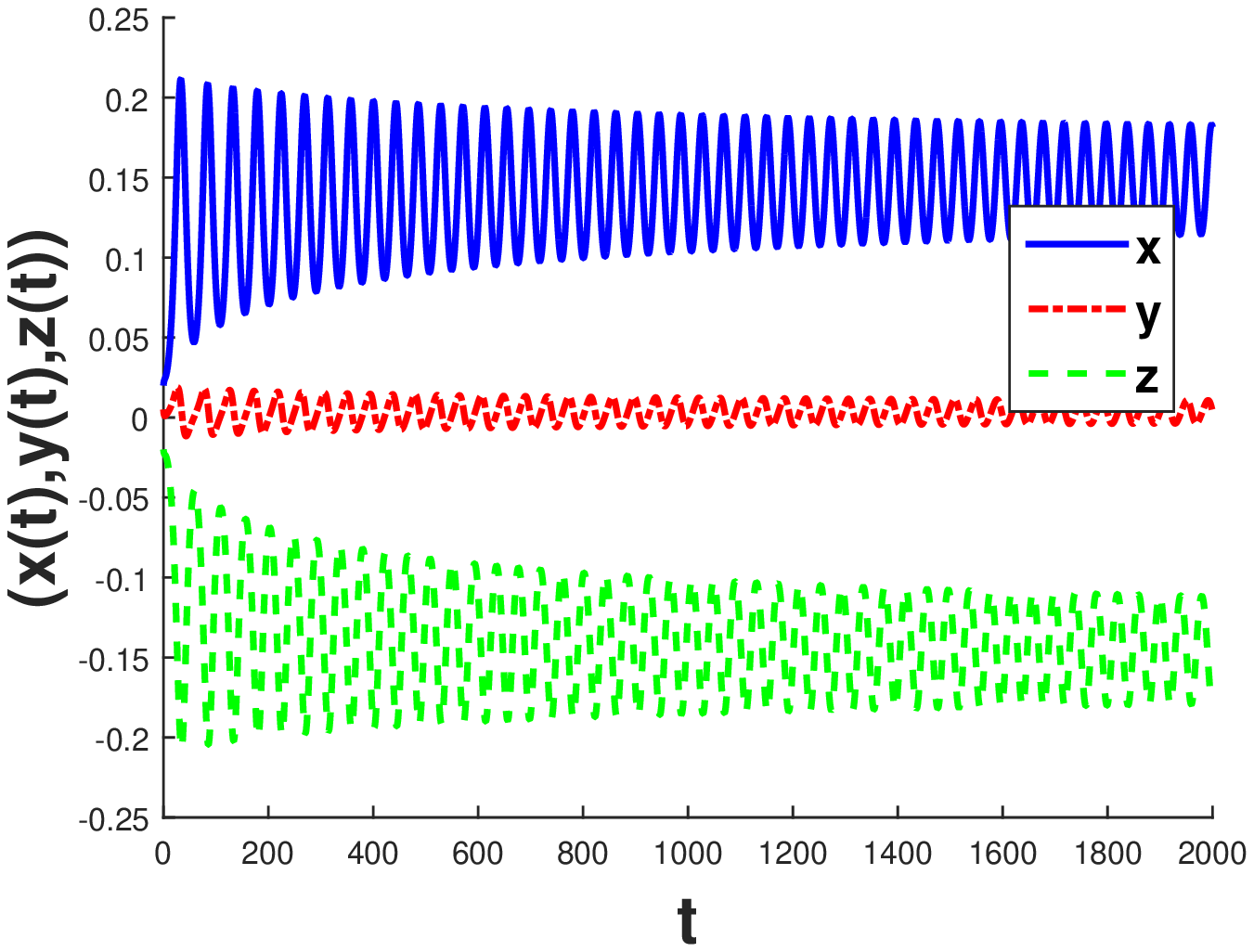}\!
\includegraphics[width=.24\columnwidth,height=.22\columnwidth]{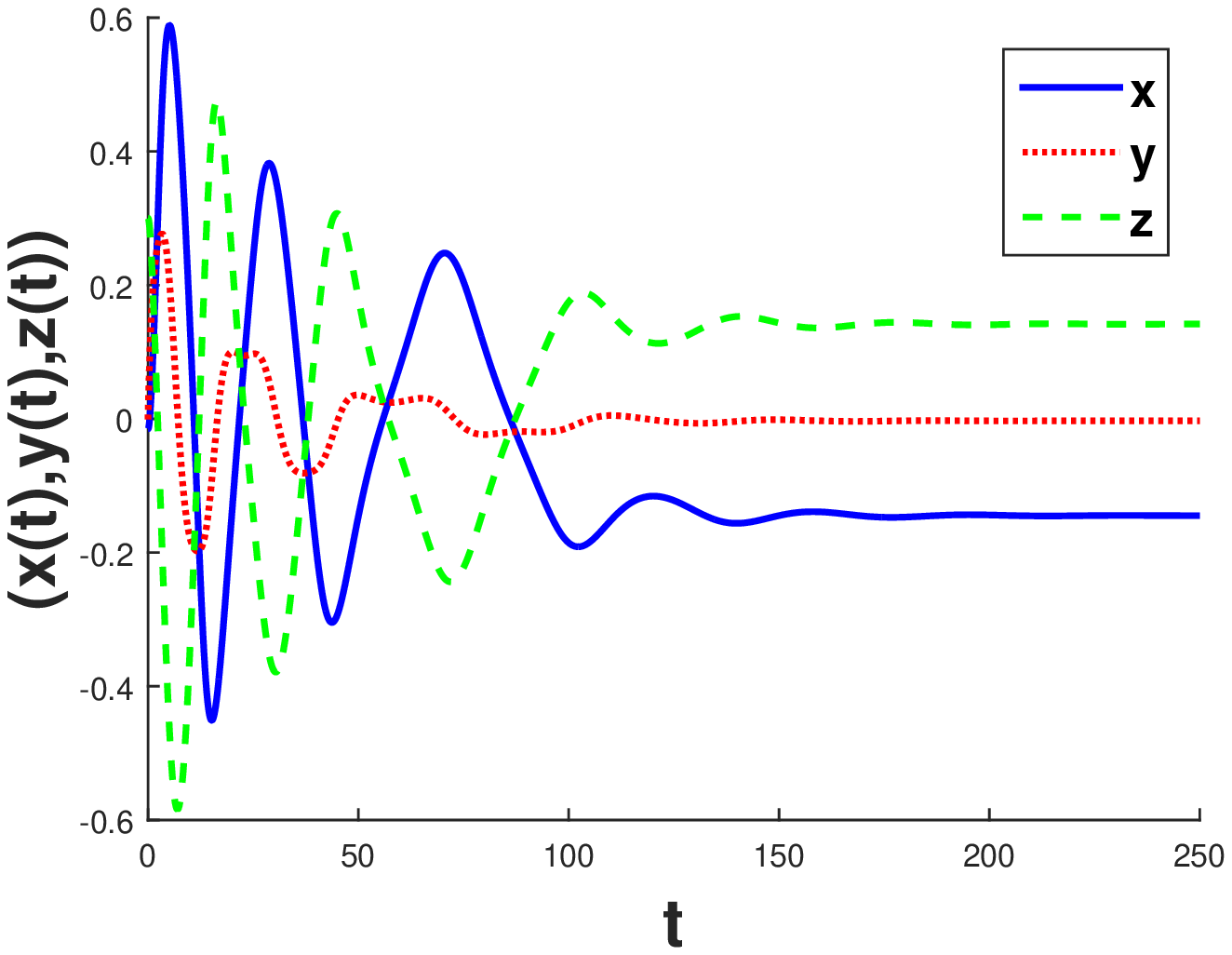}}\;\,
\subfigure[Region (h): \(\nu_1=0.02,\) \(\nu_2=-0.025.\) Trajectories converge to \(e_\pm.\) \label{Fig7j}]
{\includegraphics[width=.24\columnwidth,height=.22\columnwidth]{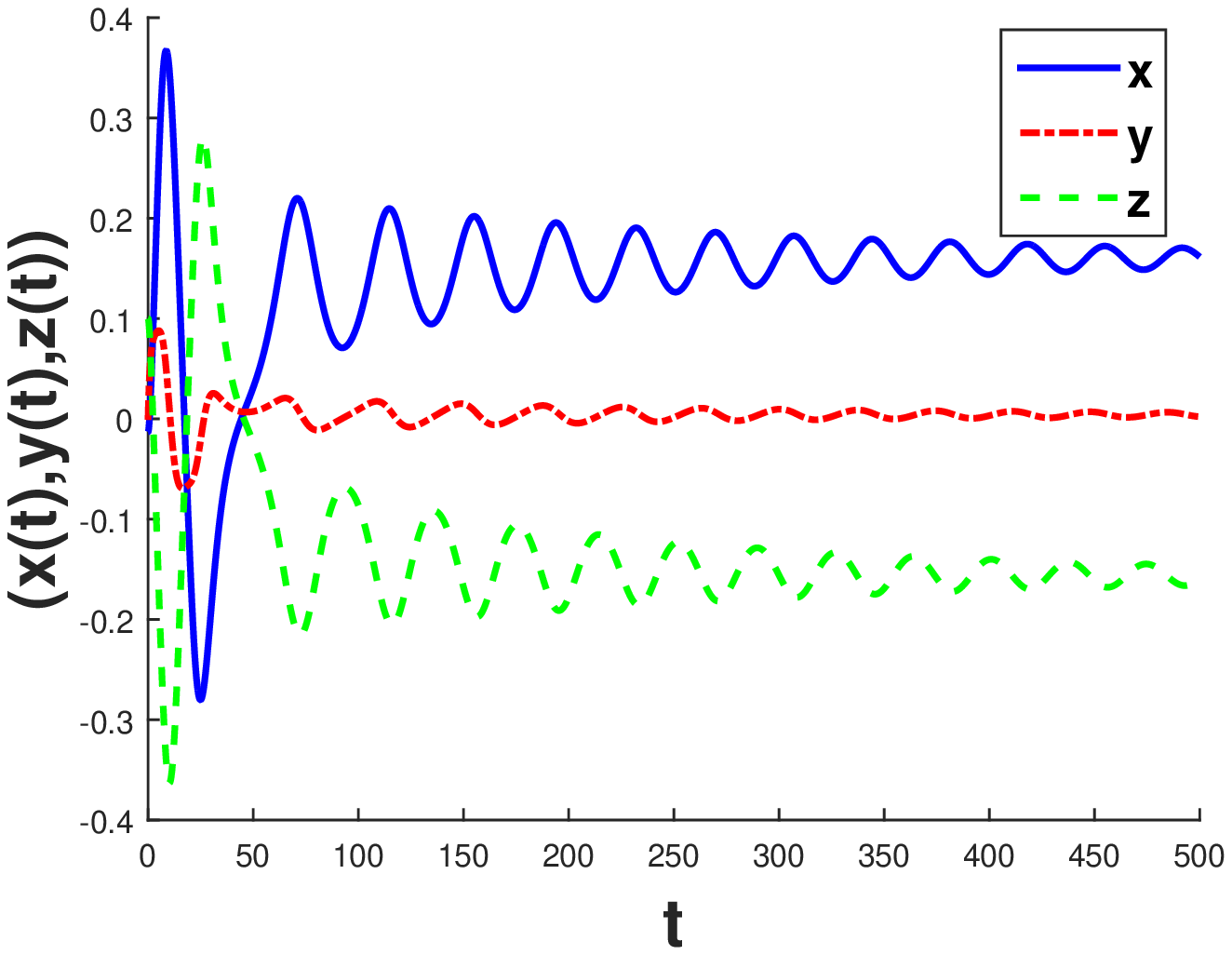}\!
\includegraphics[width=.24\columnwidth,height=.22\columnwidth]{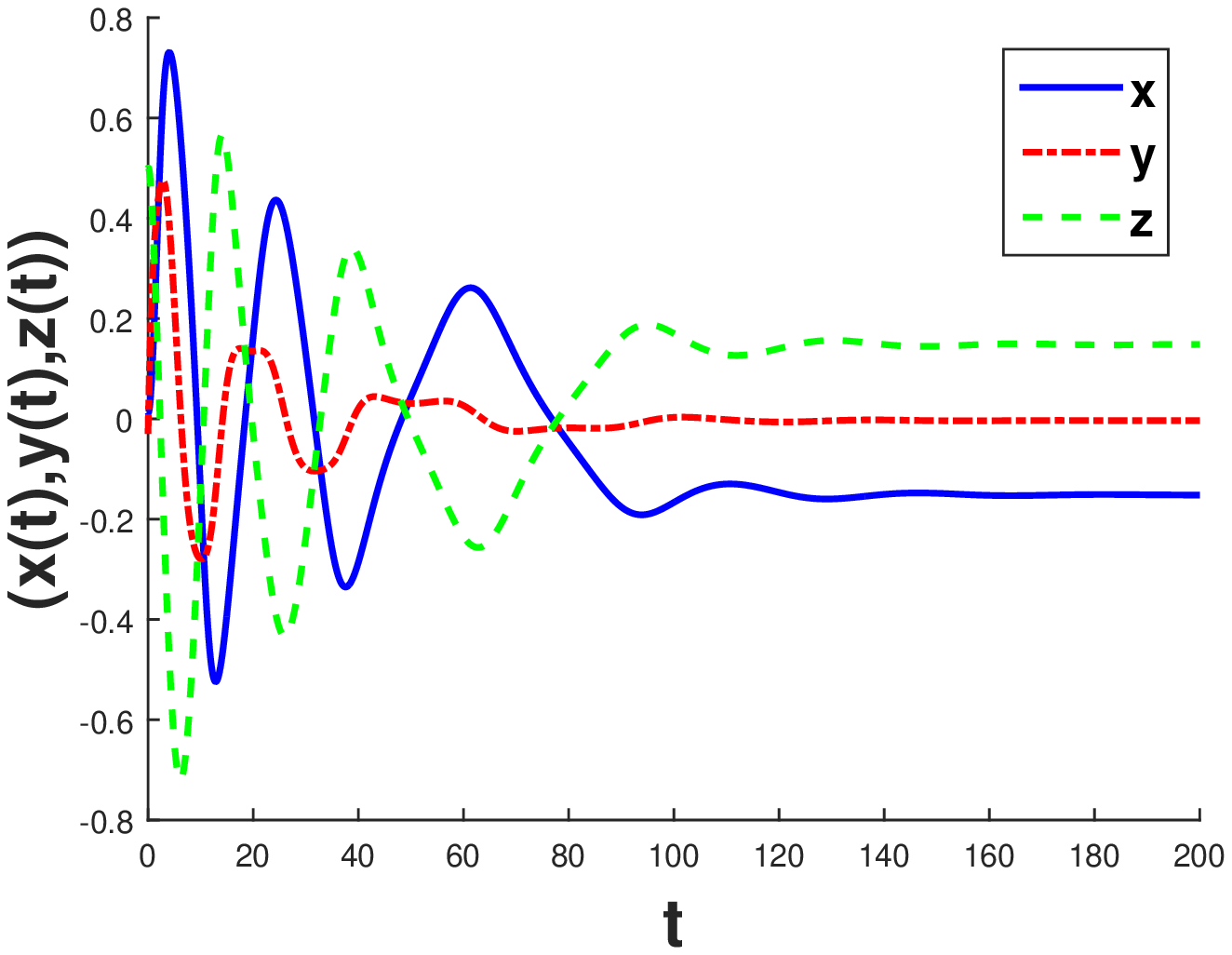}}\;\,
\end{center}
\vspace{-0.250 in}
\caption{Controlled trajectories \ref{Fig7g}-\ref{Fig7j} correspond with regions (g) and (h) from Figure \ref{fig71} for \eqref{chua}, \(\alpha:=0.8,\) \(a:=1\), \(\nu_3=0.3,\) \(\nu_0=0\).}\label{breaking2}
\vspace{-0.100 in}
\end{figure}

\subsubsection*{Region (g): Limit cycle \(\mathscr{C}^1_+\) appears when controller coefficients cross \(T_{HmC+}\).}

By Theorem \ref{Hom0}, there is a stable limit cycle
\(\mathscr{C}^1_+\) for controller coefficients from region (g), \ie there is a homoclinic variety \(T_{HmC+}\) at the border between regions (f) and (g). In fact, there is a homoclinic cycle \({\Gamma}_+\) (encircling \(e_+\)) for controller coefficients from \(T_{HmC+}\). There are a stable limit cycle \(\mathscr{C}^1_+\) encircling unstable equilibrium \(e_+,\) and two more equilibria outside of \(\mathscr{C}^1_+\). The latter equilibria are a saddle and a sink \(e_-\). All equilibria, homoclinic cycle \({\Gamma}_+\) and limit cycle \(\mathscr{C}^1_+\) live on center manifold \(\mathscr{M}.\) Trajectories in Figures \ref{Fig7g} converge to the stable limit cycle \(\mathscr{C}^1_+\) and \(e_{-}\) for controller coefficients \(\nu_1=0.018\) and \(\nu_2=-0.016\) from region (g), initial values \((0.02, 0.005, -0.02)\) and \((-0.015, -0.001, 0.3),\) respectively.

\subsubsection*{Region (h): Limit cycle \(\mathscr{C}^1_+\) shrinks in size to coalesce with \(e_+\) and disappear.}

Theorem \ref{thm3} concludes that there is a supercritical
Hopf bifurcation at \(T_{H+},\) where \(\mathscr{C}^1_+\) is born from \(e_+.\) Controller coefficients crossing \(T_{H+}\) and entering region (h) lead to change of stability for \(e_+\) and disappearance of \(\mathscr{C}^1_+.\) Hence for controller coefficients from region (g), there exist two stable equilibria \(e_\pm,\)  a saddle origin and no limit cycle. Trajectories in Figures \ref{Fig7j} for initial values \((0.013, -0.03, 0.5)\) and \((-0.013, -0.001, 0.1)\) converge to \(e_{+}\) and \(e_{-},\) respectively.


\begin{rem}[Stabilization and regularization for controller coefficient choices from regions (g) and (h)]
Controller coefficients associated with region (h) lead to two asymptotically stable equilibria \(e_+\) and \(e_-,\) whose their combined basins of attraction is the whole state space except the stable and unstable manifolds of the saddle origin. Therefore, controller coefficients chosen from region (h) give a regularization approach for the system, where the controlled trajectories converge to equilibria close to the origin; also see \cite{KangIEEE} for a similar phenomenon. Since Hopf bifurcation at the border between regions (h) and (g) is supercritical, controller coefficient choices from region (g) have a stabilization property. In fact, every trajectory converges to either \(\mathscr{C}^1_+\) or \(e_-,\) except that a trajectory would exactly fall on the stable and unstable manifolds of the saddle origin. The latter is practically impossible due to imperfections, errors, etc., and the fact that the stable and unstable manifolds are structurally unstable and have a zero-Lebesgue measure.
\end{rem}


\baselineskip=13pt

\section{Conflict of interest} There is no conflict of interest to report.

\section{Data availability statement} There is no data associated with this paper to report.

\end{document}